\documentclass{amsart}
\usepackage{amsmath,amsfonts,amssymb,latexsym}
\usepackage[margin=1in]{geometry}
\usepackage{color}
\usepackage[pagebackref, bookmarks=true, bookmarksopen=true, bookmarksdepth=3,bookmarksopenlevel=2, colorlinks=true, linkcolor=blue, citecolor=blue, filecolor=blue, menucolor=blue, urlcolor=blue]{hyperref}
\input xy
\xyoption{all}
\usepackage{tikz}
\usepackage{mathalfa}

\usetikzlibrary{calc}

\newtheorem{theorem}{Theorem}[section]
\newtheorem{corollary}[theorem]{Corollary}
\newtheorem{definition}[theorem]{Definition}
\newtheorem{lemma}[theorem]{Lemma}

\newtheorem{proposition}[theorem]{Proposition}
\newtheorem{remark}[theorem]{Remark}

\numberwithin{equation}{section}

\newcommand{\bfa}{{\boldsymbol{a}}}

\newcommand{\bfc}{\mathbf{c}}

\newcommand{\bfg}{\mathbf{g}}
\newcommand{\bfp}{{\boldsymbol{p}}}
\newcommand{\bfq}{{\boldsymbol{q}}}
\newcommand{\bfr}{{\boldsymbol{r}}}
\newcommand{\bfs}{{\boldsymbol{s}}}
\newcommand{\bft}{{\boldsymbol{t}}}
\newcommand{\bfu}{{\boldsymbol{u}}}
\newcommand{\bfv}{{\boldsymbol{v}}}
\newcommand{\bfx}{{\boldsymbol{x}}}
\newcommand{\bfy}{{\boldsymbol{y}}}

\newcommand{\bfxi}{\boldsymbol{\xi}}
\newcommand{\bfzero}{\mathbf{0}}

\newcommand{\sfx}{{\mathsf{x}}}

\newcommand{\sfy}{{\mathsf{y}}}
\newcommand{\sfp}{{\mathsf{p}}}

\newcommand{\cA}{\mathcal{A}}
\newcommand{\cB}{\mathcal{B}}
\newcommand{\cD}{\mathcal{D}}

\newcommand{\cG}{\mathcal{G}}
\newcommand{\cH}{\mathcal{H}}
\newcommand{\cL}{\mathcal{L}}
\newcommand{\cO}{\mathcal{O}}

\newcommand{\cT}{\mathcal{T}}
\newcommand{\cU}{\mathcal{U}}
\newcommand{\cX}{\mathcal{X}}

\newcommand{\CC}{\mathbb{C}}

\newcommand{\RR}{\mathbb{R}}
\newcommand{\TT}{\mathbb{T}}
\newcommand{\ZZ}{\mathbb{Z}}

\newcommand{\diag}{\operatorname{diag}}

\newcommand{\Li}{\operatorname{Li}}
\newcommand{\Lie}{\operatorname{Lie}}
\renewcommand{\max}{\operatorname{max}}

\newcommand{\rra}{\rightrightarrows}

\title{Symplectic Groupoids for Cluster Manifolds}
\author{Songhao Li}
\address[Songhao Li]{University of Notre Dame, Department of Mathematics, Notre Dame, IN 46556, USA}
\email{sli19@nd.edu}

\author{Dylan Rupel}
\address[Dylan Rupel]{University of Notre Dame, Department of Mathematics, Notre Dame, IN 46556, USA}
\email{drupel@nd.edu}

\begin{document}
\begin{abstract}
  We construct symplectic groupoids integrating log-canonical Poisson structures on cluster varieties of type $\cA$ and $\cX$ over both the real and complex numbers.
  Extensions of these groupoids to the completions of the cluster varieties where cluster variables are allowed to vanish are also considered.  

  In the real case, we construct source-simply-connected groupoids for the cluster charts via the Poisson spray technique of Crainic and M\u{a}rcu\c{t}.
  These groupoid charts and their analogues for the symplectic double and blow-up groupoids are glued by lifting the cluster mutations to groupoid comorphisms whose formulas are motivated by the Hamiltonian perspective of cluster mutations introduced by Fock and Goncharov.
\end{abstract}
\maketitle

\section{Introduction}
This is the first of a series of papers whose aim is to implement the Weinstein program of geometric quantization for Poisson manifolds \cite{Wei91} in the case of Poisson structures compatible with a cluster structure.
In this first paper, we construct the symplectic groupoids for both the cluster $\cA$-varieties and the cluster $\cX$-varieties. 

The cluster $\cA$-varieties are the geometric realization of cluster algebras defined by Fomin and Zelevinsky \cite{FZ02} as the culmination of their study of total positivity and canonical bases for algebraic groups \cite{BFZ96,FZ99}.
Many varieties arising in Lie theory, e.g.\ Grassmannians and double Bruhat cells, are examples of $\cA$-varieties \cite{BFZ05,Sco06,GSV03,Wil13b}.
The cluster $\cA$-varieties are often endowed with a class of compatible Poisson structures in the sense of Gekhtman, Shapiro, and Vainshtein \cite{GSV10}.
In particular, our results immediately give rise to a trio of symplectic/Poisson groupoids integrating a compatible Poisson structure on any cluster $\cA$-variety, including those mentioned above. 

The cluster $\cX$-varieties were introduced by Fock and Goncharov \cite{FG09a} in their study of higher Teichm\"uller space.
An alternative view of double Bruhat cells and Grassmannians reveal cluster $\cX$-variety structures as well.
A cluster $\cX$-variety is always endowed with a canonical Poisson structure.
In particular, as with the cluster $\cA$-varieties, our results produce symplectic/Poisson groupoids integrating any cluster $\cX$-variety.
If an $\cA$-variety carries a compatible Poisson structure, then the natural map from this $\cA$-variety to the corresponding $\cX$-variety is Poisson.
We provide a lift of this map to comorphisms between analogous groupoids over the $\cA$- and $\cX$-varieties.

In the case of cluster algebras, the quantum $\cA$-varieties \cite{BZ05} and the quantum $\cX$-varieties \cite{FG09a} are both concrete examples of deformation quantization of Poisson manifolds.
In fact, the $q$-quantization of log-canonical variables can be realized simply as the exponential version of the standard Weyl quantization for canonical variables \cite{FG09c}.
Geometric quantization for symplectic manifolds takes one step further by constructing a Hilbert space on which the quantized algebra acts, but this requires that the cohomology class of the symplectic 2-form be integral (see e.g.\ \cite{BW97}).
The notion of symplectic groupoids was introduced by Weinstein \cite{Wei87}, Karas\"{e}v \cite{Kar89} and Zakrzewski \cite{Zak90a, Zak90b}.
Weinstein's motivation was to find a geometric quantization schema for Poisson manifolds \cite{Wei91}, which has only been successfully implemented for a handful of examples \cite{Tan06, Haw08, BCST12}.
The concrete nature of cluster coordinates provides an ideal testing ground for groupoid quantization.
Indeed, this was implicitly implemented by Fock and Goncharov \cite{FG09c} in the case of $\cX$-varieties.

In this paper, we take the first step towards the groupoid quantization for both cluster $\cA$-varieties and cluster $\cX$-varieties by describing their symplectic groupoids.
The construction of each is detailed below as we discuss the organization of the paper.

We being in Section~\ref{sec:Poisson generalities} with a general discussion of Poisson manifolds and their associated algebroids and groupoids.
In Section~\ref{sec:local}, we construct three groupoids $\cG$, $\cB$ and $\cD$ for a log-canonical Poisson structure on a vector space $L$, these each possess important properties which justify their consideration.
The groupoid $\cG \rra L$ is the source-simply-connected symplectic groupoid.
The definition of $\cG \rra L$ utilizes the construction of local symplectic groupoids by Poisson sprays \cite{CM11, CMS17}.
For a log-canonical Poisson structure, there is a natural choice of a Poisson spray which yields $\cG$ under the spray construction.
An important note is that, although this provides the source-simply-connected groupoid, the source and the target maps of $\cG \rra L$ will be transcendental and in particular they do not conform (strictly speaking) to the algebraic nature of cluster theory.

One fix for this is an analogue of the symplectic double introduced by Fock and Goncharov \cite{FG09c}.
This groupoid which we denote $\cD \rra L$ is a source-connected Poisson groupoid.
Over an orthant $L^\times$ of $L$ the groupoid $\cD \rra L$ is actually symplectic and we denote its restiction by $\cD^\times \rra L^\times$, our construction gives a covering of the symplectic groupoid considered by Fock and Goncharov in the case of $\cX$-varieties.

The linear degeneracy of the Poisson structure on $\cD$ over the coordinate hyperplanes is corrected via the blow-up construction introduced in works of the first author \cite{Li13, GL14}.
This gives rise to the groupoid $\cB \rra L$ which is source-connected and symplectic.
Comparing to $\cG \rra L$, the main advantage of $\cB \rra L$ is that the source and the target maps of $\cB \rra L$ are rational.
In this way, it seems likely that the groupoid $\cB$ will be most useful in the theory of cluster algebras.
Observe that both $\cB \rra L$ and $\cD \rra L$ receive maps from the source-simply-connected symplectic groupoid $\cG$.

The groupoids of the cluster $\cA$-variety and cluster $\cX$-variety are not constructed globally but rather are glued from groupoid charts over each cluster chart.
In Section~\ref{sec:mutations}, we give the explicit formulas for the cluster groupoid mutations $\mu: \cG \to \cG'$, $\mu: \cB \to \cB'$ and $\mu: \cD \to \cD'$.
We begin in Section~\ref{sec:cluster mutations} in the setting of generalized cluster mutations $\mu: L \to L'$ which we decompose into two maps $\varphi^1$ and $\tau$ according to the Hamiltonian perspective of mutations \cite{FG09a, GNR17}.
The first map $\varphi^1$ is the time-$1$ flow of a Hamiltonian vector field defined using the Euler dilogarithm function; the second map $\tau$ is a transformation of the log-canonical coordinates.
We then lift $\varphi^1$ and $\tau$ to groupoid maps which compose to give the cluster groupoid mutations.
In principle, a Poisson map induces a Lie algebroid comorphism \cite{HM90}, which is then lifted to the groupoid comorphisms \cite{Cat04, CDW13}.
Indeed, the Poisson ensemble map $\rho: L_\cA \to L_\cX$ can be lifted only to a groupoid comorphism when the exchange matrix is not square.
However, if a Poisson morphism is a diffeomorphism, then the dual bundle map of the induced Lie algebroid comorphism is a Lie algebroid morphism, which then lifts to an honest groupoid morphism.
Since both the Hamiltonian flow $\varphi^1$ and the transformation $\tau$ are diffeomorphisms on an open dense set, we may lift the cluster mutation $\mu = \tau \circ \varphi^1$ to cluster groupoid mutations.
\bigskip

Throughout the paper, we use the following notations.
\begin{itemize}
	\item We write $\RR_\times = (0, \infty)$, $\bar\RR_\times = [0, \infty)$ and $\CC_\times = \CC \setminus \{0\}$.
	\item For the cluster charts, we use the following notations:
	$$
		L = \RR^m \text{ or } \CC^m, \quad L^\times = \RR_\times^m \text{ or } \CC_\times^m, \quad \bar L^\times = \bar\RR_\times^m.
	$$
	\item We denote vectors by boldface, e.g. $\bfx = (x_1, \ldots, x_m)$.
        \item The \emph{Hadamard product} of two vectors $\bfx$ and $\bfy$ is given by $\bfx \circ \bfy = (x_1y_1, \ldots, x_my_m)$.
	\item For a vector $\bfx$ and a real number $t$, we set $\bfx^t := (x_1^t, \ldots, x_m^t)$.
	\item For a smooth (resp.\ complex) manifold $M$, we denote the space of smooth (resp.\ holomorphic) functions by $\cO_M$.
	\item For a vector bundle $E$ over $M$, we denote the space of sections by $\Gamma(M, E)$.
          In particular, we denote the space of vector fields by $\cT_M = \Gamma(M, TM)$ and the space of multi-vector fields by $\cT^k_M = \Gamma(M, \wedge^k TM)$.
          To follow the conventional notation, we denote the space of differential $k$-forms by $\Omega^k(M) = \Gamma(M, \wedge^k T^*M)$.
\end{itemize}

\subsection*{Acknowledgements}
The authors would like to thank Sam Evens, Rui Loja Fernandes, Michael Gekhtman, Marco Gualtieri, and Alan Weinstein for useful discussions related to this project.

\section{Poisson structures and symplectic groupoids}
\label{sec:Poisson generalities}
We begin by recalling the equivalent notions of Poisson brackets and Poisson bi-vectors.
\begin{definition} 
  \label{def: PoissonMfld}
  Let $M$ be either a smooth manifold or a complex manifold.
  A \emph{Poisson structure} on $M$ is one of the following two equivalent structures:
  \begin{enumerate}
    \item a \emph{Poisson bracket}
      $$\{\cdot, \cdot\}: \cO_M \times \cO_M \to \cO_M$$
      which is a Lie bracket satisfying the Leibniz rule
      $$\{f, gh\} = g\{f,h\} + h\{f,g\};$$
    \item a \emph{Poisson bi-vector} $\pi \in \cT^2_M$ such that $[\pi, \pi] = 0$, where $[\cdot, \cdot]$ is the Schouten-Nijenhuis bracket.
  \end{enumerate}
  We say $f \in \cO_M$ is a Casimir if $\{f, g\} = 0$ for every $g\in \cO_M$.
\end{definition}

The two notions are related by the formula: $\{f, g\} = \pi (df \otimes dg)$ for $f, g\in \cO_M$.
The pair $(M, \pi)$, or equivalently $(M, \{\cdot,\cdot\})$, is called a \emph{Poisson manifold}.
The \emph{Hamiltonian vector field} of $f \in \cO_M$ is defined as the contraction $X_f = \iota_{df}\pi$ or equivalently as the vector field naturally associated to the derivation $\{f,\cdot\}$.
A Poisson map from $(M_1, \pi_1)$ to $(M_2, \pi_2)$ is a map $\varphi: M_1\to M_2$ such that $\varphi_*\pi_1 = \pi_2$ or equivalently $\{\varphi^*f, \varphi^*g\} = \varphi^*(\{f, g\})$ for $f, g \in \cO_{M_2}$.
For $f \in \cO_{M_2}$, the pullback of the Hamiltonian vector field $X_f$ is $X_{\varphi^*f}$.
A Poisson map $\varphi: (M_1, \pi_1) \to (M_2, \pi_2)$ is \emph{complete} if the pullback of a complete Hamiltonian vector field is again a complete vector field.

A bi-vector $\pi \in \cT^2_M$ is called \emph{non-degenerate} if the bundle map
\[\pi^\flat: \Omega^1(M) \to \cT_M, \qquad \theta \mapsto \iota_\theta \pi\]
is invertible.
The inverse of this bundle map then defines a non-degenerate 2-form $\omega \in \Omega^2(M)$.
That is, the bundle map
\[\omega^\sharp: \cT_M \to \Omega^1(M), \qquad v \mapsto \iota_v \omega\]
is the inverse of $\pi^\flat$.
The condition $[\pi,\pi]=0$ is equivalent to $d\omega = 0$, so a non-degenerate Poisson bi-vector is the same as a symplectic 2-form.
Hence for a non-degenerate Poisson bi-vector $\pi$, we denote the corresponding symplectic 2-form by $\pi^{-1}$ and for a symplectic 2-form $\omega$, we denote the corresponding Poisson bi-vector as $\omega^{-1}$.

\begin{definition}
  For a Poisson manifold $(M, \pi)$, a symplectic realization is a symplectic manifold $(S, \omega)$ together with a surjective Poisson map $\rho: (S, \omega) \to (M, \pi)$.
\end{definition}

Of particular importance among all the symplectic realizations is the symplectic groupoid, but first we recall the notion of Lie groupoids and Lie algebroids.
\begin{definition} \label{def:groupoid}
  A \emph{groupoid} $\cG\rightrightarrows M$ consists of two sets $\cG$ and $M$ with the following maps:
  \begin{enumerate}
    \item a surjective \emph{source map} $\alpha: \cG \to M$ and a surjective \emph{target map} $\beta: \cG \to M$;
    \item an injective identity map $\mathtt{1}: M \to \cG, \enskip x \mapsto \mathtt{1}_x$;
    \item an associative multiplication map $m: \cG {_\alpha \times_\beta} \cG \to \cG, \enskip (g, h) \mapsto gh$;
    \item and an involutive inversion map $\iota: \cG \to \cG, \enskip g \mapsto g^{-1}$;
  \end{enumerate}
  which satisfy the following properties:
  \begin{enumerate}
    \item $\alpha(\mathtt{1}_x) = \beta(\mathtt{1}_x) = x$;
    \item $\alpha(gh) = \alpha(h), \enskip \beta(gh) = \beta(g)$;
    \item $\alpha(g^{-1}) = \beta(g), \enskip \beta(g^{-1}) = \alpha(g)$;
    \item $(\mathtt{1}_x)^{-1} = \mathtt{1}_x$;
    \item $\mathtt{1}_{\beta(g)}g=g=g\mathtt{1}_{\alpha(g)}, \enskip g^{-1}g=\mathtt{1}_{\alpha(g)}, \enskip gg^{-1}=\mathtt{1}_{\beta(g)}$.
  \end{enumerate}
  A \emph{Lie groupoid} $\cG \rightrightarrows M$ has the following additional properties:
  \begin{enumerate}
    \item $\cG$ and $M$ are smooth (or complex) manifolds;
    \item the source and target $\alpha,\beta: \cG \to M$ are surjective submersions;
    \item the multiplication map $m: \cG {_\alpha \times_\beta} \cG \to \cG$ is smooth (or holomorphic);
    \item the inversion map $i: \cG \to \cG$ is smooth (or holomorphic).
  \end{enumerate}
  A Lie groupoid $\cG \rightrightarrows M$ is \emph{source-connected} if the source fiber $\alpha^{-1}(x)$ is connected for every $x \in M$; it is \emph{source-simply-connected} if the source fiber $\alpha^{-1}(x)$ is connected and simply-connected for every $x \in M$.
\end{definition}

A groupoid is naturally considered as a category with objects the elements of $M$ and morphisms the elements of $\cG$.
Then a morphism of groupoids from $\cG_1\rra M_1$ to $\cG_2\rra M_2$ is simply a functor between these categories.

Next we recall the notion of a Lie algebroid, which is the infinitesimal object of a Lie groupoid.
\begin{definition} \label{def:algeroid}
  For a smooth (or holomorphic) manifold $M$, a Lie algebroid over $M$ is a triple $(A, [\cdot, \cdot], \rho)$, where
  \begin{enumerate}
    \item $A$ is a vector bundle over a $M$;
    \item $[\cdot, \cdot]$ is a Lie bracket on the space of sections $\Gamma(M, A)$;
    \item $\rho: A \to TM$ is a bundle morphism preserving the Lie bracket;
  \end{enumerate}
  with Lie bracket satisfying the Leibniz rule: for sections $X, Y \in \Gamma(M, A)$ and $f \in \cO_M$,
  \[[X, fY] = f\cdot[X, Y] + (\rho X)(f) \cdot Y.\]
\end{definition}

There is a Lie functor from the Lie groupoids to the Lie algebroids.
For a Lie groupoid $\cG \rra M$, we define its Lie algebroid $A = \Lie \cG$ as follows.
As a vector bundle, we have
\begin{equation} \label{eq:Liefunctor}
	A = \ker \left(\alpha_*: T\cG|_{\mathtt{1}_M} \to TM \right).
\end{equation}
The Lie bracket is the bracket of left-invariant vector fields and the anchor map is the restriction of the target map $\beta_*: T\cG \to TM$ to $A$.
In this case, we say the Lie groupoid $\cG$ \emph{integrates} the Lie algebroid $A$.

Given a Lie groupoid $\cG \rra M$, its \emph{$k$-nerve},
\[\cG^{(k)} := \{(g_1, g_2, \ldots, g_k) \in \cG^k ~|~ \beta(g_{i+1}) = \alpha(g_i)\},\]
is the set of $k$-composable elements.
In particular, we have
\[\cG^{(2)} = \cG {_\alpha \times_\beta} \cG,\qquad \cG^{(1)} = \cG,\qquad \cG^{(0)} = M.\]
The \emph{nerve} of a Lie groupoid is naturally a simplicial manifold that carries a coboundary operator $\partial: \Omega^\bullet(\cG^{(k-1)}) \to \Omega^\bullet(\cG^{(k)})$.
The first two operators are given as below:
\begin{align} 
  \label{eq:gpdcob}
  & \partial: \Omega^\bullet(M) \to \Omega^\bullet(\cG), \qquad \mu \mapsto \alpha^*(\mu) - \beta^*(\mu); \\
  \nonumber
  & \partial: \Omega^\bullet(\cG) \to \Omega^\bullet(\cG^{(2)}), \qquad \mu \mapsto \mathrm{pr}_1^*(\mu) - m^*(\mu) + \mathrm{pr}_2^*(\mu);
\end{align}
where $\mathrm{pr}_1: \cG^{(2)} \to \cG$ and $\mathrm{pr}_2: \cG^{(2)} \to \cG$ are the first and second projections.
A differential form $\mu \in \Omega^\bullet(\cG)$ is called \emph{multiplicative} if $\partial \mu = 0$.
Our main interest will be with Lie groupoids equipped with a multiplicative symplectic structure.
\begin{definition}
  A \emph{symplectic groupoid} is a Lie groupoid $\cG \rra M$ with a multiplicative symplectic structure $\omega \in \Omega^2(\cG)$.
  That is, $\mathrm{pr}_1^*(\omega)+\mathrm{pr}_2^*(\omega) = m^*(\omega)$ or equivalently the graph of the multiplication map $\Gamma_m := \{(g, h, gh) \in \cG \times \cG \times \cG\}$ is Lagrangian with repsect to the symplectic structure $\omega \oplus \omega \oplus -\omega$ on $\cG\times\cG\times\cG$.
\end{definition}

The source fibers of a symplectic groupoid are symplectic orthogonal to the target fibers.
Some important examples of symplectic groupoids include: the Kostant-Kirillov-Souriau Poisson structures \cite{CDW87}, the Drinfeld doubles of Poisson Lie groups \cite{LW89}, the double Bruhat cells \cite{LuM16}, the blow-up groupoids of log-symplectic manifolds \cite{GL14}, and the symplectic doubles of (the positive part of) the cluster $\cX$-varieties \cite{FG09c}.
We note that symplectic groupoids are a special case of Poisson groupoids.
\begin{definition}
  A \emph{Poisson groupoid} is a Lie groupoid $\cG \rra M$ with a multiplicative Poisson structure $\sigma \in \cT^2_\cG$.
  That is, the graph of the multiplication map $\Gamma_m$ is coisotropic with respect to the Poisson structure $\sigma \oplus \sigma \oplus -\sigma$ on $\cG\times\cG\times\cG$.
\end{definition}

A Poisson groupoid map from $(\cG_1, \sigma_1) \rra M_1$ to $(\cG_2, \sigma_2) \rra M_2$ is a Lie groupoid map $\varphi: \cG_1 \to \cG_2$ which is Poisson.
For a Poisson groupoid $(\cG, \sigma) \rra M$, there is a natural Poisson structure $\pi$ on $M$ such that $\alpha: (\cG, \sigma) \to (M, \pi)$ and $\beta: (\cG, \sigma) \to (M, -\pi)$ are Poisson maps.
For a Poisson manifold $(M, \pi)$, one could ask if there is a symplectic groupoid $(\cG, \omega) \rra (M, \pi)$ such that $\alpha: (\cG, \omega) \to (M, \pi)$ and $\beta: (\cG, \omega) \to (M, -\pi)$ are Poisson maps.
If the answer is yes, we say that the Poisson manifold $(M, \pi)$ is \emph{integrable}.
The Lie algeboid of a symplectic groupoid $\cG \rra M$ is the cotangent bundle $T^*M$ \cite{Wei87} with the anchor map
\[\pi^\flat: T^*M \to TM, \qquad \theta \mapsto \iota_\theta \pi,\]
and the Koszul bracket: for $\theta, \psi \in \Omega^1(M)$,
\[[\theta, \psi] = \cL_{\pi(\theta)} \psi - \cL_{\pi(\psi)} \theta - d\pi(\theta \otimes \psi).\]
For a Poisson manifold $(M, \pi)$, we denote its \emph{cotangent Lie algebroid} by $T^*_\pi M$.

Just like the infinitesimal object of a Poisson Lie group is a Lie bialgebra, the infinitesimal object of a Poisson groupoid is a Lie bialgebroid \cite{MX94}.
As a Poisson groupoid, the symplectic groupoid $(\cG, \omega) \rra (M, \pi)$ integrates the Lie bialgebroid $(T^*_\pi M, TM)$.
The integrability of Poisson manifolds, and more generally the integrability of Lie algebroids, is characterized in \cite{CF03, CF04}.

In general, a Poisson map does not induce a Lie algebroid morphism and hence does not integrate to a bona fide symplectic groupoid map.
Instead a Poisson map naturally induces a Lie algebroid comorphism.
Following \cite{CDW13}, we outline the procedure to lift Lie algebroid morphisms and comorphisms to Lie groupoid morphisms and comorphisms.
\begin{definition} 
  \cite{HM90, Mac05, CDW13}
 	Let $A_1$ be a Lie algebroid over $M_1$ and $A_2$ be a Lie algebroid over $M_2$ with anchor maps
 	\[ \rho_1: A_1 \to TM_1, \qquad \rho_2: A_2 \to TM_2. \]
 	\begin{enumerate}
 		\item A Lie algebroid morphism is a bundle map $(\varphi, \Phi): (M_1, A_1) \to (M_2, A_2)$ such that the pullback of sections $\Phi^*: \Gamma(\wedge^\bullet A_2) \to \Gamma(\wedge^\bullet A_1)$ is a chain map of complexes.
 		\item A Lie algebroid comorphism is a bundle comorphism $(\varphi, \Psi)$ with
 			\[ \varphi: M_1 \to M_2,~ \Psi: \varphi^! A_2 = M_1 {\times_{M_2}} A_2 \to A_1 \]
 			such that $\varphi_* \circ \rho_1 \circ \Psi = \rho_2$. 
 	\end{enumerate}
\end{definition}
In principle, the graph of a comorphism $(\varphi, \Psi)$ from $A_1$ to $A_2$ is a subset of $\varphi^! A_2 \times A_1 = \left(M_1 {\times_{M_2}} A_2 \right) \times A_1$, but we will interpret the graph of $\Psi$ as a subset of $A_2 \times A_1$. Hence for a diffeomorphism $\varphi: M_1 \to M_2$, we have that $(\varphi, \Phi)$ is a Lie algebroid morphism if and only if $(\varphi^{-1}, \Phi^\vee)$ is a Lie algebroid comorphism where $\Phi^\vee$ is the dual bundle map of $\Phi$.
\begin{definition} 
  \cite{Mac05, CDW13} 
  \label{def:gpdcomor}
  Let $\cG_1 \rra M_1$ and $\cG_2 \rra M_2$ be Lie groupoids.
  \begin{enumerate}
    \item A Lie groupoid morphism from $\cG_1$ to $\cG_2$ is a map $(\varphi, \Phi): (M_1, \cG_1) \to (M_2, \cG_2)$ that is compatible with groupoid structures. 
    \item A Lie groupoid comorphism from $\cG_1$ to $\cG_2$ is a base map $\varphi: M_1 \to M_2$ together with a map
      \[ \Psi: M_1 {_\varphi \times_\alpha} \cG_2 \to \cG_1 \]
      that is compatible with groupoid structures (for details, see p.5 of \cite{CDW13}). 
  \end{enumerate}
\end{definition}
Note the difference in convention: the source and target maps in Definition~\ref{def:gpdcomor} are switched comparing to the convention in \cite{CDW13}.
As with the Lie algebroids, we will interpret the graph of a Lie groupoid comorphism $\Psi$ from $\cG_1$ to $\cG_2$ as a subset of $\cG_2 \times \cG_1$.

In general, a Lie algebroid morphism from $A_1$ to $A_2$ integrates to a Lie groupoid morphism from $\cG_1$ to $\cG_2$ if $\cG_1$ is source-simply-connected; a Lie algebroid comorphism from $A_1$ to $A_2$ integrates to a Lie groupoid comorphism from $\cG_1$ to $\cG_2$ if $\cG_2$ is source-simply-connected and the algebroid comorphism is complete.
Concretely, the graph of a morphism (or a comorphism) from $A_1$ to $A_2$ is a Lie subalgebroid of $A_1 \times A_2$, and the exponential map $A_1 \times A_2 \to \cG_1 \times \cG_2$ integrates the graph to a Lie subgroupoid of $\cG_1 \times \cG_2$ which happens to be the graph of a morphism (or a comorphism) from $\cG_1$ to $\cG_2$ under the assumptions above.

A Poisson map $\varphi: (M_1, \pi_1) \to (M_2, \pi_2)$ naturally induces the Lie algebroid comorphism $\varphi^*: \varphi^!  T^*_{\pi_2} M_2 \to T^*_{\pi_1} M_1$, which means that the following diagram commutes \cite{HM90}.
\begin{equation*} 
	\xymatrix{
		T^*_{\pi_1} M_1 \ar[d]^{\pi_1^\flat} & \varphi^! T^*_{\pi_2}  M_2 \ar[l]_{\varphi^*} \ar[r]& T^*_{\pi_2} M_2  \ar[d]^{\pi_2^\flat} \\
		T M_1 \ar[rr]^{\varphi_*} & & TM_2
	}
\end{equation*}
This comorphism $\varphi^*$ is complete if for a complete Hamiltonian vector field $X_f$ with $f \in \cO_{M_2}$, the pullback Hamiltonian vector field $X_{\varphi^*f}$ is also complete.
That is, a Poisson map $\varphi: (M_1, \pi_1) \to (M_2, \pi_2)$ lifts to a unique symplectic groupoid comorphism from $(\cG_1, \omega_1)$ to $(\cG_2, \omega_2)$ if $\cG_2$ is source-simply-connected and $\varphi$ is complete.
In fact, the graph of $\varphi$, which is a coisotropic submanifold of $(M_1 \times M_2, \pi_1 \oplus -\pi_2)$, integrates to a Lagrangian subgroupoid of $(\cG_1 \times \cG_2, \omega_1 \oplus -\omega_2)$ \cite{Cat04}.

If the Poisson map $\varphi: (M, \pi) \to (M, \pi)$ is a diffeomorphism, then $\varphi$ induces both a Lie algebroid morphism and a Lie algebroid comorphism, which can be lifted either to a symplectic groupoid morphism or to a symplectic groupoid comorphism. For the periodicity of groupoid mutations Proposition~\ref{prop:gpd periodicity}, we single out the special case when the Poisson map is the identity.

\begin{proposition} \label{prop:gpd identity}
If $\cG \rra M$ is the source-simply-connected symplectic groupoid of $(M, \pi)$, then the identity map $\mathtt{1}_M: M\to M$ as a Poisson diffeomorphism lifts to the identity groupoid map $\mathtt{1}_\cG: \cG \to \cG$.
\end{proposition}

Following \cite{Fer07}, we outline a method to lift the Hamiltonian Poisson maps to Hamiltonian symplectic groupoid morphisms.
In general, Poisson vector fields are lifted to multiplicative Hamiltonian vector fields.
For the Hamiltonian vector field $X_f \in \cT_M$ of $f \in \cO_M$, if $(\cG, \omega) \rra (M, \pi)$ is a symplectic groupoid, then $X_f$ is lifted to the Hamiltonian vector field $X_F \in \cT_\cG$, where $F := \alpha^*f - \beta^*f$.
Indeed, the function $F \in \cO_\cG$ is multiplicative since $F = \partial f$, where $\partial$ is the groupoid coboundary operator \eqref{eq:gpdcob}, and the symplectic form $\omega$ is multiplicative by definition, so the Hamiltonian vector field $X_F = \iota_{dF}\omega^{-1}$ is multiplicative.
It follows that $X_F$ preserves the symplectic groupoid structures of $(\cG,\omega)$; and the time-$t$ flow $\varphi^t_f: (M, \pi) \to (M, \pi)$ of $X_f$ is lifted to the time-$t$ flow $\varphi^t_F: (\cG, \omega) \to (\cG, \omega)$ of $X_F$.

\section{Symplectic groupoids of log-canonical Poisson structures}
\label{sec:local}
We focus in this paper on the symplectic groupoids of log-canonical Poisson structures.
\begin{definition} 
  \label{def: logPoisson}
  Let $L$ be either $\RR^m$ or $\CC^m$ and write $\bfx = (x_1, \ldots, x_m)$ for a system of coordinates on $L$.
  A Poisson structure on $L$ is \emph{log-canonical} if
  $$
    \{x_i, x_j\} = \Omega_{ij} x_ix_j, \quad 1 \leq i,j \leq m \qquad \text{or equivalently} \qquad
    \pi = \sum_{j < i} \Omega_{ij} x_ix_j\frac{\partial}{\partial x_i} \wedge \frac{\partial}{\partial x_j},
  $$
  for some skew-symmetric $m\times m$ matrix $\Omega = (\Omega_{ij})$.
\end{definition}

Using the results in \cite{CM11, CMS17}, we construct the source-simply-connected symplectic groupoid of a log-canonical Poisson structure by chosing an appropriate Poisson spray. 
\begin{definition} 
  \cite{CM11}
  For a Poisson manifold $(M, \pi)$, a \emph{Poisson spray} is a vector field $X \in \cT_{T^*M}$ such that
  \begin{enumerate}
    \item for $(x,p) \in T^*M$ we have
      \[\tau_* X|_{(x,p)} = \pi^\flat(p),\]
      where $\tau: T^*M \to M$ is the bundle projection;
    \item $X$ is homogeneous of degree 1, i.e.
      \[(m_\lambda)_*(X) = \lambda X,\]
      where $m_\lambda: T^*M \to T^*M, \enskip (x,p) \mapsto (x,\lambda p)$ is the fiberwise scaling map.
  \end{enumerate}
\end{definition}

\begin{theorem} 
  \cite{CM11, CMS17} 
  \label{thm:poissp}
  For a smooth Poisson manifold $(M, \pi)$ with a Poisson spray $X \in \cT_{T^*M}$, there exists a neighborhood $U$ of the zero section of $T^*M$ which is a local symplectic groupoid over $(M, \pi)$ with the following structures:
  \begin{enumerate}
    \item the source map $\alpha: U \to M$ is the bundle projection;
    \item the target map is
      \[\beta: U \to M, \qquad \beta = \tau \circ \varphi_X^1,\]
      where $\varphi_X^t: T^*M \to T^*M$ denotes the time-$t$ flow of $X$;
    \item the identity map $\mathtt{1}: M \to U$ is the zero section;
    \item the inverse map is
      \[\iota: U \to U, \qquad \iota = -\varphi_X^1;\]
    \item the multiplication $m: U {_\alpha \times_\beta} U \to U$ is defined as the solution of an ODE (see \cite{CMS17} for details);
    \item the symplectic form on $U$ is
      \[\omega = \int_{0}^{1} (\varphi_X^t)^*\omega_0 dt,\]
      where $\omega_0$ is the standard symplectic structure on $T^*M$.
  \end{enumerate}
\end{theorem}

\begin{remark}
  For the standard symplectic structure $\omega_0$ in Theorem~\ref{thm:poissp}, we use the sign convention that $\omega_0 = - d \theta_0$ for $\theta_0$ the tautological 1-form on $T^*M$.
  This choice ensures that the source map $\alpha$ is Poisson.
  In coordinates, we have $\omega_0 = \sum_i d x_i \wedge d p_i$.
\end{remark}

\begin{remark}
  By a local symplectic groupoid $(\cG, \omega) \rra (M, \pi)$, we mean that the multiplication $m: \cG {_\alpha \times_\beta} \cG \to \cG$ may not be defined on all of its domain.
  In general, the local symplectic groupoid structure cannot be extended to the total space $T^*M$.
  Indeed: the Poisson spray $X$ may not be complete; the flow of the Poisson spray $X$ may contain loops; or the 2-form $\omega$, though non-degenerate near the zero section of $T^*M$, may be degenerate globally.
\end{remark}

In the next results, we introduce a Poisson spray whose local symplectic groupoid provides an integration of a log-canonical Poisson structure.
\begin{lemma} 
  \label{lemma:PoisSp}
  For the log-canonical Poisson structure $\{x_i, x_j\} = \Omega_{ij} x_i x_j$ on $L = \RR^m$, the vector field $X \in \cT_{T^*L}$ given in coordinates $(\bfx, \bfp) = (x_1, \ldots, x_m, p_1, \ldots, p_m)$ on $T^*L$ by
  \begin{equation} 
    \label{eq: PoisSp}
    X = \sum_{1 \leq i,j \leq m}\Omega_{ij}x_i p_i x_j\frac{\partial}{\partial x_j} - \sum_{1 \leq i,j \leq m}\Omega_{ij}x_ip_i p_j\frac{\partial}{\partial p_j}
  \end{equation}
  is a Poisson spray.
  Its flow is given by
  \[\varphi_X^t: T^*L \to T^*L, \qquad (\bfx, \bfp) \mapsto (\bfa^t \circ \bfx, \bfa^{-t} \circ \bfp),\]
  where $a_j = e^{\sum_i \Omega_{ij} x_ip_i}$.
  This flow exists for all time $t \in \RR$ and contains no loops.
\end{lemma}

\begin{proof}
  For the co-vector $\theta = p_1 dx_1 + \ldots + p_m dx_m$ and the point $(\bfx, \bfp)$, we have
  \[\tau_* X|_{(\bfx,\bfp)} = \sum_{i, j}\Omega_{ij} x_i p_i x_j \frac{\partial}{\partial x_j} = \iota_\theta \pi.\]
  To find the flow of $X$, we note that $x_i p_i$ is a constant under the flow of $X$:
  \[\frac{d}{dt}(x_j p_j) = \dot{x}_j p_j + x_j \dot{p}_j = \sum_{1 \leq i \leq m}\Omega_{ij}x_i p_i x_j p_j - \sum_{1 \leq i \leq m}\Omega_{ij}x_ip_i p_j x_j = 0.\]
  Therefore $\sum_{i}\Omega_{ij}x_i p_i$ is constant, and
  \begin{align*}
    x_j(t) & = e^{t \sum_i \Omega_{ij} x_ip_i} x_j(0) = a_j^t x_j(0), \\
    p_j(t) & = e^{-t \sum_i \Omega_{ij} x_ip_i} p_j(0) = a_j^{-t} p_j(0).
  \end{align*}
\end{proof}

This Poisson spray $X$ induces the symplectic groupoid structure below.
\begin{theorem} 
  \label{thm:PoiSpLogC}
  For the log-canonical Poisson structure $\{x_i, x_j\} = \Omega_{ij} x_i x_j$ on $L$ (which is either $\RR^m$ or $\CC^m$), there is a source-simply-connected symplectic groupoid $(\cG, \omega_\cG) \rra (L, \pi)$ with the following structures:
  \begin{enumerate}
    \item $\cG \cong T^*L$ has the coordinates $(\bfx, \bfp) = (x_1, \ldots, x_m, p_1, \ldots, p_m)$;
    \item the source map is the bundle projection $\alpha: T^*L \to L, \quad (\bfx, \bfp) \mapsto \bfx$;
    \item the target map is $\beta: T^*L \to L, \quad (\bfx, \bfp) \mapsto \bfa \circ \bfx$, where $a_j = e^{\sum_i \Omega_{ij} x_ip_i}$;
    \item the identity map is $\mathtt{1}: L \to T^*L, \quad \bfx \mapsto (\bfx, 0)$;
    \item the inverse map is $\iota: T^*L \to T^*L, \quad (\bfx, \bfp) \mapsto (\bfa \circ \bfx, -\bfa^{-1} \circ \bfp)$;
    \item the multiplication map is $m: T^*L {_\alpha \times_\beta} T^*L \to T^*L, \quad \big((\bfa \circ \bfx, \bfp'), (\bfx, \bfp)\big) \mapsto (\bfx, \bfa \circ \bfp' + \bfp)$;
    \item the multiplicative symplectic form $\omega$ is
      \begin{equation*}
	\omega_\cG = \sum_{j} dx_j \wedge dp_j + \sum_{i, j} \Omega_{ij}p_ix_j dx_i \wedge dp_j + \sum_{j < i} \Omega_{ij}p_ip_j dx_i \wedge dx_j + \sum_{j < i} \Omega_{ij}x_ix_j dp_i \wedge dp_j, 
      \end{equation*}
      and equivalently the multiplicative Poisson bi-vector
      \begin{equation*}
      	\sigma_\cG = \omega_\cG^{-1} = -\sum_{i} \frac{\partial}{\partial x_i} \wedge \frac{\partial}{\partial p_i} - \sum_{i, j} \Omega_{ij} x_i p_j \frac{\partial}{\partial x_i} \wedge \frac{\partial}{\partial p_j} +\sum_{j < i} \Omega_{ij}p_ip_j \frac{\partial}{\partial p_i} \wedge \frac{\partial}{\partial p_j} +\sum_{j < i} \Omega_{ij}x_ix_j \frac{\partial}{\partial x_i} \wedge \frac{\partial}{\partial x_j}.
      \end{equation*}
  \end{enumerate}
\end{theorem}
\begin{proof}
  When the underlying field is $\RR$, it is straightforward to check that the Poisson spray \eqref{eq: PoisSp} induces the given groupoid structures. To find the sympletic structure $\omega_\cG$, we have
  \begin{align*}
    \left(\varphi_X^t\right)^*\omega_0 & = \sum_{j} d \left(e^{t \sum_{i} \Omega_{ij} x_ip_i}x_j \right) \wedge d \left(e^{-t \sum_{i} \Omega_{ij} x_ip_i}p_j \right) \\
    & = \sum_{j} \left(d x_j + t \sum_{i} \Omega_{ij} x_i x_j dp_i + t \sum_{i} \Omega_{ij} p_i x_j dx_i \right) \wedge \left(d p_j - t \sum_{i} \Omega_{ij} x_i p_j dp_i - t \sum_{i} \Omega_{ij} p_i p_j dx_i \right) \\
    & = \sum_{j} dx_j \wedge dp_j + 2t \left( \sum_{i, j} \Omega_{ij}p_i x_j d x_i \wedge d p_j  + \sum_{j < i} \Omega_{ij}p_ip_j d x_i \wedge d x_j + \sum_{j < i} \Omega_{ij}x_ix_j d p_i \wedge d p_j \right),
  \end{align*}
  so it follows that
  \begin{align*}
    \omega_\cG & = \int_{0}^{1} (\varphi_X^t)^*\omega_0 dt \\
    & = \sum_{j} dx_j \wedge dp_j
      + \left(
	\sum_{i, j} \Omega_{ij}p_ix_j dx_i \wedge dp_j 
	+ \sum_{j < i} \Omega_{ij}p_ip_j dx_i \wedge dx_j
	+ \sum_{j < i} \Omega_{ij}x_ix_j dp_i \wedge dp_j
      \right).
  \end{align*}
  Note that $\omega_\cG$ is non-degenerate since
  \[(\omega_\cG)^m = m! \bigwedge\limits_{1\leq j\leq m} dx_j \wedge dp_j\]
  is a volume form.

  Choosing the standard frames: $\{dx_i, dp_i\}_i$ for $T^*\cG$ and $\left\{\frac{\partial}{\partial x_i}, \frac{\partial}{\partial p_i}\right\}_i$ for $T\cG$, the bundle map $\sigma_\cG^\flat: T^*\cG \to T\cG$, $\theta \mapsto \iota_\theta \pi_\cG$ is given by
  \begin{align*}
    \sigma_\cG^\flat: &~ dx_i \mapsto -\frac{\partial}{\partial p_i} - \sum_{j} \Omega_{ij} x_i p_j \frac{\partial}{\partial p_j}+ \sum_{j} \Omega_{ij} x_i x_j \frac{\partial}{\partial x_j}, \\
    &~ dp_i \mapsto \frac{\partial}{\partial x_i} - \sum_{j} \Omega_{ij} p_i x_j \frac{\partial}{\partial x_j}+ \sum_{j} \Omega_{ij} p_i p_j \frac{\partial}{\partial p_j},
  \end{align*}
  and the bundle map $\omega_\cG^\sharp: T\cG \to T^*\cG$, $v \mapsto \iota_v \omega_\cG$ is given by
  \begin{align*}
    \omega_\cG^\sharp: &~ \frac{\partial}{\partial x_i} \mapsto dp_i + \sum_{j} \Omega_{ij} p_i x_j dp_j + \sum_{j} \Omega_{ij} p_i p_j dx_j, \\
    &~ \frac{\partial}{\partial p_i} \mapsto -dx_i + \sum_{j} \Omega_{ij} x_i p_j dx_j + \sum_{j} \Omega_{ij} x_i x_j dp_j.
  \end{align*}
  We leave it to the reader to check that $\sigma_\cG^\flat$ and $\omega_\cG^\sharp$ are inverse to each other.

  In the case when the underlying field is $\CC$, the symplectic groupoid structures can be verified directly.
\end{proof}

\begin{corollary} \label{cor:exp}
  For the log-canonical Poisson structure $\{x_i, x_j\} = \Omega_{ij} x_i x_j$ on $L$, the exponential map
  \[\exp: T^*_\pi L \to \cG, \qquad (\bfx, \bfp) \mapsto \left(\bfa \circ \bfx, \bfa^{-1} \circ \bfp\right),\]
  where $a_j = e^{\sum_i \Omega_{ij} x_ip_i}$, is a diffeomorphism.
\end{corollary}

Motivated by \cite{FG09c}, we make the following definition of the symplectic double, which is a source-connected Poisson groupoid of the log-canonical Poisson structure on $L$.
\begin{theorem} \label{thm:SymDBLoc}
  \cite{FG09c}
  For the log-canonical Poisson structure $\{x_i, x_j\} = \Omega_{ij} x_i x_j$ on $L$ (which is either $\RR^m$ or $\CC^m$), the \emph{symplectic double} $(\cD, \sigma_\cD) \rra (L, \pi)$ is a source-connected Poisson groupoid with the following structures:
  \begin{enumerate}
    \item $\cD \cong L \times L^\times$ has the coordinates $(\bfx, \bfs) = (x_1, \ldots, x_m, s_1, \ldots, s_m)$ (where $L^\times$ is respectively either $\RR_\times^m$ or $\CC_\times^m$);
    \item the source map is
      \[\alpha: L \times L^\times \to L, \qquad (\bfx, \bfs) \mapsto \bfx;\]
    \item the target map is
      \[\beta: L \times L^\times \to L, \qquad (\bfx, \bfs) \mapsto \left(x_1 \prod_{i=1}^m s_i^{\Omega_{i1}}, \ldots, x_m \prod_{i=1}^m s_i^{\Omega_{im}}\right);\]
    \item the identity map is $\mathtt{1}: L \to L \times L^\times, \quad \bfx \mapsto (\bfx, 1, \ldots, 1)$;
    \item the inverse map is
      \[\iota: L \times L^\times \to L \times L^\times, \qquad (\bfx,\bfs) \mapsto \left(x_1\prod_{i=1}^m s_i^{\Omega_{i1}}, \ldots, x_m\prod_{i=1}^m s_i^{\Omega_{im}}, \frac{1}{s_1}, \ldots, \frac{1}{s_m}\right);\]
    \item the multiplication map is
      \[\begin{aligned}
	  m: & \left(L \times L^\times\right) {_\alpha \times_\beta} \left(L \times L^\times\right) \to L \times L^\times, \\
	     & \left(\left(x_1 \prod_{i=1}^m s_i^{\Omega_{i1}}, \ldots, x_m \prod_{i=1}^m s_i^{\Omega_{im}}, \bfs'\right), (\bfx, \bfs)\right) \mapsto (\bfx, \bfs' \circ \bfs);
        \end{aligned}\]
    \item the multiplicative Poisson structure is
    \begin{equation*}
      \sigma_\cD = -\sum_{i} x_i s_i \frac{\partial}{\partial x_i} \wedge \frac{\partial}{\partial s_i}+\sum_{j < i} \Omega_{ij}x_ix_j \frac{\partial}{\partial x_i} \wedge \frac{\partial}{\partial x_j}
    \end{equation*}
    and equivalently the singular multiplicative 2-form $\omega_\cD$ is
    \begin{equation*} 
      \omega_\cD = \sigma_\cD^{-1} = \sum_{i} \frac{dx_i}{x_i} \wedge \frac{ds_i}{s_i} + \sum_{j < i} \Omega_{ij} \frac{ds_i}{s_i} \wedge \frac{ds_j}{s_j}.
    \end{equation*}
  \end{enumerate}
\end{theorem}

It may be strange that the Poisson groupoid $\cD \rra L$ is called the \emph{symplectic} double since the Poisson structure drops rank along the coordinate axes.
The symplectic double in \cite{FG09c} is defined to be the restriction $\cD^\times =\{(\bfx, \bfs) ~|~ \bfx \in L^\times \} \cong L^\times \times L^\times$ which is genuinely a symplectic groupoid over $L^\times$.
The Poisson groupoid $\cD \rra L$ may be viewed as the natural extension of the symplectic groupoid $\cD^\times \rra L^\times$.
\begin{remark}
  We relate the target map in Theorem~\ref{thm:SymDBLoc} to the symplectic realization in \cite{GNR17}.
	
  With the notation in Theorem~\ref{thm:SymDBLoc} and the change of variables $x_i = e^{\chi_i}$ and $s_i = e^{\xi_i}$, the symplectic double $(\cD^\times, \omega_\cD)$ becomes the symplectic vector space $(\cU, \omega_\cU)$ with coordinates $(\chi_1, \ldots, \chi_m, \xi_1, \ldots, \xi_m)$, where $\cU$ is either $\RR^{2m}$ or $\CC^{2m}$, and
  \[\omega_\cU = \sum_{i} d\chi_i \wedge d\xi_i + \sum_{j < i} \Omega_{ij} d\xi_i \wedge d\xi_j.\]
  Although $(\cU, \omega_\cU)$ is not a symplectic groupoid integrating $(L, \pi)$ (e.g.\ the identity map cannot be defined), the maps $\alpha$ and $\beta$ are well-defined in these coordinates:
  \begin{align*}
    & \alpha: \cU \to L^\times, \qquad (\chi_1, \ldots, \chi_m, \xi_1, \ldots, \xi_m) \mapsto \left(e^{\chi_1}, \ldots, e^{\chi_m}\right); \\
    & \beta: \cU \to L^\times, \qquad (\chi_1, \ldots, \chi_m, \xi_1, \ldots, \xi_m) \mapsto \left(e^{\chi_1+\sum_i\Omega_{i1}\xi_i}, \ldots, e^{\chi_m+\sum_i\Omega_{im}\xi_i}\right).
  \end{align*}
  The $\alpha$-fibers are symplectic orthogonal to the $\beta$-fibers; both $\alpha: (\cU, \omega_\cU) \to (L, \pi)$ and $\beta: (\cU, - \omega_\cU) \to (L, \pi)$ are Poisson maps.

  For the standard symplectic structure $\omega_0 = \sum_{i} d\chi_i \wedge d\xi_i$ on $\cU$, since
  \[\omega_\cU - \omega_0 = \left(\sum_{i} d\chi_i \wedge d\xi_i + \sum_{j < i} \Omega_{ij} d\xi_i \wedge d\xi_j \right) - \sum_{i} d\chi_i \wedge d\xi_i  =\sum_{j < i} \Omega_{ij} d\xi_i \wedge d\xi_j\]
  is supported on the $\alpha$-fiber, we have that $\beta_*(-\omega_0^{-1}) = \beta_*(-\omega_\cU^{-1}) = \pi$.
  The symplectic realization $\beta: (\cU, -\omega_0) \to (L^\times, \pi)$ plays an essential role in \cite{GNR17}.
\end{remark}

Next, we describe the Lie algebroid of $\cD \rra L$.
We introduce coordinates on $A_\cD \cong L\times L$ with the bundle projection
\[A_\cD \cong L\times L \to L, \qquad (\bfx, \bfxi) \mapsto \bfx.\]
For an $m\times m$ skew-symmetric matrix $\Omega$, we define a Lie algebroid structure on $A_\cD$ with the anchor map
\[\rho_\cD: A_\cD \cong L\times L \to TL, \qquad (\bfx, \bfxi) \mapsto \sum_{i,j} \Omega_{ij}\xi_i x_j\frac{\partial}{\partial x_j},\]
where the kernel of $\rho_\cD$ has trivial bracket.

\begin{proposition}
  The Lie algebroid of the Poisson groupoid $\cD \rra L$ in Theorem~\ref{thm:SymDBLoc} is isomorphic to $A_\cD$.
\end{proposition}
\begin{proof}
  With the notation in Theorem~\ref{thm:SymDBLoc}, if we write $e^{\xi_i} = s_i$, then $\Lie \cD = \ker \left(\alpha_*: T\cD|_{\mathtt{1}_L} \to TL \right)$ is generated by $\frac{\partial}{\partial \xi_i}$, $i = 1, \ldots, m$.
  Rewriting the target map $\beta$ in the $(\bfx, \bfxi)$ coordinates, we get
  \[\beta: (\bfx, \bfxi) \mapsto \left(e^{\sum_i\Omega_{i1}\xi_i}x_1, \ldots, e^{\sum_i\Omega_{im}\xi_i}x_m\right).\]
  Therefore, 
  \begin{equation} 
    \label{eq: D anchor}
    \beta_* \left(\frac{\partial}{\partial \xi_i}\right) = \sum_j \Omega_{ij} x_j \frac{\partial}{\partial x_j}, \quad 
    \beta_* \left(\sum_i \xi_i \frac{\partial}{\partial \xi_i}\right) = \sum_{i,j} \Omega_{ij} \xi_i x_j \frac{\partial}{\partial x_j}.
  \end{equation}
  This shows that the bundle map
  \[A_\cD \to \Lie \cD, \qquad (\bfx, \bfxi) \mapsto \sum_i \xi_i \frac{\partial}{\partial \xi_i}\]
  is a Lie algebroid isomorphism.
\end{proof}

Fixing an $m\times m$ skew-symmetric matrix $\Omega$ and the corresponding log-canonical Poisson space $L$ as in Definition~\ref{def: logPoisson}, the cotangent Lie algebroid $T^*_\pi L$ is not isomorphic to the Lie algebroid $A_\cD$.
In fact, $T^*_\pi L$ is isomorphic to an iterated elementary modification of $A_\cD$ and we obtain a source-connected symplectic groupoid of $(L,\pi)$ via the blow-up construction \cite{GL14}. 
\begin{definition}
  Let $E$ be a vector bundle over $M$ and let $F$ be a subbundle of $E|_N$ for some hypersurface $N\subset M$.
  The \emph{elementary (lower) modfication} of $E$ along $F$, denoted by $[E\!:\!F]$, is the vector bundle with the sheaf of sections
  \[\Gamma\big(M, [E\!:\!F]\big) = \{s \in \Gamma(M, E) : s|_N \in \Gamma(N, F) \}.\]
\end{definition}

\begin{proposition} 
  \cite{GL14}
  Let $A$ be a Lie algebroid over a manifold $M$ and let $B$ be a subbundle of $A|_N$ for some hypersurface $N\subset M$ such that $B$ is also a Lie algebroid over $N$.
  Then $[A\!:\!B]$ is a Lie algebroid over~$M$.
\end{proposition}

Let $\cG \rra M$ be a Lie groupoid and let $\cH \rra N$ be a Lie subgroupoid over a hypersurface $N\subset M$.
We denote the blow-up of $\cG$ along $\cH$ by $\mathrm{Bl}(\cG, \cH)$ with the blow-down map $\nu: \mathrm{Bl}(\cG, \cH) \to \cG$ and write $[\cG\!:\!\cH] = \mathrm{Bl}(\cG, \cH) \setminus (S_N \cup T_N)$, where $S_N$ is the proper transform of $\alpha^{-1}(N)$ and $T_N$ is the proper transform of $\beta^{-1}(N)$.

\begin{theorem} 
  \cite{GL14}
  For a Lie groupoid $\cG \rra M$ and a Lie subgroupoid $\cH \rra N$ for some hypersurface $N\subset M$, there is a unique Lie groupoid structure $[\cG\!:\!\cH] \rra M$ such that the blow-down map $\nu: [\cG\!:\!\cH] \to \cG$ is a groupoid morphism.
  Moreover, blow-up of Lie groupoids corresponds to elementary modification of Lie algebroids.
  More precisely, we have
  \[\Lie[\cG\!:\!\cH] = [\Lie\cG\!:\!\Lie\cH].\]
\end{theorem}

Our task is now clear.
We need to find the Lie subalgebroids of $A_\cD$ such that the iterated elementary modifications along these Lie subalgebroids yield the cotangent algebroid $T^*_\pi L$.

For each $k = 1, \ldots, m$, we define the Lie subalgebroid $A_\cD^k$ of $A_\cD$ as follows.
Let $L_k$ be the codimension-1 subpace in $L$ defined by $x_k = 0$.
As a vector bundle, $A_\cD^k$ is the corank-$1$ subbundle of $A_\cD |_{L_k}$ defined by $\xi_k = 0$.
That is, the anchor map of $A_\cD^k$ is given by
\begin{align*}
  \rho_\cD^k: ~& A_\cD^k \cong L_k\times L_k \to TL_k, \\
  & (\xi_1, \ldots, \xi_{k-1}, \xi_{k+1}, \ldots, \xi_m, x_1, \ldots, x_{k-1}, x_{k+1}, \ldots, x_m) \mapsto \sum_{\substack{1\le i,j \le m\\ i,j\ne k}} \Omega_{ij}\xi_i x_j\frac{\partial}{\partial x_j}.
\end{align*}

\begin{proposition}
  Fixing an $m\times m$ skew-symmetric matrix $\Omega$ and the corresponding log-canonical Poisson space $L$ as in Definition~\ref{def: logPoisson}, the cotangent algebroid $T^*_\pi L$ is isomorphic to the Lie algebroid
  \[[\ldots[[A_\cD\!:\!A_\cD^1]\!:\!A_\cD^2] \ldots \!:\!A_\cD^m].\]
\end{proposition}
\begin{proof}
  The Poisson structure $\pi = \sum_{j < i} \Omega_{ij} x_ix_j\frac{\partial}{\partial x_i} \wedge \frac{\partial}{\partial x_j}$ gives rise to the Poisson anchor map
  \begin{equation}	
    \label{eq:log-can Poiss}
    \pi^\flat: T^*L \to TL, \qquad dx_i \mapsto \sum_j \Omega_{ij}  x_ix_j \frac{\partial}{\partial x_j}.
  \end{equation}
  The result follows by comparing \eqref{eq:log-can Poiss} with the anchor map \eqref{eq: D anchor} for $\Lie \cD\cong A_\cD$.
\end{proof}

To obtain a source-connected symplectic groupoid of the log-canonical Poisson space $(L, \pi)$ which integrates $T^*_\pi L$, we need to iteratively blow up the Poisson groupoid $\cD \rra L$ along the subgroupoids integrating $A_\cD^k$, for $k =1, \ldots, m$.
Indeed, the Lie subalgebroid $A_\cD^k$ integrates to the subgroupoid $\cD_k \rra L_k$ defined by $x_k=0$ and $s_k=1$, so we define the Lie groupoid
\[\cB = [\ldots[[\cD\!:\!\cD_1]\!:\!\cD_2] \ldots \!:\!\cD_m].\]
Using the blow-up coordinates $u_i = \frac{s_i-1}{x_i}$, we have the following symplectic groupoid structure on $\cB \rra L$.
\begin{theorem} 
  \label{th:blowup groupoid}
  For the log-canonical Poisson structure $\{x_i, x_j\} = \Omega_{ij} x_i x_j$ on $L$ (which is either $\RR^m$ or $\CC^m$), we have that $(\cB, \omega_\cB) \rra (L, \pi)$ is a source-connected symplectic groupoid with the following structures:
  \begin{enumerate}
    \item $\cB \subset L \times L$ has the coordinates $(\bfx, \bfu) = (x_1, \ldots, x_m, u_1, \ldots, u_m)$;
    \item $\cB$ is given by $x_i u_i + 1 > 0$ if $L = \RR^m$ and $x_i u_i + 1\ne 0$ if $L = \CC^m$;
    \item the source map is
      \[\alpha: \cB \to L, \qquad (\bfx, \bfu) \mapsto \bfx;\]
    \item the target map is
      \[\beta: \cB \to L, \qquad (\bfx, \bfu) \mapsto \left(x_1 \prod_{i=1}^m (x_iu_i+1)^{\Omega_{i1}}, \ldots, x_m \prod_{i=1}^m (x_iu_i+1)^{\Omega_{im}}\right);\]
	\item the identity map $\mathtt{1}: L \to \cB, \quad \bfx \mapsto (\bfx, \bfzero)$;
    \item the inverse map is $\iota: \cB \to \cB$ sending $(\bfx, \bfu)$ to 
      \[\left(x_1 \prod_{i=1}^m (x_iu_i+1)^{\Omega_{i1}}, \ldots, x_m \prod_{i=1}^m (x_iu_i+1)^{\Omega_{im}}, -\frac{u_1\prod_{i=1}^m (x_iu_i+1)^{-\Omega_{i1}}}{x_1u_1+1}, \ldots, -\frac{u_m\prod_{i=1}^m (x_iu_i+1)^{-\Omega_{im}}}{x_mu_m+1}\right);\]
    \item the multiplication map is
      \[\begin{aligned}
	  m: & \cB {_\alpha \times_\beta} \cB \to \cB, \\
	     & \left((\bfx, \bfu), \left(x_1 \prod_{i=1}^m (u_ix_i+1)^{\Omega_{i1}}, \ldots, x_m \prod_{i=1}^m (u_ix_i+1)^{\Omega_{im}}, \bfu'\right)\right) \mapsto (\bfx, \bfu''),
        \end{aligned}\]
        where $u''_j = \left( u'_j (x_ju_j+1) \prod_{i=1}^m (x_iu_i+1)^{\Omega_{ij}} + u_j \right)$;
    \item the multiplicative Poisson structure is
    \begin{equation*}
	\sigma_\cB = -\sum_{i} (x_i u_i + 1) \frac{\partial}{\partial x_i} \wedge \frac{\partial}{\partial u_i} - \sum_{i, j} \Omega_{ij}x_i u_j \frac{\partial}{\partial x_i} \wedge \frac{\partial}{\partial u_j} +\sum_{j < i} \Omega_{ij}u_iu_j \frac{\partial}{\partial u_i} \wedge \frac{\partial}{\partial u_j} + \sum_{j < i} \Omega_{ij}x_ix_j \frac{\partial}{\partial x_i} \wedge \frac{\partial}{\partial x_j},
    \end{equation*}
    and equivalently the multiplicative symplectic structure $\omega_\cB$ is
    \begin{align*} 
	\omega_\cB & = \sum_{i} \frac{1}{x_iu_i+1}dx_i \wedge du_i + \sum_{i, j} \frac{\Omega_{ij}u_ix_j}{(x_iu_i+1)(x_ju_j+1)}  dx_i \wedge du_j \\
	& \quad + \sum_{j < i}\frac{\Omega_{ij}u_iu_j}{(x_iu_i+1)(x_ju_j+1)}  dx_i \wedge dx_j + \sum_{j < i} \frac{\Omega_{ij}x_ix_j}{(x_iu_i+1)(x_ju_j+1)}  du_i \wedge du_j.
      \end{align*}
  \end{enumerate}
\end{theorem}

\begin{proof}
  These are the same groupoid structures as in Theorem~\ref{thm:SymDBLoc} with the change of variables $u_i = \frac{s_i-1}{x_i}$.
  The difference is precisely that $\omega_\cB$ is no longer singular.
  Note that $\omega_\cB$ is also non-degenerate since
  \[(\omega_\cB)^m = m!  \prod_{1\leq i\leq m} \frac{1}{x_iu_i+1} \bigwedge\limits_{1\leq i\leq m} dx_i \wedge du_i\]
  is a volume form.
\end{proof}

\begin{remark}
  Recall that the infinitesimal object of a Poisson groupoid is a Lie bialgebroid \cite{MX94}.
  The Lie bialgebroid of the Poisson groupoid $\cD \rra L$ is isomorphic to $\big(A_\cD, TL(-\log D)\big)$, where $TL(-\log D)$ is the log-tangent bundle with respect to the normal crossing divisor $D = L_1 + L_2 + \ldots + L_k$ \cite{GL14}.
  That is,
  \[TL(-\log D) = [\ldots[[TL\!:\!TL_1]\!:\!TL_2] \ldots \!:\!TL_m],\]
  or equivalently the sections of $TL(-\log D)$ are the vector fields on $L$ tangent to $L_k$ for $k = 1, \ldots, m$.
  As a vector bundle, $TL(-\log D)$ is generated by $x_k \frac{\partial}{\partial x_k}$, $k = 1, \ldots, m$, and its dual bundle $T^*L(-\log D)$ is generated by $\frac{d x_k}{x_k}$, $k = 1, \ldots, m$.
  The bundle map
  \[A_\cD \to T^*L(-\log D), \qquad (\bfx, \bfxi) \mapsto \sum_k\xi_k\frac{d x_k}{x_k}\]
  identifies $A_\cD$ with $T^*L(-\log D)$.

  The iterated blow-up construction that takes $\cD \rra L$ to $\cB \rra L$ corresponds to the elementary modification that takes the Lie bialgebroid $(A_\cD, TL(-\log D))$ to the Lie bialgebroid $(T^*_\pi L, TL)$, where $T^*_\pi L$ is an elementary lower modification of $A_\cD$ and $TL$ is an elementary upper modification of $TL(-\log D)$.
  See the first author's thesis for details \cite{Li13}.
\end{remark}

For an integrable Lie algebroid, every source-connected groupoid receives a surjective groupoid map from the source-simply-connected groupoid.
This means that we have a natural symplectic groupoid morphism $\kappa:\cG\to\cB$.
Together with the blow-down map $\nu:\cB\to\cD$, the following commutative diagram summarizes the maps among the groupoids $\cG$, $\cB$ and $\cD$.
\begin{equation} 
  \label{eq:LGpdCD}
  \xymatrix{
    \cG  \ar[dr]_{\kappa} \ar[rr]^{\lambda} && \cD \\
      & \cB \ar[ur]_{\nu}
  }
\end{equation}
Explicitly in coordinates, these maps are given by
\begin{align*}
  & \kappa: \cG \to \cB, \qquad (\bfx, \bfp) \mapsto \left(\bfx, \frac{e^{x_1p_1}-1}{x_1}, \ldots, \frac{e^{x_mp_m}-1}{x_m}\right); \\
  & \nu: \cB \to \cD, \qquad (\bfx, \bfu) \mapsto (\bfx, x_1u_1+1, \ldots, x_mu_m+1); \\
  & \lambda: \cG \to \cD, \qquad (\bfx, \bfp) \mapsto (\bfx, e^{x_1p_1}, \ldots, e^{x_mp_m}).
\end{align*}

\begin{remark}
  If $L = \RR^m$, then $\cB \rra L$ is source-simply-connected and all three groupoid maps are groupoid isomorphisms after restricting the base to $L^\times = \RR_\times^m$, but if $L = \CC^m$, then $\cB$ is not source-simply-connected and $\kappa: \cG \to \cB$ is a covering map for each source fiber.
\end{remark}

We summarize the pros and cons of the three groupoids $\cG \rra L$, $\cB \rra L$ and $\cD \rra L$.
The symplectic groupoid $\cG \rra L$ is source-simply-connected when the underlying field is $\CC$, the tradeoff is that its groupoid structure maps are transcendental.
On the other hand, the multiplicative Poisson structure of $\cD \rra L$ has a simple expression and its groupoid structure maps are algebraic, but it is only symplectic when restricted to $L^\times$.
Comparing to $\cG \rra L$, the symplectic groupoid $\cB \rra L$ has the added bonus that its groupoid structure maps are algebraic, but the tradeoff is that the expressions for its multiplicative Poisson structure and its groupoid structures are more complicated.

\section{Hamiltonian Perspective on Mutations}
\label{sec:mutations}

In this section, we introduce mutations of cluster charts from the Hamiltonian viewpoint.
This perspective was first given in \cite{FG09c} and is the foundation for the main results of \cite{GNR17}.
We demonstrate how the Hamiltonian perspective provides a canonical choice of mutations for groupoid charts which glue to give a symplectic groupoid integrating various log-canonical Poisson structures on cluster varieties.
The standard combinatorics of $\bfc$-vectors, $\bfg$-vectors, and $F$-polynomials are then lifted to provide descriptions of iterations of the gluing maps for groupoid charts.

While in Section~\ref{sec:local} it was most convenient to use covariant notation for all maps of groupoids, due to the algebraic nature of cluster mutations and also the piecewise nature of the formulas, in this section we use contravariant notation for describing all maps, i.e.\ we describe them via the pullback of coordinate functions.

Let $\tilde B=(B_{ij})$ be an $m\times n$ integer matrix with $m\ge n$.
Write $B$ for the upper $n\times n$ submatrix of $\tilde B$ and assume $B$ is skew-symmetrizable, i.e.\ there exists a diagonal integer matrix $D=\diag(d_1,\ldots,d_n)$ with each $d_i>0$ so that $DB$ is skew-symmetric. 
Such an $m\times n$ matrix $\tilde B$ with skew-symmetrizable principal submatrix $B$ is called an \emph{exchange matrix}.
We fix a skew-symmetrizing matrix $D$ and refer to a skew-symmetric $m\times m$ matrix $\Omega=(\Omega_{ij})$ as \emph{$D$-compatible} with $\tilde B$ if $\tilde B^T\Omega=[D\ \boldsymbol{0}]$, where $\boldsymbol{0}$ denotes an $n\times(m-n)$ matrix with all zero entries.
In this case, we call $(\tilde B,\Omega)$ a \emph{$D$-compatible pair} or simply a \emph{compatible pair} if the symmetrizing matrix $D$ is understood.
Note that the existence of a matrix $\Omega$ comptible with $\tilde B$ implies $\tilde B$ has rank $n$, in other words the columns of $\tilde B$ are linearly independent.

\subsection{Mutation of Cluster Charts}
\label{sec:cluster mutations}

We begin by recalling the Hamiltonian perspective of mutations for cluster charts.
Let $L_\cX=\RR^n$ and $L_\cA=\RR^m$ with corresponding orthants $L^\times_\cX=\RR_\times^n$ and $L^\times_\cA=\RR_\times^n$.
Write $\bfy=(y_1,\ldots,y_n)$ for a set of coordinates on $L_\cX$ and $\bfx=(x_1,\ldots,x_m)$ for a set of coordinates on $L_\cA$, we use the same symbols for their restrictions to $L^\times_\cX$ and $L^\times_\cA$.
Given a $D$-compatible pair of matrices $(\tilde B,\Omega)$, denote by $\{\cdot,\cdot\}_\cX$ and $\{\cdot,\cdot\}_\cA$ the log-canonical Poisson brackets on $L_\cX$ and $L_\cA$ given by
\begin{equation}
  \label{eq:brackets}
  \{y_k,y_\ell\}_\cX=d_kB_{k\ell}y_ky_\ell\qquad\text{and}\qquad\{x_i,x_j\}_\cA=\Omega_{ij}x_ix_j.
\end{equation}
An easy calculation using the compatibility condition for the pair $(\tilde B,\Omega)$ shows that there is a Poisson map
\begin{equation}\label{eq:ensemble}
	\rho:L_\cA\to L_\cX,\qquad \rho^*(y_k)=\hat y_k:=\prod_{i=1}^m x_i^{B_{ik}}
\end{equation}
defined on the locus where $x_i\ne0$ if there exists $k$ with $B_{ik}<0$, i.e.\ we have $\{\hat y_k,\hat y_\ell\}_\cA=d_kB_{k\ell}\hat y_k\hat y_\ell$ for $1\le k,\ell\le n$.
Motivated by the terminology of Fock and Goncharov \cite{FG09a}, we will refer to the pair of Poisson manifolds $L_\cX$ and $L_\cA$ together with the Poisson map $\rho$ as a \emph{Poisson ensemble} associated to the compatible pair $(\tilde B,\Omega)$.

The \emph{Euler dilogarithm} is the function of a single real variable defined by
\[\Li_2(y)=-\int_0^y \frac{\log(1-u)}{u}du,\qquad y<1.\]
It will become apparent from the results below that, from the Poisson perspective, the Euler dilogarithm lies at the heart of cluster algebra theory.
However, we will work in a slightly more general setting.

Fix a collection of positive integers $\bfr=(r_1,\ldots,r_n)\in\ZZ_{>0}^n$ and write $L_\bfr=\prod_{\ell=1}^n \RR^{r_\ell-1}$ with coordinates $z_{\ell,j}$ for $1\le\ell\le n$, $1\le j\le r_\ell-1$.
Write $L^\times_\bfr=\prod_{\ell=1}^n \RR_\times^{r_\ell-1}$ for the orthant inside $L_\bfr$.
We identify the $\ell$-th component of a point in $L_\bfr$ with the degree $r_\ell$ monic polynomial $Z_\ell(u)\in\RR[u]$ with $Z_\ell(0)=1$ given by
\begin{equation} \label{eq:gen mut poly}
Z_\ell(u)=1+z_{\ell,1}u+\cdots+z_{\ell,r_\ell-1}u^{r_\ell-1}+u^{r_\ell}.
\end{equation}
There is a natural \emph{star-involution} $z_{\ell,j}\mapsto z^*_{\ell,j}:=z_{\ell,r_\ell-j}$ on $L_\bfr$ and we write $Z^*_\ell(u)$ for the image of $Z_\ell(u)$ under this involution, more directly $Z^*_\ell(u)=u^{r_\ell} Z_\ell(u^{-1})$ is the the polynomial with reversed coefficients.
It will be convenient to introduce the notation $j^*=r_\ell-j$ so that $z^*_{\ell,j}=z_{\ell,j^*}$.
Although $j^*$ depends on $\ell$, we suppress it from the notation.

Set 
\[L_{\cX,\bfr}:=L_\cX\times L_\bfr\qquad\text{ and }\qquad L_{\cA,\bfr}:=L_\cA\times L_\bfr\]
with corresponding orthants $L^\times_{\cX,\bfr}$ and $L^\times_{\cA,\bfr}$.
Here we extend the Poisson structures on $L_\cX$ and $L_\cA$ to $L_{\cX,\bfr}$ and $L_{\cA,\bfr}$ so that the coordinate functions $z_{\ell,j}$ on $L_\bfr$ are Casimirs. 
Then the map $\rho:L_\cA\to L_\cX$ extends to a Poisson map $L_{\cA,\bfr}\to L_{\cX,\bfr}$, which we still denote by $\rho$, by acting as the identity map on $L_\bfr$ components.

Given a choice of sign $\varepsilon\in\{\pm1\}$, define Hamiltonian functions $h^{k,\varepsilon}_{\cX,\bfr}\in\cO_{L^\times_{\cX,\bfr}}$ and $h^{k,\varepsilon}_{\cA,\bfr}\in\cO_{L^\times_{\cA,\bfr}}$ for $1\le k\le n$ by
\begin{equation}
  \label{eq:hamiltonians}
  h_{\cX,\bfr}^{k,\varepsilon}:=-\frac{\varepsilon}{d_k}\int_0^{y_k^\varepsilon} \frac{\log\big(Z_k^\circ(u)\big)}{u}du\qquad\text{and}\qquad h_{\cA,\bfr}^{k,\varepsilon}:=-\frac{\varepsilon}{d_k}\int_0^{\hat y_k^\varepsilon} \frac{\log\big(Z_k^\circ(u)\big)}{u}du,
\end{equation}
where
\[ Z_k^\circ(u)=\begin{cases} Z_k(u) & \text{if $\varepsilon=+1$;}\\ Z_k^*(u) & \text{if $\varepsilon=-1$.} \end{cases}\]
Clearly, there is an equality $\rho^*(h_{\cX,\bfr}^{k,\varepsilon})=h_{\cA,\bfr}^{k,\varepsilon}$, where $\rho$ here denotes the restriction to $L^\times_{\cA,\bfr}$.
\begin{remark}
  \label{rem:cluster case}
  Observe that when $r_k=1$, i.e.\ $Z_k(u)=1+u$, we have $h_{\cX,\bfr}^{k,\varepsilon}:=\frac{\varepsilon}{d_k}\Li_2(-y_k^\varepsilon)$ and $h_{\cA,\bfr}^{k,\varepsilon}:=\frac{\varepsilon}{d_k}\Li_2(-\hat y_k^\varepsilon)$.
  We will refer to this as the ``cluster case'' since the mutation rules presented below reduce to the classical mutations for cluster algebras when $Z_k(u)=1+u$ (c.f. Remark~\ref{rem:cluster case 2}).
\end{remark}

Write $X_{\cX,\bfr}^{k,\varepsilon}\in\cT_{L^\times_{\cX,\bfr}}$ and $X_{\cA,\bfr}^{k,\varepsilon}\in\cT_{L^\times_{\cA,\bfr}}$ for the Hamiltonian vector fields associated to $h_{\cX,\bfr}^{k,\varepsilon}$ and $h_{\cA,\bfr}^{k,\varepsilon}$ respectively, i.e.\ the vector fields naturally associated to the derivations $\{h_{\cX,\bfr}^{k,\varepsilon},\cdot\}_\cX$ and $\{h_{\cA,\bfr}^{k,\varepsilon},\cdot\}_\cA$.
\begin{lemma}
  \label{le:hamiltonian dynamics}
  For $1\le k\le n$ and $\varepsilon\in\{\pm1\}$, the vector fields $X_{\cX,\bfr}^{k,\varepsilon}$ and $X_{\cA,\bfr}^{k,\varepsilon}$ determine the following dynamics on $L^\times_{\cX,\bfr}$ and $L^\times_{\cA,\bfr}$:
  \begin{align}
    \label{eq:X dynamics}
    \dot y_\ell&:=\{h_{\cX,\bfr}^{k,\varepsilon},y_\ell\}_\cX=-B_{k\ell}\log\big(Z_k^\circ(y_k^\varepsilon)\big)y_\ell\quad\text{for}\quad 1\le\ell\le n;\\
    \label{eq:A dynamics}
    \dot x_j&:=\{h_{\cA,\bfr}^{k,\varepsilon},x_j\}_\cA=-\delta_{jk}\log\big(Z_k^\circ(\hat y_k^\varepsilon)\big)x_k\quad\text{for}\quad 1\le j\le m.
  \end{align}
  Moreover, the points of $L^\times_\bfr$ are fixed by the Hamiltonian flows.
\end{lemma}
\begin{proof}
  Equation~\eqref{eq:X dynamics} was essentially proven in \cite{GNR17}.
  Since equation~\eqref{eq:A dynamics} seems to technically be new, we will indicate the key steps in the computation here:
  \[\{h_{\cA,\bfr}^{k,\varepsilon},x_j\}_\cA=-\frac{\log\big(Z_k^\circ(\hat y_k^\varepsilon)\big)}{d_k\hat y_k}\{\hat y_k,x_j\}=-\delta_{jk}\log\big(Z_k^\circ(\hat y_k^\varepsilon)\big)x_j.\]
\end{proof}

Since $B_{kk}=0$, it immediately follows that $y_k$ is a conserved quantity under the flow of $L^\times_{\cX,\bfr}$ by the vector field $X_{\cX,\bfr}^{k,\varepsilon}$.
Since the map $\rho$ is Poisson with $\hat y_k=\rho^*(y_k)$ and $h_{\cA,\bfr}^{k,\varepsilon}=\rho^*(h_{\cX,\bfr}^{k,\varepsilon})$, we also see that $\hat y_k$ is a conserved quantity under the flow of $L^\times_{\cA,\bfr}$ by the vector field $X_{\cA,\bfr}^{k,\varepsilon}$.
\begin{corollary}
  \label{cor:time-one flows}
  For $1\le k\le n$ and $\varepsilon\in\{\pm1\}$, the time-$t$ flow $\varphi_{\cX,\bfr}^t:L^\times_{\cX,\bfr}\to L^\times_{\cX,\bfr}$ of the Hamiltonian vector field $X_{\cX,\bfr}^{k,\varepsilon}$ is given on coordinates by
  \[(\varphi_{\cX,\bfr}^t)^*(y_\ell)=\big(Z_k^\circ(y_k^\varepsilon)\big)^{-tB_{k\ell}}y_\ell,\qquad (\varphi_{\cX,\bfr}^t)^*(z_{\ell,i})=z_{\ell,i},\]
  and the time-$t$ flow $\varphi_{\cA,\bfr}^t:L^\times_{\cA,\bfr}\to L^\times_{\cA,\bfr}$ of the Hamiltonian vector field $X_{\cA,\bfr}^{k,\varepsilon}$ is given on coordinates by
  \[(\varphi_{\cA,\bfr}^t)^*(x_j)=\big(Z_k^\circ(\hat y_k^\varepsilon)\big)^{-t\delta_{jk}}x_j,\qquad (\varphi_{\cA,\bfr}^t)^*(z_{\ell,i})=z_{\ell,i}.\]
  These maps preserve the respective Poisson structures.
\end{corollary}
\begin{proof}
  The computation of the flows is immediate from the preceding discussion.
  That the flow of these vector fields preserve the Poisson structures follows from their Hamiltonian definition.
\end{proof}

In what follows we use the notation $[b]_+:=\max\{b,0\}$.
For $1\le k\le n$ and a sign $\varepsilon\in\{\pm1\}$, define the \emph{mutation of the compatible pair $(\tilde B,\Omega)$ in direction $k$} by $\mu_{\bfr,k,\varepsilon}(\tilde B,\Omega)=\big(\tilde B',\Omega'\big)$, where
\begin{itemize}
  \item $\tilde B'=(B'_{ij})$ is given by
    \begin{equation}
      \label{eq:matrix mutation}
      B'_{ij}=
      \begin{cases}
        -B_{ij} & \text{if $i=k$ or $j=k$;}\\ 
        B_{ij} + [-\varepsilon B_{ik}r_k]_+ B_{kj} + B_{ik} [\varepsilon r_kB_{kj}]_+ & \text{otherwise;}
      \end{cases}
    \end{equation}
  \item $\Omega'=E_{\bfr,k,\varepsilon}^T\Omega E_{\bfr,k,\varepsilon}$ for $E_{\bfr,k,\varepsilon}$ the $m\times m$ matrix with entries
    \[E_{ij}=\begin{cases}\delta_{ij} & \text{if $j\ne k$;}\\ -1 & \text{if $i=j=k$;}\\ [-\varepsilon B_{ik}r_k]_+ & \text{if $i\ne j=k$.}\end{cases}\]
\end{itemize}
It is an easy exercise to check that $\big(\tilde B',\Omega'\big)$ is again $D$-compatible and that the mutation of compatible pairs is an involution, more precisely 
\[\mu_{\bfr,k,\varepsilon}\mu_{\bfr,k,\varepsilon'}(\tilde B,\Omega)=(\tilde B,\Omega)\]
for any signs $\varepsilon,\varepsilon'\in\{\pm1\}$.
In particular, the mutation $\mu_{\bfr,k,\varepsilon}$ acting on $D$-compatible pairs is independent of the sign $\varepsilon$ and we simply write $\mu_{\bfr,k}$ for this mutation when the choice of sign can be ignored.

Let $L_{\cX'}=\RR^n$ and $L_{\cA'}=\RR^m$ with coordinates $\bfy'=(y'_1,\ldots,y'_n)$ and $\bfx'=(x'_1,\ldots,x'_m)$ respectively and with corresponding orthants $L^\times_{\cX'}=\RR_\times^n$ and $L^\times_{\cA'}=\RR_\times^m$.
For $1\le k\le n$, set $(\tilde B',\Omega')=\mu_{\bfr,k}(\tilde B,\Omega)$ and consider the log-canonical Poisson structures $\{\cdot,\cdot\}_{\cX'}$ and $\{\cdot,\cdot\}_{\cA'}$ on $L_{\cX'}$ and $L_{\cA'}$ given by
\begin{equation*}
  \{y'_k,y'_\ell\}_{\cX'}=d_kB'_{k\ell}y'_ky'_\ell\qquad\text{and}\qquad\{x'_i,x'_j\}_{\cA'}=\Omega'_{ij}x'_ix'_j.
\end{equation*}
Define $L_{\cX',\bfr}$ and $L_{\cA',\bfr}$ as above with corresponding orthants $L^\times_{\cX',\bfr}$ and $L^\times_{\cA',\bfr}$.
Write $z'_{\ell,j}$ for coordinates on the $L_\bfr$ components.
\begin{lemma}
  \label{le:tropical cluster transformations}
  For $1\le k\le n$ and $\varepsilon\in\{\pm1\}$, there are Poisson maps $\tau_{\cX,\bfr}^{k,\varepsilon}:L^\times_{\cX,\bfr}\to L^\times_{\cX',\bfr}$ and $\tau_{\cA,\bfr}^{k,\varepsilon}:L^\times_{\cA,\bfr}\to L^\times_{\cA',\bfr}$ given on coordinates by
  \begin{align*}
    (\tau_{\cX,\bfr}^{k,\varepsilon})^*(y'_\ell)
    &=\begin{cases} 
      y_k^{-1} & \text{if $\ell=k$;}\\ 
      y_\ell y_k^{[\varepsilon r_kB_{k\ell}]_+} & \text{if $\ell\ne k$;}
    \end{cases}\\
    (\tau_{\cA,\bfr}^{k,\varepsilon})^*(x'_j)
    &=\begin{cases} 
      x_k^{-1}\prod\limits_{i=1}^m x_i^{[-\varepsilon B_{ik}r_k]_+} & \text{if $j=k$;}\\
      x_j & \text{if $j\ne k$;}
    \end{cases}\\
    (\tau_{-,\bfr}^{k,\varepsilon})^*(z'_{\ell,j})
    &=\begin{cases}
      z^*_{k,j} & \text{if $\ell=k$;}\\
      z_{\ell,j} & \text{if $\ell\ne k$.}
    \end{cases}  
  \end{align*}
\end{lemma}
\begin{proof}
  Since the coordinates on $L_\bfr$ are Casimirs there is nothing to check for these coordinates.
  By skew-symmetry of the Poisson brackets, there are essentially only two cases to check on $L_\cX$ and on $L_\cA$ but just one of these is non-trivial for each transformation.
  For $\ell,\ell'\ne k$, we have
  \begin{align*}
    \{(\tau_{\cX,\bfr}^{k,\varepsilon})^*(y'_\ell),(\tau_{\cX,\bfr}^{k,\varepsilon})^*(y'_{\ell'})\}_\cX
    &=\left\{y_\ell y_k^{[\varepsilon r_kB_{k\ell}]_+},y_{\ell'} y_k^{[\varepsilon r_kB_{k\ell'}]_+}\right\}_\cX\\
    &=(d_\ell B_{\ell\ell'}+d_k[\varepsilon r_kB_{k\ell}]_+B_{k\ell'}+d_\ell B_{\ell k}[\varepsilon r_kB_{k\ell'}]_+) y_\ell y_k^{[\varepsilon r_kB_{k\ell}]_+} y_{\ell'} y_k^{[\varepsilon r_kB_{k\ell'}]_+}\\
    &=(d_\ell B_{\ell\ell'}+d_\ell[-\varepsilon B_{\ell k}r_k]_+B_{k\ell'}+d_\ell B_{\ell k}[\varepsilon r_kB_{k\ell'}]_+) y_\ell y_k^{[\varepsilon r_kB_{k\ell}]_+} y_{\ell'} y_k^{[\varepsilon r_kB_{k\ell'}]_+}\\
    &=d_\ell B'_{\ell\ell'} (\tau_{\cX,\bfr}^{k,\varepsilon})^*(y'_\ell) (\tau_{\cX,\bfr}^{k,\varepsilon})^*(y'_{\ell'}).
  \end{align*}
  The case where $\ell$ or $\ell'$ are equal to $k$ is immediate and we omit the details.

  For $\tau_{\cA,\bfr}^{k,\varepsilon}$, we only check the brackets of $x'_k$ and $x'_j$ for $j\ne k$:
  \begin{align*}
    \{(\tau_{\cA,\bfr}^{k,\varepsilon})^*(x'_k),(\tau_{\cA,\bfr}^{k,\varepsilon})^*(x'_j)\}_\cA
    &=\left\{x_k^{-1}\prod\limits_{i=1}^m x_i^{[-\varepsilon B_{ik}r_k]_+},x_j\right\}_\cA\\
    &=\left(-\Omega_{kj}+\sum_{i=1}^m [-\varepsilon B_{ik}r_k]_+\Omega_{ij}\right) \left(x_k^{-1} \prod\limits_{i=1}^m x_i^{[-\varepsilon B_{ik}r_k]_+}\right)x_j\\
    &=\Omega'_{kj} (\tau_{\cA,\bfr}^{k,\varepsilon})^*(x'_k) (\tau_{\cA,\bfr}^{k,\varepsilon})^*(x'_j).
  \end{align*}
\end{proof}

Since the mutation of $D$-compatible pairs is an involution, we obtain analogous Poisson maps from $L^\times_{\cX',\bfr}$ to $L^\times_{\cX,\bfr}$ and from $L^\times_{\cA',\bfr}$ to $L^\times_{\cA,\bfr}$ which, by a slight abuse of notation, we denote by the same symbols $\tau_{\cX,\bfr}^{k,\varepsilon}$ and $\tau_{\cA,\bfr}^{k,\varepsilon}$ respectively.
It is important to note that the maps $\tau_{\cX,\bfr}^{k,\varepsilon}$ and $\tau_{\cA,\bfr}^{k,\varepsilon}$ are not independent of the sign $\varepsilon$ and, moreover, that they are not involutions, i.e.\ $(\tau_{\cX,\bfr}^{k,\varepsilon})^2$ and $(\tau_{\cA,\bfr}^{k,\varepsilon})^2$ do not give identity maps on $L^\times_{\cX,\bfr}$ nor on $L^\times_{\cA,\bfr}$.
However, an easy calculation shows that $\tau_{\cX,\bfr}^{k,\varepsilon}\tau_{\cX,\bfr}^{k,-\varepsilon}$ and $\tau_{\cA,\bfr}^{k,\varepsilon}\tau_{\cA,\bfr}^{k,-\varepsilon}$ will be identity maps for either choice of sign $\varepsilon$.

For $1\le k\le n$ and $\varepsilon\in\{\pm1\}$, define the \emph{cluster mutations in direction $k$} by 
\[\mu_{\cX,\bfr}^{k,\varepsilon}:=\tau_{\cX,\bfr}^{k,\varepsilon}\circ\varphi_{\cX,\bfr}^1:L^\times_{\cX,\bfr}\to L^\times_{\cX',\bfr}\qquad\text{ and }\qquad\mu_{\cA,\bfr}^{k,\varepsilon}:=\tau_{\cA,\bfr}^{k,\varepsilon}\circ\varphi_{\cA,\bfr}^1:L^\times_{\cA,\bfr}\to L^\times_{\cA',\bfr}.\]
\begin{corollary}
  \label{cor:cluster mutation}
  For $1\le k\le n$ and $\varepsilon\in\{\pm1\}$, the cluster mutations provide Poisson morphisms $\mu_{\cX,\bfr}^{k,\varepsilon}:L^\times_{\cX,\bfr}\to L^\times_{\cX',\bfr}$ and $\mu_{\cA,\bfr}^{k,\varepsilon}:L^\times_{\cA,\bfr}\to L^\times_{\cA',\bfr}$ given on coordinates by
  \begin{align}
    \label{eq:X mutation}
    (\mu_{\cX,\bfr}^{k,\varepsilon})^*(y'_\ell)&=\begin{cases} y_k^{-1} & \text{if $\ell=k$;}\\ y_\ell y_k^{[\varepsilon r_kB_{k\ell}]_+} Z_k^\circ(y_k^\varepsilon)^{-B_{k\ell}} & \text{if $\ell\ne k$;}\end{cases}\\
    \label{eq:A mutation}
    (\mu_{\cA,\bfr}^{k,\varepsilon})^*(x'_j)&=\begin{cases} x_k^{-1}\left(\prod\limits_{i=1}^m x_i^{[-\varepsilon B_{ik}r_k]_+}\right)Z_k^\circ(\hat y_k^\varepsilon) & \text{if $j=k$;}\\ x_j & \text{if $j\ne k$;}\end{cases}\\
    \nonumber (\mu_{-,\bfr}^{k,\varepsilon})^*(z'_{\ell,j})
    &=\begin{cases}
      z^*_{k,j} & \text{if $\ell=k$;}\\
      z_{\ell,j} & \text{if $\ell\ne k$.}
    \end{cases}
  \end{align}
  Moreover, the cluster mutations $\mu_{\cX,\bfr}^{k,\varepsilon}$ and $\mu_{\cA,\bfr}^{k,\varepsilon}$ are involutions which do not depend on the choice of sign $\varepsilon$.
\end{corollary}
\begin{proof}
  The first claim immediately follows by combining Corollary~\ref{cor:time-one flows} and Lemma~\ref{le:tropical cluster transformations}.
  The final claim is a consequence of the identities $B'_{k\ell}=-B_{k\ell}$, $B'_{ik}=-B_{ik}$, and $B_{ik}=[B_{ik}]_+-[-B_{ik}]_+$.
\end{proof}

\begin{remark}
  \label{rem:extended cluster charts}
  While we have restricted to considering $L^\times_{\cX,\bfr}$ and $L^\times_{\cA,\bfr}$ to motivate the Hamiltonian natures of the mutations $\mu_{\cX,\bfr}^{k,\varepsilon}$ and~$\mu_{\cA,\bfr}^{k,\varepsilon}$, the cluster mutations \eqref{eq:X mutation} and \eqref{eq:A mutation} are also well-defined on open dense subsets of $L_{\cX,\bfr}$ and $L_{\cA,\bfr}$.
  More precisely, the map $\mu_{\cX,\bfr}^{k,\varepsilon}$ is not defined on the locus where $y_k=0$ nor where $Z_k^\circ(y_k^\varepsilon)=0$ if $B_{k\ell}>0$ for some $\ell$ and that the map $\mu_{\cA,\bfr}^{k,\varepsilon}$ is not defined on the locus where $x_k=0$ nor where $\left(\prod\limits_{i=1}^m x_i^{[-\varepsilon B_{ik}r_k]_+}\right)Z_k^\circ(\hat y_k^\varepsilon)=0$.
  This holds equally well for the \emph{automorphism parts} $\varphi_{\cX,\bfr}^1$, $\varphi_{\cA,\bfr}^1$ and the \emph{tropical parts} $\tau_{\cX,\bfr}^{k,\varepsilon}$, $\tau_{\cA,\bfr}^{k,\varepsilon}$ of the cluster mutations with similar restrictions on the domains of definition.

  Similar considerations can be made when $L_\cX=\CC^n$, $L_\cA=\CC^m$, and $Z_k(u)\in\CC[u]$ is a monic polynomial with $Z_k(0)=1$ in which case analogous restrictions as above must be imposed, even when considering mutations restricted to $L^\times_\cX=\CC_\times^n$ or $L^\times_\cA=\CC_\times^m$.
  In the cluster case (c.f.\ Remark~\ref{rem:cluster case}), mutations for $\bar L^\times_\cX=\RR_{\ge0}^n$ have also been considered \cite{FG16} under the name ``special completion'' for $\cX$-varieties.
  See also \cite{BFMMNC} for similar considerations using cluster charts $L_\cX=\CC^n$.
\end{remark}

The various Poisson maps are summarized in the following commutative diagram.
\begin{equation} 
  \label{eq:PoisCD}
  \xymatrix{
    L_{\cA,\bfr} \ar[d]^{\rho} \ar[r]_{\varphi_{\cA,\bfr}} \ar@/^1pc/[rr]|-{\mu_{\cA,\bfr}} & L_{\cA,\bfr} \ar[d]^{\rho} \ar[r]_{\tau_{\cA,\bfr}} & L_{\cA',\bfr} \ar[d]^{\rho} \\
    L_{\cX,\bfr} \ar[r]^{\varphi_{\cX,\bfr}} \ar@/_1pc/[rr]|-{\mu_{\cX,\bfr}} & L_{\cX,\bfr} \ar[r]^{\tau_{\cX,\bfr}} & L_{\cX',\bfr}
  }
\end{equation}
\bigskip

To record the iteration of mutations, we introduce the $n$-regular labeled rooted tree $\TT_n$ with root vertex $t_0$ and with the $n$ edges emanating from each vertex labeled by the set $\{1,\ldots,n\}$.
In particular, each vertex $t\in\TT_n$ is uniquely determined by a sequence of indices specifying the edge labels along the unique path from $t_0$ to~$t$.

Fix an initial $m\times n$ exchange matrix $\tilde B_{t_0}=(B_{ij;t_0})$ and assign exchange matrices $\tilde B_t=(B_{ij;t})$ to the vertices $t\in\TT_n$ so that $\tilde B_{t'}=\mu_{\bfr,k}\tilde B_t$ whenever $t$ and $t'$ are joined by an edge labeled by $k$.
The collection $\{\tilde B_t\}_{t\in\TT_n}$ is called the \emph{mutation pattern} generated by $\tilde B_{t_0}$ and any two exchange matrices $\tilde B_t$, $\tilde B_{t'}$ for $t,t'\in\TT_n$ are said to be \emph{mutation equivalent}.
Given a skew-symmetric $m\times m$ matrix $\Omega_{t_0}=(\Omega_{ij;t_0})$ which is $D$-compatible with $\tilde B_{t_0}$, we define matrices $\Omega_t=(\Omega_{ij;t})$ compatible with $\tilde B_t$ by iterating mutations as above.

It will be convenient to make particular choices of the signs $\varepsilon$ as we perform sequences of mutations for the cluster charts $L^\times_{\cX,\bfr}$ and $L^\times_{\cA,\bfr}$.
To specify these signs, we observe that the mutation pattern gives rise to the following combinatorial construction.
\begin{definition}
  \label{def:tropical signs}
  Given an $n\times n$ skew-symmetrizable matrix $B$, let $\tilde B_{prin,t_0}$ denote the $2n\times n$ exchange matrix with principal submatrix $B$ and lower $n\times n$ submatrix given by the $n\times n$ identity matrix $I_n$.
  Then, given $\tilde B_{prin,t}$ mutation equivalent to $\tilde B_{prin,t_0}$, the lower $n\times n$ submatrix $C_t:=(C_{ij;t})$ of $\tilde B_{prin,t}$ transforms as follows when $t$ and $t'$ are joined by an edge labeled $k$ (note that either choice of sign $\varepsilon$ below will produce the same result):
  \begin{equation}
    \label{eq:c-matrix mutation 1}
    C_{ij;t'}=
    \begin{cases}
      -C_{ij;t} & \text{if $j=k$;}\\
      C_{ij;t}+[-\varepsilon C_{ik;t} r_k]_+B_{kj;t}+C_{ik;t}[\varepsilon r_kB_{kj;t}]_+ & \text{if $j\ne k$.}
    \end{cases}
  \end{equation}
  It is important to observe that this is a piecewise-linear map on the columns of $C_t$.

  The $C$-matrices admit the following remarkable \emph{sign-coherence} property (c.f.\ \cite{FZ07,NZ12,GHKK14,NR16}): 
  \begin{itemize}
    \item each column of $C_t$, known as a \emph{$\bfc$-vector}, has either all non-negative entries or all non-positive entries.
  \end{itemize}
  Thus each exchange matrix $\tilde B_{prin,t}$ mutation equivalent to $\tilde B_{prin,t_0}$ admits a collection of \emph{tropical signs} $\varepsilon_{\bfr,k;t}\in\{\pm1\}$, where $\varepsilon_{\bfr,k;t}=1$ if the entries in the $k$-th column of $C_t$ are non-negative and $\varepsilon_{\bfr,k;t}=-1$ if the entries in the $k$-th column of $C_t$ are non-positive. 
  Taking the tropical sign in the mutation formula \eqref{eq:c-matrix mutation 1} for $\bfc$-vector the transformation gives the following simplified mutation rule:
  \begin{equation}
    \label{eq:c-matrix mutation 2}
    C_{ij;t'}=
    \begin{cases}
      -C_{ij;t} & \text{if $j=k$;}\\
      C_{ij;t}+C_{ik;t}[\varepsilon_{\bfr,k;t} r_kB_{kj;t}]_+ & \text{if $j\ne k$.}
    \end{cases}
  \end{equation}
  Observe that the mutation rule has become an honestly linear map on the columns of $C_t$ once we choose to mutate according to the tropical sign.

  Given an arbitrary initial exchange matrix $\tilde B_{t_0}$, we may construct a corresponding principalized exchange matrix $\tilde B_{prin,t_0}$ from the principal submatrix $B$ of $\tilde B_{t_0}$.
  Then for each vertex $t\in\TT_n$, we associate the same tropical signs $\varepsilon_{\bfr,k;t}$ as above to the columns of each exchange matrix $\tilde B_t$ mutation equivalent of $\tilde B_{t_0}$.
\end{definition}

We associate a Poisson space $L^\times_{\cX;t}$ isomorphic to either $\RR_\times^n$ or $\CC_\times^n$ and a Poisson space $L^\times_{\cA;t}$ isomorphic to either $\RR_\times^m$ or $\CC_\times^m$ to each vertex $t\in\TT_n$.
That is, there is a system of coordinates $\bfy_t=(y_{1;t},\ldots,y_{n;t})$ on $L^\times_{\cX;t}$ and a system of coordinates $\bfx_t=(x_{1;t},\ldots,x_{m;t})$ on $L^\times_{\cA;t}$ satisfying
\begin{equation*}
  \{y_{k;t},y_{\ell;t}\}_\cX=d_kB_{k\ell;t}y_{k;t}y_{\ell;t}\qquad\text{and}\qquad\{x_{i;t},x_{j;t}\}_\cA=\Omega_{ij;t}x_{i;t}x_{j;t}.
\end{equation*}
As above, we also introduce the Poisson spaces $L^\times_{\cX,\bfr;t}$ isomorphic to either $\RR_\times^n\times\prod_{\ell=1}^n \RR^{r_\ell-1}$ or to $\CC_\times^n\times\prod_{\ell=1}^n \CC^{r_\ell-1}$ and $L^\times_{\cA,\bfr;t}$ isomorphic to either $\RR_\times^m\times\prod_{\ell=1}^n \RR^{r_\ell-1}$ or to $\CC_\times^m\times\prod_{\ell=1}^n \CC^{r_\ell-1}$.
Write $z_{\ell,j;t}$ for coordinates on the $L_\bfr$ components and $Z_{\ell;t}(u)$ for the associated exchange polynomials.
\begin{definition}
  \label{def:cluster varieties}
  Fix an initial $m\times n$ exchange matrix $\tilde B_{t_0}$ and an initial skew-symmetric $m\times m$ matrix~$\Omega_{t_0}$ which is $D$-compatible with~$\tilde B_{t_0}$.
  Let $\bfr=(r_1,\ldots,r_n)$ be a sequence of positive integers.
  The \emph{(generalized) cluster varieties} $\cX_\bfr=\cX_\bfr(\tilde B_{t_0},D)$ and $\cA_\bfr=\cA_\bfr(\tilde B_{t_0},\Omega_{t_0})$ are obtained by gluing the cluster charts $L^\times_{\cX,\bfr;t}$, $L^\times_{\cX,\bfr;t'}$ and $L^\times_{\cA,\bfr;t}$, $L^\times_{\cA,\bfr;t'}$ respectively, for $t,t'\in\TT_n$ joined by an edge labeled $k$, along the cluster mutations $\mu_{\cX,\bfr}^{k,\varepsilon}$ and $\mu_{\cA,\bfr}^{k,\varepsilon}$, where $\varepsilon=\varepsilon_{\bfr,k;t}$.
  That is, 
  \begin{equation*}
    \cX_\bfr=\bigcup_{t\in\TT_n} L^\times_{\cX,\bfr;t}\qquad\text{ and }\qquad\cA_\bfr=\bigcup_{t\in\TT_n} L^\times_{\cA,\bfr;t}
  \end{equation*}
  with cluster charts glued by the mutations as above.
\end{definition}
\begin{remark}
  \label{rem:cluster case 2}
  When $r_\ell=1$ for all $\ell$, the construction above gives rise to the ordinary cluster varieties \cite{FZ02,FG09a}.
  This justifies Remark~\ref{rem:cluster case}.
\end{remark}

For both the cluster $\cA$-varieties \cite{GSV10} and $\cX$-varieties \cite{FG09c}, a Poisson structure is \emph{compatible} with mutations if the coordinates of each cluster chart are log-canonical (c.f.\ equation \eqref{eq:brackets}).
\begin{theorem}
  \cite{GSV10}
  The Poisson structures on the cluster charts $L^\times_{\cX,\bfr;t}$ and $L^\times_{\cA,\bfr;t}$ determined by the $D$-compatible pairs $(\tilde B_t,\Omega_t)$ glue to give global Poisson structures on the cluster varieties $\cX_\bfr$ and $\cA_\bfr$ which are log-canonical over each cluster chart.
  Moreover, the Poisson morphisms $\rho_t:L^\times_{\cA,\bfr;t}\to L^\times_{\cX,\bfr;t}$ glue to give a Poisson morphism $\rho:\cA_\bfr\to\cX_\bfr$.
\end{theorem}

The following result is the famous Laurent phenomenon for (generalized) cluster algebras.
\begin{theorem}
  \cite{FZ02,CS14}
  Each cluster coordinate $x_{i;t}$ on the cluster chart $L^\times_{\cA,\bfr;t}$ determines a global function on $\cA_\bfr$ which is given by a Laurent polynomial in the coordinates of any other cluster chart $L^\times_{\cA,\bfr;t'}$.
\end{theorem}

This result can be made more precise as follows, c.f.\ \cite{FZ07,Nak15,NR16}.
Given a pair of vertices $t,t'\in\TT_n$, the composition of mutations from $L^\times_{-,\bfr;t}$ to $L^\times_{-,\bfr;t'}$ along the unique path from $t$ to $t'$ in $\TT_n$ (always taken with respect to the tropical sign from Definition~\ref{def:tropical signs}) gives rise to birational maps
\[\mu_{\cX,\bfr}^{t',t}:L^\times_{\cX,\bfr;t}\to L^\times_{\cX,\bfr;t'}\qquad\text{and}\qquad\mu_{\cA,\bfr}^{t',t}:L^\times_{\cA,\bfr;t}\to L^\times_{\cA,\bfr;t'}.\]
To understand these iterated mutations, we introduce the following combinatorial constructions.

The \emph{$\bfg$-vector} of $x_{j;t}$ is the integer vector $(G_{ij;t})_{i=1}^m\in\ZZ^m$ defined recursively as follows.
We begin with the identity matrix $G_{ij;t_0}=\delta_{ij}$.
Then for vertices $t,t'\in\TT_n$ which are joined by an edge labeled $k$, we compute using the following recursion (recall that we mutate according to the tropical sign $\varepsilon_{\bfr,k;t}$):
\begin{equation}
  \label{eq:g-matrix mutation}
  G_{ij;t'}=
  \begin{cases}
    -G_{ik;t}+\sum\limits_{\ell=1}^m G_{i\ell;t}[-\varepsilon_{\bfr,k;t} B_{\ell k;t} r_k]_+ & \text{if $j=k$;}\\
    G_{ij;t} & \text{if $j\ne k$.}
  \end{cases}
\end{equation}
\begin{remark}
  Since we only perform mutations in directions $k$ with $1\le k\le n$, the definitions immediately imply $G_{ij;t}=\delta_{ij}$ for $n+1\le j\le m$.
\end{remark}
Next we introduce the \emph{$F$-polynomials} $F_{\bfr,j;t}\in\CC[u_1,\ldots,u_n]$ of the cluster variables $x_{j;t}$.
These are defined recursively by $F_{\bfr,j;t_0}=1$ for $1\le j\le m$ and via the following recursion when $t$ and $t'$ are joined in $\TT_n$ by an edge labeled $k$:
\begin{equation}
  \label{eq:F-polynomial mutation}
  F_{\bfr,j;t'}=
  \begin{cases}
    F_{\bfr,k;t}^{-1}\cdot \left(\prod_{\ell=1}^m F_{\bfr,\ell;t}^{[-\varepsilon_{\bfr,k;t} B_{\ell k;t} r_k]_+}\right) \cdot Z_{k;t}^\circ\left(\prod_{\ell=1}^n u_\ell^{\varepsilon_{\bfr,k;t} C_{\ell k;t}} F_{\bfr,\ell;t}^{\varepsilon_{\bfr,k;t} B_{\ell k;t}}\right) & \text{if $j=k$;}\\
    F_{\bfr,j;t} & \text{if $j\ne k$;}
  \end{cases}
\end{equation}
where $Z_{k;t}$ is the exchange polynomial associated to the $L_\bfr$ component of the chart $L_{\cA,\bfr;t}$ and
\[ Z_{k;t}^\circ(u)=\begin{cases} Z_{k;t}(u) & \text{if $\varepsilon_{\bfr,k;t}=+1$;}\\ Z_{k;t}^*(u) & \text{if $\varepsilon_{\bfr,k;t}=-1$.} \end{cases}\]

\begin{remark}
  Again, since we only perform mutations in directions $k$ with $1\le k\le n$, the definitions immediately imply $F_{\bfr,j;t}=1$ for $n+1\le j\le m$.
\end{remark}
As promised, the $\bfc$-vectors, $\bfg$-vectors, and $F$-polynomials completely control the iterated mutations via the following formulas.
\begin{theorem}
  \label{th:separation}
  \cite{FZ07,Nak15,NR16}
  For any vertex $t\in\TT_n$, we have
  \begin{align}
    \label{eq:separation of additions 1}
    \big(\mu_{\cX,\bfr}^{t,t_0}\big)^*(y_{\ell;t})&=\prod_{k=1}^n \left( y_{k;t_0}^{C_{k\ell;t}} F_{\bfr,k;t}(y_{1;t_0},\ldots,y_{n;t_0})^{B_{k\ell;t}} \right);\\
    \label{eq:separation of additions 2}
    \big(\mu_{\cA,\bfr}^{t,t_0}\big)^*(x_{j;t})&=\left(\prod_{i=1}^m x_{i;t_0}^{G_{ij;t}}\right) F_{\bfr,j;t}(\hat y_{1;t_0},\ldots,\hat y_{n;t_0}),\qquad \hat y_{k;t_0}=\prod_{i=1}^m x_{i;t_0}^{B_{ik;t_0}};\\
    \big(\mu_{\cA,\bfr}^{t,t_0}\big)^*(z_{\ell,j;t})&=z^\circ_{\ell,j;t_0},
  \end{align}
  where $z^\circ_{\ell,j;t_0}=z_{\ell,j;t_0}$ if the mutation sequence from $t_0$ to $t$ has an even number of mutations in direction $\ell$ and $z^\circ_{\ell,j;t_0}=z^*_{\ell,j;t_0}$ if this mutation sequence has an odd number of mutations in direction $\ell$.
\end{theorem}
\begin{remark}
  The separation of additions formula in \cite{FZ07} for the cluster variable $\big(\mu_{\cA,\bfr}^{t,t_0}\big)^*(x_{j;t})$ is usually presented with a denominator computed via a tropical evaluation of the $F$-polynomial $F_{\bfr,j;t}(\hat y_{1;t_0},\ldots,\hat y_{n;t_0})$, this is the auxiliary addition that is usually being ``separated'' from the ordinary addition in the numerator.

  We have bypassed this by using $\bfg$-vectors from $\ZZ^m$ rather than the standard convention with $\bfg$-vectors in $\ZZ^n$.
  It is not hard to see by considering $x_{n+1;t_0}, \ldots, x_{m;t_0}$ simply as frozen cluster variables that our ``separation of additions'' formula \eqref{eq:separation of additions 2} reduces to the standard one in the case of geometric coefficients.
\end{remark}

It is often the case that a sequence of cluster mutations is essentially an identity mapping, this is made precise as follows.
\begin{definition}
  \label{def:periodicity}
  Fix an initial $m\times n$ exchange matrix $\tilde B_{t_0}$ and an initial skew-symmetric $m\times m$ matrix~$\Omega_{t_0}$ which is $D$-compatible with~$\tilde B_{t_0}$.
  Let $\bfr=(r_1,\ldots,r_n)$ be a sequence of positive integers.
  A \emph{periodicity of cluster mutations} $(t_1,\ldots,t_w;\sigma)$ is a sequence of vertices $t_1,t_2,\ldots,t_w\in\TT_n$ with $t_i$ joined to $t_{i+1}$ by an edge labeled $k_i$ together with a permutation $\sigma$ of the set $\{1,\ldots,m\}$ fixing each of $n+1,\ldots,m$ so that the following equalities hold:
  \begin{enumerate}
    \item $\big(\mu_{\cX,\bfr}^{t_w,t_1}\big)^*(y_{\ell;t_w})=y_{\sigma(\ell);t_1}$ for $1\le\ell\le n$;
    \item $\big(\mu_{\cA,\bfr}^{t_w,t_1}\big)^*(x_{j;t_w})=x_{\sigma(j);t_1}$ for $1\le j\le m$;
    \item $\big(\mu_{-,\bfr}^{t_w,t_1}\big)^*(z_{\ell,j;t_w})=z_{\sigma(\ell),j;t_1}$ for $1\le\ell\le n$ and $1\le j\le r_\ell-1$;
    \item $\tilde B_{ij;t_1}=\tilde B_{\sigma(i)\sigma(j);t_w}$ for $1\le i,j\le m$;
    \item $\Omega_{ij;t_1}=\Omega_{\sigma(i)\sigma(j);t_w}$ for $1\le i,j\le m$.
  \end{enumerate}
\end{definition}
In other words, a periodicity of cluster mutations determines a gluing of cluster charts which is (up to a permutation) just an identity mapping.
We will need the following important consequence of the separation of additions formulas and the sign-coherence of $\bfc$-vectors.
\begin{corollary}
  \label{cor:periodicity}
  A periodicity of cluster mutations gives the following equalities:
  \begin{enumerate}
    \item $C_{ij;t_1}=C_{i\sigma(j);t_w}$ for $1\le i,j\le m$;
    \item $G_{ij;t_1}=G_{i\sigma(j);t_w}$ for $1\le i,j\le m$;
    \item $F_{\bfr,j;t_1}=F_{\bfr,\sigma(j);t_w}$ for $1\le j\le m$.
  \end{enumerate}
\end{corollary}
\begin{proof}
  As in \cite[Proposition 5.6]{FZ07}, the sign-coherence of $\bfc$-vectors implies that all $F$-polynomials have constant term 1.
  Since the columns of $\tilde B_{t_0}$ are linearly independent, the Newton polytope of each evaluation $F_{\bfr,j;t}(\hat y_{1;t_0},\ldots,\hat y_{n;t_0})$ has vertices in the positive cone of the $n$-dimensional lattice spanned by the columns of $\tilde B_{t_0}$.
  Moreover, by the first observation, each of these Newton polytopes will contain the vertex of this cone.
  
  These observations immediately give (2) and (3).
  Indeed, the second condition of a periodicity is equivalent to the equality $\big(\mu_{\cA,\bfr}^{t_1,t_0}\big)^*(x_{j;t_1})=\big(\mu_{\cA,\bfr}^{t_w,t_0}\big)^*(x_{\sigma(j);t_w})$ which, together with the ``separation of additions'' formula \eqref{eq:separation of additions 2}, shows that 
  \[\frac{F_{\bfr,j;t_1}(\hat y_{1;t_0},\ldots,\hat y_{n;t_0})}{F_{\bfr,\sigma(j);t_w}(\hat y_{1;t_0},\ldots,\hat y_{n;t_0})}=\prod_{i=1}^m x_{i;t_0}^{G_{i\sigma(j);t_w}-G_{ij;t_1}}\]
  is a single monomial.
  But then the pointedness of their Newton polytopes implies this monomial must be 1, giving (2) and (3).

  Finally, (3) immediately implies (1) by the ``separation of additions'' formula \eqref{eq:separation of additions 1} and condition (1) of a periodicity of cluster mutations.
\end{proof}

\subsection{Mutation of Groupoid Charts}
\label{sec:groupoid mutations}

We have presented thus far the Hamiltonian viewpoint of mutations for cluster charts.
Our goal now is to lift these results to the level of the symplectic groupoids $\cG_\cX$, $\cB_\cX$, $\cD_\cX$ and $\cG_\cA$, $\cB_\cA$, $\cD_\cA$ integrating the log-canonical Poisson structures \eqref{eq:brackets} on $L^\times_\cX\cong\RR_\times^n$ and $L^\times_\cA\cong\RR_\times^m$ respectively as in Section~\ref{sec:local}.
Here is an outline for this section:
\begin{enumerate}
  \item We lift the Poisson ensemble map $\rho:L_\cA\to L_\cX$ in \eqref{eq:ensemble} to symplectic groupoid comorphisms.
    When the exchange matrix is square (and hence invertible), the Poisson ensemble map is a diffeomorphism from its domain onto its image and so the graph of each groupoid comorphism also defines groupoid morphisms over suitable loci.
  \item We lift the Poisson maps $\varphi_{\cX,\bfr}^t$ and $\varphi_{\cA,\bfr}^t$ in Corollary~\ref{cor:time-one flows} to symplectic groupoid morphisms using their Hamiltonian nature.
    We also show that the Lie functor maps the graphs of these groupoid morphisms to the graphs of the algebroid comorphisms naturally induced by the Poisson maps $\varphi_{\cX,\bfr}^t$ and $\varphi_{\cA,\bfr}^t$.
  \item We lift the Poisson maps $\tau_{\cX,\bfr}^{k,\varepsilon}$ and $\tau_{\cA,\bfr}^{k,\varepsilon}$ in Lemma~\ref{le:tropical cluster transformations} to groupoid comorphisms and show that they can also be interpreted as groupoid morphisms.
  \item Combining the last two steps above, we obtain the mutation rule for groupoid charts analogous to the (generalized) cluster mutations in Corollary~\ref{cor:cluster mutation}.
    Through analogues of the separation of additions formulas \eqref{eq:separation of additions 1} and \eqref{eq:separation of additions 2} (c.f.\ Theorem~\ref{th:groupoid separation of additions}), we show that any periodicity of cluster mutations lifts to a periodicity of groupoid mutations (assuming groupoid mutations are always performed according to the tropical signs of the corrseponding $\bfc$-vectors).
    In particular, this shows that groupoid charts indeed glue to give symplectic groupoids over the cluster $\cA$- and $\cX$-varieties.
\end{enumerate}

As in Section~\ref{sec:cluster mutations}, to motivate the Hamiltonian natures of the mutations, we begin working with the following groupoids over individual cluster charts with real coordinates: 
\begin{align*}
  \cG_\cX&\cong\RR^n\times L_\cX,& \cB_\cX&\cong\RR^n\times L_\cX,& \cD_\cX&\cong\RR_\times^n\times L_\cX,\\
  \cG_\cA&\cong\RR^m\times L_\cA,& \cB_\cA&\cong\RR^m\times L_\cA,& \cD_\cA&\cong\RR_\times^m\times L_\cA.
\end{align*}
We introduce coordinates on these groupoids given as follows:
\begin{itemize}
  \item $\cG_\cX$ has coordinates $(\bfq,\bfy)=(q_1,\ldots,q_n,y_1,\ldots,y_n)$; 
  \item $\cB_\cX$ has coordinates $(\bfv,\bfy)=(v_1,\ldots,v_n,y_1,\ldots,y_n)$;
  \item $\cD_\cX$ has coordinates $(\bft,\bfy)=(t_1,\ldots,t_n,y_1,\ldots,y_n)$; 
  \item $\cG_\cA$ has coordinates $(\bfp,\bfx)=(p_1,\ldots,p_m,x_1,\ldots,x_m)$; 
  \item $\cB_\cA$ has coordinates $(\bfu,\bfx)=(u_1,\ldots,u_m,x_1,\ldots,x_m)$;
  \item $\cD_\cA$ has coordinates $(\bfs,\bfx)=(s_1,\ldots,s_m,x_1,\ldots,x_m)$.
\end{itemize}
The symplectic groupoid structure on $\cG_\cA \rra L_\cA$ is the same as in Theorem~\ref{thm:PoiSpLogC}, while the symplectic groupoid structure on $\cD_\cA \rra L_\cA$ is the same as in Theorem~\ref{thm:SymDBLoc}, and the symplectic groupoid structure on $\cB_\cA \rra L_\cA$ is the same as in Theorem~\ref{th:blowup groupoid}.
The symplectic groupoid structures over $L_\cX$ mimic the structures in the theorems above by replacing $x_i$, $x_j$ with $y_k$, $y_\ell$ and $\Omega_{ij}$ with $d_kB_{k\ell}$ together with the corresponding replacements of coordinates $p_j$ with $q_\ell$, $u_j$ with $v_\ell$, and $s_j$ with $t_\ell$.
For simplicity of notation, we will write $\alpha$ and $\beta$ for the source and target maps of all groupoids; this slight abuse of notation should not lead to any confusion.

As above, for a sequence $\bfr=(r_1,\ldots,r_n)$ of positive integers, we introduce the symplectic groupoids 
\begin{align*}
  \cG_{\cX,\bfr}&=\cG_\cX\times T^*L_\bfr,& \cB_{\cX,\bfr}&=\cB_\cX\times T^*L_\bfr,& \cD_{\cX,\bfr}&=\cD_\cX\times T^*L_\bfr,\\
  \cG_{\cA,\bfr}&=\cG_\cA\times T^*L_\bfr,& \cB_{\cA,\bfr}&=\cB_\cA\times T^*L_\bfr,& \cD_{\cA,\bfr}&=\cD_\cA\times T^*L_\bfr,
\end{align*}
where each cotangent bundle $T^*L_\bfr$ is equipped with its canonical symplectic structure and we write $a_{\ell,j}$ for the cotangent coordinate associated to the coordinate $z_{\ell,j}$ of $L_\bfr$.
Translating this symplectic structure into a Poisson bracket on $T^*L_\bfr$, we have
\begin{equation*}
  \{z_{k,i},z_{\ell,j}\}_\bfr=0,\qquad\{z_{k,i},a_{\ell,j}\}_\bfr=\delta_{k\ell}\delta_{ij},\qquad\{a_{k,i},a_{\ell,j}\}_\bfr=0.
\end{equation*}
We extend the star-involution on $L_\bfr$ to $T^*L_\bfr$ via $a_{\ell,j}\mapsto a^*_{\ell,j}:=a_{\ell,r_\ell-j}=a_{\ell,j^*}$.

Define the groupoids 
\begin{align*}
  \cG^\times_{\cX,\bfr}&\cong\RR^n\times L^\times_\cX\times T^*L^\times_\bfr,& \cB^\times_{\cX,\bfr}&\cong\RR^n\times L^\times_\cX\times T^*L^\times_\bfr,& \cD^\times_{\cX,\bfr}&\cong\RR_\times^n\times L^\times_\cX\times T^*L^\times_\bfr,\\
  \cG^\times_{\cA,\bfr}&\cong\RR^m\times L^\times_\cA\times T^*L^\times_\bfr,& \cB^\times_{\cA,\bfr}&\cong\RR^m\times L^\times_\cA\times T^*L^\times_\bfr,& \cD^\times_{\cA,\bfr}&\cong\RR_\times^m\times L^\times_\cA\times T^*L^\times_\bfr.
\end{align*}
as the restrictions of the groupoids above to appropriate orthants in $L_{\cX,\bfr}$ and $L_{\cA,\bfr}$.
\begin{proposition} 
  \label{prop:ensemblegpd}
  The Poisson ensemble map $\rho:L_\cA\to L_\cX$ \eqref{eq:ensemble} lifts to the following comorphisms of symplectic groupoids:
  \begin{enumerate}
    \item $(\rho, P)$ from $\cG_\cA$ to $\cG_\cX$ given on fiber coordinates by
      \begin{equation} \label{eq:ensembleG}
	P^*\!\left(p_i\right) = \frac{1}{x_i}\sum_{k=1}^n B_{ik} \hat{y}_k q_k;
      \end{equation}
    \item $(\rho, P)$ from $\cB_\cA$ to $\cB_\cX$ given on fiber coordinates by
      \[P^*\!\left(u_i\right) = \frac{\prod_{k=1}^n (v_k\hat{y}_k+1)^{B_{ik}} - 1}{x_i};\]
    \item $(\rho, P)$ from $\cD_\cA$ to $\cD_\cX$ given on fiber coordinates by
      \[P^*\!\left(s_i \right) = \prod_{k=1}^n t_k^{B_{ik}}.\]
  \end{enumerate}
\end{proposition}
\begin{proof}
  We prove \eqref{eq:ensembleG} over the positive orthants, the formulas for $\cB$ and $\cD$ follow easily using the diagram \eqref{eq:LGpdCD}.
  Taking the derivative of $\rho:L^\times_\cA\to L^\times_\cX$, whose graph is given by the equation
  \[y_k = \prod_{i=1}^m x_i^{B_{ik}},\]
  we obtain the pullback map $\rho^*: \Omega^1(L^\times_\cX) \to \Omega^1(L^\times_\cA)$ given by
  \begin{align*}
    & dy_k = \sum_{i=1}^m \frac{B_{ik}}{x_i} \hat{y}_k dx_i, \qquad k = 1, \ldots, n, \\
    & \sum_{k=1}^n q_k dy_k = \sum_{i=1}^m \frac{1}{x_i} \sum_{k=1}^n B_{ik} \hat{y}_k q_k dx_i.
  \end{align*}
  We may represent this as the bundle map $\rho^*:\rho^! T^*_{\pi_\cX} L^\times_\cX \to T^*_{\pi_\cA} L^\times_\cA$ whose graph is given by the equation
  \[p_i = \frac{1}{x_i} \sum_{k=1}^n B_{ik} \hat{y}_k q_k, \qquad i = 1, \ldots, m.\]
  The graph of the Lie algebroid comorphism $\rho^*$ from $T^*_{\pi_\cA} L^\times_\cA$ to $T^*_{\pi_\cX} L^\times_\cX$ is defined by
  \[x_i p_i =  \sum_{k=1}^n B_{ik} y_k q_k, \qquad i = 1, \ldots, m.\]
  By Corollary~\ref{cor:exp}, both $x_i p_i$ and $y_k q_k$ are invariant under the exponential map to the Lie groupoids.
  But Corollary~\ref{cor:exp} shows that the cotangent Lie algebroids $T^*_{\pi_\cA} L^\times_\cA$ and $T^*_{\pi_\cX} L^\times_\cX$ are diffeomorphic respectively to the groupoids $\cG_\cA$ and $\cG_\cX$, so the associated Lie groupoid comorphism is given by the equivalent formula \eqref{eq:ensembleG}.
\end{proof}

\begin{remark}
  \label{rmk:source-connected lift}
  When $m=n$, i.e.\ the map $\rho:L_\cA\to L_\cX$ is a birational isomorphism, the Poisson ensemble map actually lifts to true morphisms of symplectic groupoids $\rho:\cG_\cA\to\cG_\cX$, $\rho:\cB_\cA\to\cB_\cX$, and $\rho:\cD_\cA\to\cD_\cX$ given respectively on fiber coordinates by 
  \begin{align*}
    \rho^*(q_\ell)&=-(d_\ell\hat y_\ell)^{-1}\sum\limits_{j=1}^m\Omega_{j\ell}x_jp_j;\\
    \rho^*(v_\ell)&=\prod\limits_{j=1}^m (u_j x_j + 1)^{-d_\ell^{-1}\Omega_{\ell j}};\\
    \rho^*(t_\ell)&=\prod\limits_{j=1}^m s_j^{-d_\ell^{-1}\Omega_{\ell j}}.
  \end{align*}
  Observe that care must me taken when considering the fractional exponents above, however the maps are in fact well-defined.
  Indeed, a morphism of groupoids must take the identity image to the identity image and thus there is a unique choice of branch to take, in other words the maps $\rho:\cB_\cA\to\cB_\cX$ and $\rho:\cD_\cA\to\cD_\cX$ are actually well-defined covering maps on the fibers.

  The choice of coordinates used in \cite{FG09c} can be identified with our coordinates $t_\ell^{d_\ell}$, one may motivate such a choice by the desire to avoid the fractional exponents above, however as observed such considerations are unnecessary.
\end{remark}

Next, we lift the Hamiltonian Poisson maps $\varphi_{\cX,\bfr}^t$ and $\varphi_{\cA,\bfr}^t$ in Corollary~\ref{cor:time-one flows}.
Using appropriate source and target maps, define functions $H_{\cX,\bfr}^{k,\varepsilon}:=\alpha^*(h_{\cX,\bfr}^{k,\varepsilon})-\beta^*(h_{\cX,\bfr}^{k,\varepsilon})$ on the groupoids $\cG^\times_{\cX,\bfr}$, $\cB^\times_{\cX,\bfr}$, $\cD^\times_{\cX,\bfr}$ and functions $H_{\cA,\bfr}^{k,\varepsilon}:=\alpha^*(h_{\cA,\bfr}^{k,\varepsilon})-\beta^*(h_{\cA,\bfr}^{k,\varepsilon})$ on the groupoids $\cG^\times_{\cA,\bfr}$, $\cB^\times_{\cA,\bfr}$, $\cD^\times_{\cA,\bfr}$.
More explicitly, these Hamiltonian functions are given by the following analogues of \eqref{eq:hamiltonians} using appropriate source and target maps:
\begin{align}
  \label{eq:X hamiltonian}
  H_{\cX,\bfr}^{k,\varepsilon}&:=\frac{\varepsilon}{d_k}\int_{\alpha^*(y_k^\varepsilon)}^{\beta^*(y_k^\varepsilon)} \frac{\log\big(Z_k^\circ(u)\big)}{u}du;\\
  \label{eq:A hamiltonian}
  H_{\cA,\bfr}^{k,\varepsilon}&:=\frac{\varepsilon}{d_k}\int_{\alpha^*(\hat y_k^\varepsilon)}^{\beta^*(\hat y_k^\varepsilon)} \frac{\log\big(Z_k^\circ(u)\big)}{u}du;
\end{align}
where
\[ Z_k^\circ(u)=\begin{cases} Z_k(u) & \text{if $\varepsilon=+1$;}\\ Z_k^*(u) & \text{if $\varepsilon=-1$.} \end{cases}\]
Note that, although we use the same symbols for the Hamiltonian functions on each groupoid, it will be clear from context which source and target map should be used in their definition.

Write $X_{\cX,\bfr}^{k,\varepsilon}$ and $X_{\cA,\bfr}^{k,\varepsilon}$ for the Hamiltonian vector fields associated to the functions $H_{\cX,\bfr}^{k,\varepsilon}$ and $H_{\cA,\bfr}^{k,\varepsilon}$ on the groupoids $\cG^\times_{\cX,\bfr}$, $\cB^\times_{\cX,\bfr}$, $\cD^\times_{\cX,\bfr}$ and $\cG^\times_{\cA,\bfr}$, $\cB^\times_{\cA,\bfr}$, $\cD^\times_{\cA,\bfr}$ respectively, i.e.\ they are the vector fields naturally associated to the derivations $\{H_{\cX,\bfr}^{k,\varepsilon},\cdot\}$ and $\{H_{\cA,\bfr}^{k,\varepsilon},\cdot\}$ on the appropriate groupoids.
\begin{lemma}
  For $1\le k\le n$ and $\varepsilon\in\{\pm1\}$, the Hamiltonian vector fields $X_{\cX,\bfr}^{k,\varepsilon}$ determine the following dynamics on the symplectic groupoids $\cG^\times_{\cX,\bfr}$, $\cB^\times_{\cX,\bfr}$, $\cD^\times_{\cX,\bfr}$ respectively:
  \begin{align*}
    \dot q_\ell&=\{H_{\cX,\bfr}^{k,\varepsilon},q_\ell\}_{\cG_\cX}=-\frac{\delta_{k\ell}}{d_ky_k}\log\left(\frac{Z_k^\circ\left(y_k^\varepsilon e^{\varepsilon\sum_{\ell'=1}^n d_{\ell'} B_{\ell' k}q_{\ell'} y_{\ell'}}\right)}{Z_k^\circ(y_k^\varepsilon)}\right)+B_{k\ell}\log\big(Z_k^\circ(y_k^\varepsilon)\big)q_\ell;\\
    \dot v_\ell&=\{H_{\cX,\bfr}^{k,\varepsilon},v_\ell\}_{\cB_\cX}=-\frac{\delta_{k\ell}}{d_ky_k}\log\left(\frac{Z_k^\circ\left(y_k^\varepsilon \prod_{\ell'=1}^n (1+v_{\ell'}y_{\ell'})^{\varepsilon d_{\ell'} B_{\ell' k}}\right)}{Z_k^\circ(y_k^\varepsilon)}\right)(v_ky_k+1)+B_{k\ell}\log\big(Z_k^\circ(y_k^\varepsilon)\big)v_\ell;\\
    \dot t_\ell&=\{H_{\cX,\bfr}^{k,\varepsilon},t_\ell\}_{\cD_\cX}=-\frac{\delta_{k\ell}}{d_k}\log\left(\frac{Z_k^\circ\left(y_k^\varepsilon \prod_{\ell'=1}^nt_{\ell'}^{\varepsilon d_{\ell'} B_{\ell' k}}\right)}{Z_k^\circ(y_k^\varepsilon)}\right)t_k;
  \end{align*}
  and the Hamiltonian vector fields $X_{\cA,\bfr}^{k,\varepsilon}$ determine the following dynamics on the symplectic groupoids $\cG^\times_{\cA,\bfr}$, $\cB^\times_{\cA,\bfr}$, $\cD^\times_{\cA,\bfr}$ respectively:
  \begin{align}
    \label{eq:p dot}
    \dot p_j&=\{H_{\cA,\bfr}^{k,\varepsilon},p_j\}_{\cG_\cA}=-\frac{B_{jk}}{d_kx_j}\log\left(\frac{Z_k^\circ\left(\hat y_k^\varepsilon e^{-\varepsilon d_kp_kx_k}\right)}{Z_k^\circ(\hat y_k^\varepsilon)}\right)+\delta_{jk}\log\big(Z_k^\circ(\hat y_k^\varepsilon)\big)p_k;\\
    \nonumber
    \dot u_j&=\{H_{\cA,\bfr}^{k,\varepsilon},u_j\}_{\cB_\cA}=-\frac{B_{jk}}{d_kx_j}\log\left(\frac{Z_k^\circ\left(\hat y_k^\varepsilon (1+u_kx_k)^{-\varepsilon d_k}\right)}{Z_k^\circ(\hat y_k^\varepsilon)}\right)(u_jx_j+1)+\delta_{jk}\log\big(Z_k^\circ(\hat y_k^\varepsilon)\big)u_k\\
    \nonumber
    \dot s_j&=\{H_{\cA,\bfr}^{k,\varepsilon},s_j\}_{\cD_\cA}=-\frac{B_{jk}}{d_k}\log\left(\frac{Z_k^\circ\left(\hat y_k^\varepsilon s_k^{-\varepsilon d_k}\right)}{Z_k^\circ(\hat y_k^\varepsilon)}\right)s_j.
  \end{align}
  The coordinates $a_{\ell,j}$ on $T^*L^\times_\bfr$ evolve according to
  \begin{equation}
    \label{eq:a dot}
    \dot a_{\ell,j}=\{H_{-,\bfr}^{k,\varepsilon},a_{\ell,j}\}_-=\delta_{k\ell}\frac{\varepsilon}{d_k}\int_*^* \frac{u^{\varepsilon j-1}}{Z_k(u^\varepsilon)}du, 
  \end{equation}
  where the bounds of integration are respectively those from equations \eqref{eq:X hamiltonian} or \eqref{eq:A hamiltonian}.
\end{lemma}
\begin{proof}
  We prove equation \eqref{eq:p dot}, leaving the other groupoid coordinates as an exercise for the reader.
  By Theorem~\ref{thm:PoiSpLogC}, we have $\{x_i,p_j\}_{\cG_\cA}=-\delta_{ij}-\Omega_{ij}x_ip_j$ and $\{p_i,p_j\}_{\cG_\cA}=\Omega_{ij}p_ip_j$ so that
  \[\{\hat y_k,p_j\}_{\cG_\cA}=\sum_{i=1}^m B_{ik}\frac{\hat y_k}{x_i}\{x_i,p_j\}_{\cG_\cA}=-B_{jk}\frac{\hat y_k}{x_j}-\delta_{jk}d_k\hat y_kp_k\]
  and
  \[\{e^{-d_kp_kx_k},p_j\}_{\cG_\cA}=\delta_{jk}d_kp_ke^{-d_kp_kx_k}.\]
  Now equation~\eqref{eq:p dot} follows by applying the chain rule for the derivation $\{\cdot,p_j\}_{\cG_\cA}$.

  To see equation \eqref{eq:a dot}, we observe that $\frac{d\log(Z_k(u))}{dz_{\ell,j}}=\delta_{k\ell}\frac{u^j}{Z_k(u)}$ and $\frac{d\log(Z_k^*(u))}{dz_{\ell,j}}=\delta_{k\ell}\frac{u^{r_k-j}}{Z_k^*(u)}=\delta_{k\ell}\frac{u^{-j}}{Z_k(u^{-1})}$.
\end{proof}

Next we consider the time-$t$ flows, all denoted $\varphi_{\cX,\bfr}^t$, of the vector fields $X_{\cX,\bfr}^{k,\varepsilon}$ on the groupoids $\cG^\times_{\cX,\bfr}$, $\cB^\times_{\cX,\bfr}$, $\cD^\times_{\cX,\bfr}$ and the time-$t$ flows, all denoted $\varphi_{\cA,\bfr}^t$, of the vector fields $X_{\cA,\bfr}^{k,\varepsilon}$ on the groupoids $\cG^\times_{\cA,\bfr}$, $\cB^\times_{\cA,\bfr}$, $\cD^\times_{\cA,\bfr}$.
\begin{corollary}
  \label{cor:groupoid hamiltonian flows}
  For $1\le k\le n$ and $\varepsilon\in\{\pm1\}$, the Hamiltonian flow $\varphi_{\cX,\bfr}^t: L^\times_{\cX,\bfr} \to L^\times_{\cX,\bfr}$ in Corollary~\ref{cor:time-one flows} is lifted to the multiplicative Hamiltonian flows $\varphi_{\cX,\bfr}^t: \cG^\times_{\cX,\bfr} \to \cG^\times_{\cX,\bfr}$, $\varphi_{\cX,\bfr}^t: \cB^\times_{\cX,\bfr} \to \cB^\times_{\cX,\bfr}$, $\varphi_{\cX,\bfr}^t: \cD^\times_{\cX,\bfr} \to \cD^\times_{\cX,\bfr}$, given on coordinates by
  \begin{align*}
    (\varphi_{\cX,\bfr}^t)^*(q_\ell y_\ell)&=q_\ell y_\ell-t\frac{\delta_{k\ell}}{d_k}\log\left(\frac{Z_k^\circ\left(y_k^\varepsilon e^{\varepsilon\sum_{\ell'=1}^n d_{\ell'} B_{\ell' k}q_{\ell'} y_{\ell'}}\right)}{Z_k^\circ(y_k^\varepsilon)}\right);\\
    (\varphi_{\cX,\bfr}^t)^*(v_\ell y_\ell + 1)&=\left(\frac{Z_k^\circ\left(y_k^\varepsilon \prod_{\ell'=1}^n (v_{\ell'} y_{\ell'} + 1)^{\varepsilon d_{\ell'} B_{\ell' k}}\right)}{Z_k^\circ(y_k^\varepsilon)}\right)^{-t\frac{\delta_{k\ell}}{d_k}}(v_\ell y_\ell+1);\\
    (\varphi_{\cX,\bfr}^t)^*(t_\ell)&=\left(\frac{Z_k^\circ\left(y_k^\varepsilon \prod_{\ell'=1}^n t_{\ell'}^{\varepsilon d_{\ell'} B_{\ell' k}}\right)}{Z_k^\circ(y_k^\varepsilon)}\right)^{-t\delta_{k\ell}/d_k}t_\ell;
  \end{align*}
  and the Hamiltonian flow $\varphi_{\cA,\bfr}^t: L^\times_{\cA,\bfr} \to L^\times_{\cA,\bfr}$ in Corollary~\ref{cor:time-one flows} is lifted to the multiplicative Hamiltonian flows $\varphi_{\cA,\bfr}^t: \cG^\times_{\cA,\bfr} \to \cG^\times_{\cA,\bfr}$, $\varphi_{\cA,\bfr}^t: \cB^\times_{\cA,\bfr} \to \cB^\times_{\cA,\bfr}$, $\varphi_{\cA,\bfr}^t: \cD^\times_{\cA,\bfr} \to \cD^\times_{\cA,\bfr}$, given on coordinates by
  \begin{align*}
    (\varphi_{\cA,\bfr}^t)^*(p_jx_j)&=p_jx_j-t\frac{B_{jk}}{d_k}\log\left(\frac{Z_k^\circ\left(\hat y_k^\varepsilon e^{-\varepsilon d_kp_kx_k}\right)}{Z_k^\circ(\hat y_k^\varepsilon)}\right);\\
    (\varphi_{\cA,\bfr}^t)^*(u_jx_j+1)&=\left(\frac{Z_k^\circ\left(\hat y_k^\varepsilon (u_kx_k+1)^{-\varepsilon d_k}\right)}{Z_k^\circ(\hat y_k^\varepsilon)}\right)^{-t\frac{B_{jk}}{d_k}}(u_jx_j+1);\\
    (\varphi_{\cA,\bfr}^t)^*(s_j)&=\left(\frac{Z_k^\circ\left(\hat y_k^\varepsilon s_k^{-\varepsilon d_k}\right)}{Z_k^\circ(\hat y_k^\varepsilon)}\right)^{-tB_{jk}/d_k}s_j.
  \end{align*}
  Each of the Hamiltonian flows above transform the coordinates $a_{\ell,j}$ on $T^*L^\times_\bfr$ according to
  \begin{equation*}
    (\varphi_{-,\bfr}^t)^*(a_{\ell,j})=a_{\ell,j}+t\delta_{k\ell}\frac{\varepsilon}{d_k}\int_*^* \frac{u^{\varepsilon j-1}}{Z_k(u^\varepsilon)}du
  \end{equation*}
  where the bounds of integration are respectively those from equations \eqref{eq:X hamiltonian} or \eqref{eq:A hamiltonian}.
\end{corollary}
\begin{proof}
  To see this, it suffices to make the following observations:
  \begin{itemize}
    \item The quantities $y_k$, $e^{\varepsilon\sum_{\ell'=1}^n d_{\ell'} B_{\ell' k}q_{\ell'} y_{\ell'}}$, $\prod_{\ell'=1}^n (v_{\ell'}y_{\ell'}+1)^{\varepsilon d_{\ell'} B_{\ell' k}}$, and $\prod_{\ell'=1}^nt_{\ell'}^{\varepsilon d_{\ell'} B_{\ell' k}}$ are conserved under the flow of the vector fields $X_{\cX,\bfr}^{k,\varepsilon}$.
      It follows that $\frac{d}{dt}(q_\ell y_\ell)$ is a constant, $\frac{d}{dt}(v_\ell y_\ell+1)$ is a constant multiple of $v_\ell y_\ell+1$, and $\frac{d}{dt}(t_\ell)$ is a constant multiple of $t_\ell$.
    \item The quantities $\hat y_k$, $e^{-\varepsilon d_kp_kx_k}$, $(u_k x_k + 1)^{-\varepsilon d_k}$, and $s_k^{-\varepsilon d_k}$ are conserved under the flow of the vector fields $X_{\cA,\bfr}^{k,\varepsilon}\in\cT_{\cG_{\cA,\bfr}}$.
      It follows that $\frac{d}{dt}(p_jx_j)$ is a constant, $\frac{d}{dt}(u_j x_j + 1)$ is a constant multiple of $u_j x_j+1$, and $\frac{d}{dt}(s_j)$ is a constant multiple of $s_j$.
    \item As a result, in the flow of any groupoid, $\frac{d}{dt}(a_{\ell,j})$ is constant.
  \end{itemize}
  In each case, the resulting differential equation is easy to solve and we leave as an exercise for the reader to check the formulas given above.
\end{proof}

For the groupoid morphism $\varphi^1_{\cA,\bfr}: \cG^\times_{\cA,\bfr} \to \cG^\times_{\cA,\bfr}$ in Corollary~\ref{cor:groupoid hamiltonian flows}, we give an explicit verification that the corresponding Lie algebroid morphism is the dual bundle map of the Lie algebroid comorphism induced by the base Poisson map $\varphi^1_{\cA,\bfr}:L^\times_{\cA,\bfr}\to L^\times_{\cA,\bfr}$.
A similar story exists for the other groupoid maps in Corollary~\ref{cor:groupoid hamiltonian flows}.

To ease the notation, we drop the time parameter in the superscript, fix $\bfr$ and use Sans-serif font for the coordinates for the range, e.g.\ $\sfx$ for the range versus $x$ for the domain.
To begin with, we apply the Lie functor, see \eqref{eq:Liefunctor} and the ensuing discussion.
The source map of $\cG_\cA \rra L_\cA$ is 
\[
	\alpha: \cG_\cA \to L_\cA, \qquad (\bfp, \bfx) \mapsto \bfx,
\]
so the kernel of $\alpha_*: T\cG_\cA \to TL_\cA$ is generated by $\frac{\partial}{\partial p_j}, \> j = 1, \ldots, m$.
Given the Hamiltonian groupoid morphism
\begin{align*}
	\varphi^1_{\cA,\bfr}: ~& \cG_\cA \to \cG_\cA, \\
	& \sfx_j = \big(Z_k^\circ(\hat y_k^\varepsilon)\big)^{-\delta_{jk}}x_j, \\
	& \sfp_j \sfx_j = p_jx_j-\frac{B_{jk}}{d_k}\log\left(\frac{Z_k^\circ\left(\hat y_k^\varepsilon e^{-\varepsilon d_kp_kx_k}\right)}{Z_k^\circ(\hat y_k^\varepsilon)}\right),
\end{align*}
its pushfoward is
\begin{align*}
	\left(\varphi^1_{\cA,\bfr}\right)_*: ~& \ker(\alpha_*: T\cG_\cA \to TL_\cA) \to \ker(\alpha_*: T\cG_\cA \to TL_\cA),  \\
	& \frac{\partial}{\partial p_j} 
	= \frac{\partial}{\partial \sfp_j} + \frac{\varepsilon x_k \hat{y}_k^\varepsilon B_{jk}}{x_j} \frac{ e^{-\varepsilon d_k p_k x_k} (Z_k^\circ)'\left(\hat y_k^\varepsilon e^{-\varepsilon d_k p_k x_k}\right)}{Z_k^\circ\left(\hat y_k^\varepsilon e^{-\varepsilon d_k p_k x_k}\right)} \frac{\partial}{\partial p_k}, \qquad \text{if}~ j \ne k, \\
	& \frac{\partial}{\partial p_k}
	= \frac{1}{Z_k^\circ(\hat y_k^\varepsilon)}\frac{\partial}{\partial \sfp_k}.
\end{align*}
Restricting to the identity image, which is given by $p_j = 0, \> j = 1, \ldots, m$, we have
\begin{align*}
	\Lie \left(\varphi^1_{\cA,\bfr} \right): ~& \Lie \left(\cG_\cA \right) \to \Lie \left( \cG_\cA \right),  \\
	& \frac{\partial}{\partial p_j}
	= \frac{\partial}{\partial \sfp_j} + \frac{\varepsilon \sfx_k \hat{\sfy}_k^\varepsilon B_{jk}}{\sfx_j} \frac{(Z_k^\circ)'\left(\hat \sfy_k^\varepsilon \right)}{Z_k^\circ\left(\hat \sfy_k^\varepsilon \right)} \frac{\partial}{\partial \sfp_k}, \qquad \text{if}~ j \ne k, \\
	& \frac{\partial}{\partial p_k} 
	= \frac{1}{Z_k^\circ(\hat{\sfy}_k^\varepsilon)}\frac{\partial}{\partial \sfp_k}.
\end{align*}
On the other hand, taking derivatives of $\sfx_j = \big(Z_k^\circ(\hat{y}_k^\varepsilon)\big)^{-\delta_{jk}}x_j$,
we obtain the pullback map
\begin{align*}
	(\varphi^1_{\cA,\bfr})^*: ~& \Omega^1(L_\cA) \to \Omega^1(L_\cA), \\
	& d \sfx_j = dx_j, \qquad \text{if}~j\ne k, \\
	& d \sfx_k = \frac{1}{Z_k^\circ(\hat{y}_k^\varepsilon)} dx_k + \sum\limits_{j\ne k} \frac{\varepsilon x_k \hat{y}_k^\varepsilon B_{jk}}{x_j} \frac{(Z_k^\circ)'(\hat{y}_k^\varepsilon)}{\left(Z_k^\circ(\hat{y}_k^\varepsilon)\right)^2} dx_j \\
	& \quad \; \, = \frac{1}{Z_k^\circ(\hat{\sfy}_k^\varepsilon)} dx_k + \sum\limits_{j\ne k} \frac{\varepsilon \sfx_k \hat{\sfy}_k^\varepsilon B_{jk}}{\sfx_j} \frac{(Z_k^\circ)'(\hat{\sfy}_k^\varepsilon)}{Z_k^\circ(\hat{\sfy}_k^\varepsilon)} dx_j.
\end{align*}
If we identify $\Lie \left(\cG_\cA \right)$ and $T^*_{\pi_\cA} L_\cA$ via
\[
	\frac{\partial}{\partial p_j} = dx_j, \> j =1, \ldots, m,
\]
then $(\varphi^1_{\cA,\bfr})^*$ is the dual bundle map of $\Lie \left(\varphi^1_{\cA,\bfr} \right)$.
That is, $(\varphi^1_{\cA,\bfr})^*: T^*_{\pi_\cA} L_\cA \to T^*_{\pi_\cA} L_\cA$ is a Lie algebroid comorphism and its corresponding Lie algebroid morphism is $\Lie \left(\varphi^1_{\cA,\bfr} \right)$, see the discussion preceding Definition~\ref{def:gpdcomor}.
It follows that $(\varphi^1_{\cA,\bfr})^*$ integrates to a groupoid comorphism whose graph coincides with the graph of $\varphi^1_{\cA,\bfr}: \cG_\cA \to \cG_\cA$.

Write $\cG_{\cX',\bfr}$, $\cB_{\cX',\bfr}$, $\cD_{\cX',\bfr}$ and $\cG_{\cA',\bfr}$, $\cB_{\cA',\bfr}$, $\cD_{\cA',\bfr}$ for the analogous groupoids over the cluster charts $L_{\cX',\bfr}$ and $L_{\cA',\bfr}$ with Poisson structures related to those on $L_{\cX,\bfr}$ and $L_{\cA,\bfr}$ by mutation in direction $k$.
\begin{lemma}
  \label{le:tropical groupoid transformations}
  For $1\le k\le n$ and $\varepsilon\in\{\pm1\}$, there are symplectic groupoid morphisms 
  \begin{align*}
    \tau_{\cX,\bfr}^{k,\varepsilon}&:\cG^\times_{\cX,\bfr}\to\cG^\times_{\cX',\bfr},& \tau_{\cX,\bfr}^{k,\varepsilon}&:\cB^\times_{\cX,\bfr}\to\cB^\times_{\cX',\bfr},& \tau_{\cX,\bfr}^{k,\varepsilon}&:\cD^\times_{\cX,\bfr}\to\cD^\times_{\cX',\bfr}\\
    \tau_{\cA,\bfr}^{k,\varepsilon}&:\cG^\times_{\cA,\bfr}\to\cG^\times_{\cA',\bfr},& \tau_{\cA,\bfr}^{k,\varepsilon}&:\cB^\times_{\cA,\bfr}\to\cB^\times_{\cA',\bfr},& \tau_{\cA,\bfr}^{k,\varepsilon}&:\cD^\times_{\cA,\bfr}\to\cD^\times_{\cA',\bfr}
  \end{align*}
  lifting respectively the Poisson morphisms $\tau_{\cX,\bfr}^{k,\varepsilon}:L^\times_{\cX,\bfr}\to L^\times_{\cX',\bfr}$ and $\tau_{\cA,\bfr}^{k,\varepsilon}:L^\times_{\cA,\bfr}\to L^\times_{\cA',\bfr}$ of Lemma~\ref{le:tropical cluster transformations}.
  These are given on fiber coordinates by
  \begin{align}
    \nonumber
    (\tau_{\cX,\bfr}^{k,\varepsilon})^*(q'_\ell)
    &=\begin{cases} 
      -q_k y_k^2 + \sum\limits_{\ell'=1}^n [\varepsilon r_k B_{k\ell'}]_+ q_{\ell'} y_{\ell'} y_k & \text{if $\ell=k$;}\\ 
      q_\ell y_k^{-[\varepsilon r_k B_{k\ell}]_+} & \text{if $\ell\ne k$;}
    \end{cases}\\
    \nonumber
    (\tau_{\cX,\bfr}^{k,\varepsilon})^*(v'_\ell)
    &=\begin{cases} 
      y_k\Big[(v_k y_k + 1)^{-1}\prod\limits_{\ell'=1}^n (v_{\ell'} y_{\ell'} + 1)^{[\varepsilon r_k B_{k\ell'}]_+} -1\Big] & \text{if $\ell=k$;}\\
      v_\ell y_k^{-[\varepsilon r_k B_{k\ell}]_+} & \text{if $\ell\ne k$;}
    \end{cases}\\
    \nonumber
    (\tau_{\cX,\bfr}^{k,\varepsilon})^*(t'_\ell)
    &=\begin{cases} 
      t_k^{-1}\prod\limits_{\ell'=1}^n t_{\ell'}^{[\varepsilon r_k B_{k\ell'}]_+} & \text{if $\ell=k$;}\\
      t_\ell & \text{if $\ell\ne k$;}
    \end{cases}\\
    \label{eq:tropical GA transformation}
    (\tau_{\cA,\bfr}^{k,\varepsilon})^*(p'_j)&=
    \begin{cases} 
      -p_k x_k^2 \prod\limits_{i=1}^m x_i^{-[-\varepsilon B_{ik} r_k]_+} & \text{if $j=k$;}\\ 
      p_j + [-\varepsilon B_{jk} r_k]_+ \frac{p_k x_k}{x_j} & \text{if $j\ne k$;}
    \end{cases}\\
    \nonumber
    (\tau_{\cA,\bfr}^{k,\varepsilon})^*(u'_j)
    &=\begin{cases} 
      x_k \left[ (u_k x_k +1)^{-1} -1\right] \left(\prod_{i=1}^m x_i^{-[-\varepsilon B_{ik} r_k]_+}\right) & \text{if $j=k$;}\\ 
      x_j^{-1}\big[(u_j x_j + 1) (u_k x_k + 1)^{[-\varepsilon B_{jk} r_k]_+}-1\big] & \text{if $j\ne k$.}
    \end{cases}\\
    \nonumber
    (\tau_{\cA,\bfr}^{k,\varepsilon})^*(s'_j)
    &=\begin{cases} 
      s_k^{-1} & \text{if $j=k$;}\\ 
      s_j s_k^{[-\varepsilon B_{jk} r_k]_+} & \text{if $j\ne k$;}
    \end{cases}\\
    \nonumber
    (\tau_{-,\bfr}^{k,\varepsilon})^*(a'_{\ell,j})
    &=
    \begin{cases}
      a^*_{k,j} & \text{if $\ell=k$;}\\
      a_{\ell,j} & \text{if $\ell\ne k$.}
    \end{cases}
  \end{align}
\end{lemma}
\begin{proof}
  We prove \eqref{eq:tropical GA transformation}.
  The other groupoid maps follows from \eqref{eq:LGpdCD} and/or Proposition~\ref{prop:ensemblegpd}.
  To ease the notation, we fix $\bfr$ and drop all subscripts and supscripts.
  Taking derivatives of
  \begin{align*}
    \tau:~ & L^\times_{\cA}\to L^\times_{\cA'}, \\
      & x'_k =x_k^{-1}\prod\limits_{i=1}^m x_i^{[-\varepsilon B_{ik}r_k]_+}, \\
      & x'_j = x_j, \qquad \text{if } j \ne k,
  \end{align*}
  we obtain the pullback map
  \begin{align*}
    \tau^*:~ & \Omega^1(L^\times_{\cA'}) \to \Omega^1(L^\times_{\cA}), \\
      & dx'_k = - x_k^{-2}\prod\limits_{i=1}^m x_i^{[- \varepsilon B_{jk} r_k]_+} dx_k + \sum\limits_{j=1}^m \frac{[- \varepsilon B_{jk} r_k]_+}{x_j x_k} \prod\limits_{i=1}^m x_i^{[- \varepsilon B_{ik} r_k]_+} dx_j, \\
      & dx'_j = dx_j \qquad \text{if}~j\ne k, \\
      & \sum\limits_{j=1}^m p'_j dx'_j = - \frac{p'_k}{x_k^2}\prod\limits_{i=1}^m x_i^{[- \varepsilon B_{ik} r_k]_+} dx_k + \sum\limits_{j\ne k} \left(\frac{p'_k [- \varepsilon B_{jk} r_k]_+ }{x_jx_k} \prod\limits_{i=1}^m x_i^{[- \varepsilon B_{ik} r_k]_+}  + p'_j\right)dx_j
  \end{align*}
  or as a bundle map
  \begin{align*}
    \tau^*:~ & \tau^! T^* L^\times_{\cA'} \to T^* L^\times_{\cA},  \\
      & p_k = - \frac{x'_k p'_k}{x_k}, \\
      & p_j = \frac{x'_kp'_k [- \varepsilon B_{jk} r_k]_+ }{x_j}   + p'_j \qquad \text{if}~j\ne k.
  \end{align*}
  The graph of the Lie algebroid comorphism $\tau^*$ from $T^*_{\pi_{\cA}} L_{\cA}$ to $T^*_{\pi_{\cA'}} L_{\cA'}$ is defined by
  \begin{align*}
    & x'_k p'_k = - x_k p_k , \\
    & x'_j p'_j = x_kp_k [- \varepsilon B_{jk} r_k]_+  + x_jp_j \quad \text{if}~j\ne k.
  \end{align*}
  By Lemma~\ref{lemma:PoisSp} and Corollary~\ref{cor:exp}, both $x_j p_j$ and $x'_j p'_j$ are invariant under the exponential map, so \eqref{eq:tropical GA transformation} follows.
\end{proof}

\begin{theorem}
  \label{th:groupoid mutation}
  There are morphisms of Poisson groupoids $\mu_{\cX,\bfr}^{k,\varepsilon}$ and $\mu_{\cA,\bfr}^{k,\varepsilon}$:
  \begin{align*}
    \mu_{\cX,\bfr}^{k,\varepsilon}&:\cG^\times_{\cX,\bfr}\to\cG^\times_{\cX',\bfr},& \mu_{\cX,\bfr}^{k,\varepsilon}&:\cB^\times_{\cX,\bfr}\to\cB^\times_{\cX',\bfr},& \mu_{\cX,\bfr}^{k,\varepsilon}&:\cD^\times_{\cX,\bfr}\to\cD^\times_{\cX',\bfr}\\
    \mu_{\cA,\bfr}^{k,\varepsilon}&:\cG^\times_{\cA,\bfr}\to\cG^\times_{\cA',\bfr},& \mu_{\cA,\bfr}^{k,\varepsilon}&:\cB^\times_{\cA,\bfr}\to\cB^\times_{\cA',\bfr},& \mu_{\cA,\bfr}^{k,\varepsilon}&:\cD^\times_{\cA,\bfr}\to\cD^\times_{\cA',\bfr}
  \end{align*}
  defined as the compositions $\tau_{-,\bfr}^{k,\varepsilon}\circ\varphi^1_{-,\bfr}$ given on fiber coordinates by
  \begin{align}
    \nonumber
    (\mu_{\cX,\bfr}^{k,\varepsilon})^*(q'_\ell)
    &=\begin{cases} 
      -q_k y_k^2 + \sum\limits_{\ell'=1}^n [\varepsilon r_k B_{k\ell'}]_+ q_{\ell'} y_{\ell'} y_k + \frac{y_k}{d_k}\log\left(\frac{Z_k^\circ\left(y_k^\varepsilon e^{\varepsilon\sum_{\ell'=1}^n d_{\ell'} B_{\ell' k}q_{\ell'} y_{\ell'}}\right)}{Z_k^\circ(y_k^\varepsilon)}\right) & \text{if $\ell=k$;}\\ 
      q_\ell y_k^{-[\varepsilon r_k B_{k\ell}]_+} Z_k^\circ(y_k^\varepsilon)^{B_{k\ell}} & \text{if $\ell\ne k$;}
    \end{cases}\\
    \nonumber
    (\mu_{\cX,\bfr}^{k,\varepsilon})^*(v'_\ell)
    &=\begin{cases} 
      y_k \left[ (v_k y_k + 1)^{-1} \left(\prod\limits_{\ell'=1}^n (v_{\ell'} y_{\ell'} +1)^{[\varepsilon r_k B_{k\ell'}]_+}\right)\left(\frac{Z_k^\circ\left(y_k^\varepsilon \prod_{\ell'=1}^n (v_{\ell'} y_{\ell'} + 1)^{\varepsilon d_{\ell'} B_{\ell' k}}\right)}{Z_k^\circ(y_k^\varepsilon)}\right)^{1/d_k} - 1\right] & \text{if $\ell=k$;}\\ 
      v_\ell y_k^{-[\varepsilon r_k B_{k\ell}]_+} Z_k^\circ(y_k^\varepsilon)^{B_{k\ell}} & \text{if $\ell\ne k$;}
    \end{cases}\\
    \nonumber
    (\mu_{\cX,\bfr}^{k,\varepsilon})^*(t'_\ell)
    &=\begin{cases} 
      t_k^{-1}\prod\limits_{\ell'=1}^n t_{\ell'}^{[\varepsilon r_k B_{k\ell'}]_+}\left(\frac{Z_k^\circ\left(y_k^\varepsilon \prod_{\ell'=1}^n t_{\ell'}^{\varepsilon d_{\ell'} B_{\ell' k}}\right)}{Z_k^\circ(y_k^\varepsilon)}\right)^{1/d_k} & \text{if $\ell=k$;}\\
      t_\ell & \text{if $\ell\ne k$;}
    \end{cases}\\
    \nonumber
    (\mu_{\cA,\bfr}^{k,\varepsilon})^*(p'_j)&=
    \begin{cases} 
      -p_k x_k^2 \left(\prod_{i=1}^m x_i^{-[-\varepsilon B_{ik} r_k]_+}\right) Z_k^\circ(\hat y_k^\varepsilon)^{-1} & \text{if $j=k$;}\\ 
      p_j + [-\varepsilon B_{jk} r_k]_+ \frac{p_k x_k}{x_j} - \frac{B_{jk}}{d_k x_j}\log\left(\frac{Z_k^\circ\left(\hat y_k^\varepsilon e^{-\varepsilon d_kp_kx_k}\right)}{Z_k^\circ(\hat y_k^\varepsilon)}\right) & \text{if $j\ne k$;}
    \end{cases}\\
    \nonumber
    (\mu_{\cA,\bfr}^{k,\varepsilon})^*(u'_j)
    &=\begin{cases}
      x_k \left[ (u_k x_k +1)^{-1} -1\right] \left(\prod_{i=1}^m x_i^{-[-\varepsilon B_{ik} r_k]_+}\right) Z_k^\circ(\hat y_k^\varepsilon)^{-1} & \text{if $j=k$;}\\ 
      x_j^{-1} \left[ (u_j x_j + 1) (u_k x_k + 1)^{[-\varepsilon B_{jk} r_k]_+} \left(\frac{Z_k^\circ\left(\hat y_k^\varepsilon (u_k x_k +1)^{-\varepsilon d_k}\right)}{Z_k^\circ(\hat y_k^\varepsilon)}\right)^{-B_{jk}/d_k} - 1\right] & \text{if $j\ne k$;}
    \end{cases}\\
    \label{eq:DA mutation}
    (\mu_{\cA,\bfr}^{k,\varepsilon})^*(s'_j)
    &=\begin{cases} 
      s_k^{-1} & \text{if $j=k$;}\\ 
      s_j s_k^{[-\varepsilon B_{jk} r_k]_+} \left(\frac{Z_k^\circ\left(\hat y_k^\varepsilon s_k^{-\varepsilon d_k}\right)}{Z_k^\circ(\hat y_k^\varepsilon)}\right)^{-B_{jk}/d_k} & \text{if $j\ne k$.}
    \end{cases}
  \end{align}
  The coordinates $a_{\ell,j}$ on $T^*L^\times_\bfr$ transform according to
  \begin{align*}
    (\mu_{\cX,\bfr}^{k,\varepsilon})^*(a'_{\ell,j})
    &=\begin{cases}
      a^*_{k,j}+\frac{\varepsilon}{d_k}\int_{y_k^\varepsilon}^{\beta^*(y_k^\varepsilon)} \frac{u^{\varepsilon j^*-1}}{Z_k(u^\varepsilon)}du & \text{if $\ell=k$;}\\
      a_{\ell,j} & \text{if $\ell\ne k$;}
    \end{cases}\\
    (\mu_{\cA,\bfr}^{k,\varepsilon})^*(a'_{\ell,j})
    &=\begin{cases}
      a^*_{k,j}+\frac{\varepsilon}{d_k}\int_{\hat y_k^\varepsilon}^{\beta^*(\hat y_k^\varepsilon)} \frac{u^{\varepsilon j^*-1}}{Z_k(u^\varepsilon)}du & \text{if $\ell=k$;}\\
      a_{\ell,j} & \text{if $\ell\ne k$;}
    \end{cases}
  \end{align*}
  where $\beta$ is the appropriate target map and $j^*=r_k-j$.
\end{theorem}
\begin{proof}
  This is immediate from Corollary~\ref{cor:groupoid hamiltonian flows} and Lemma~\ref{le:tropical groupoid transformations}.
\end{proof}

The various groupoid maps lifting the Poisson maps in \eqref{eq:PoisCD} are summarized in the following commutative diagram. The groupoid morphisms are represented by solid arrows and the groupoid comorphisms are represented by dotted arrows.
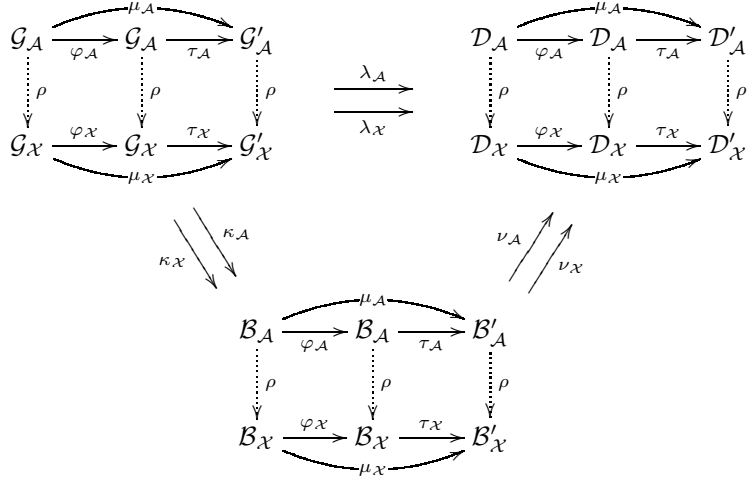
\begin{equation*} 
    \xymatrix{
      \cG_{\cA} \ar@{.>}[d]^{\rho} \ar[r]_{\varphi_{\cA}} \ar@/^1pc/[rr]|-{\mu_{\cA}} & \cG_{\cA} \ar@{.>}[d]^{\rho} \ar[r]_{\tau_{\cA}} & \cG'_{\cA} \ar@{.>}[d]^{\rho} 
	&& \cD_{\cA} \ar@{.>}[d]^{\rho} \ar[r]_{\varphi_{\cA}} \ar@/^1pc/[rr]|-{\mu_{\cA}} & \cD_{\cA} \ar@{.>}[d]^{\rho} \ar[r]_{\tau_{\cA}} & \cD'_{\cA} \ar@{.>}[d]^{\rho} \\
      \cG_{\cX} \ar[r]^{\varphi_{\cX}} \ar@/_1pc/[rr]|-{\mu_{\cX}} & \cG_{\cX} \ar[r]^{\tau_{\cX}} \ar@{}[ddr]^(.325){}="a"^(.775){}="b" \ar@<1ex> "a";"b"^-{\kappa_\cA} \ar@<-1ex> "a";"b"_-{\kappa_\cX} & \cG'_{\cX} \ar@{}[rr]^(.3){}="e"^(.7){}="f" \ar@<5ex> "e";"f"^-{\lambda_\cA} \ar@<3ex> "e";"f"_-{\lambda_\cX}
	&& \cD_{\cX} \ar[r]^{\varphi_{\cX}} \ar@/_1pc/[rr]|-{\mu_{\cX}} & \cD_{\cX} \ar[r]^{\tau_{\cX}} & \cD'_{\cX} \\
	&&&&&&\\
	&&\cB_{\cA} \ar@{.>}[d]^{\rho} \ar[r]_{\varphi_{\cA}} \ar@/^1pc/[rr]|-{\mu_{\cA}} & \cB_{\cA} \ar@{.>}[d]^{\rho} \ar[r]_{\tau_{\cA}} & \cB'_{\cA} \ar@{.>}[d]^{\rho} \ar@{}[uur]^(.2){}="c"^(.65){}="d" \ar@<1ex> "c";"d"^-{\nu_\cA} \ar@<-1ex> "c";"d"_-{\nu_\cX}  \\
      &&\cB_{\cX} \ar[r]^{\varphi_{\cX}} \ar@/_1pc/[rr]|-{\mu_{\cX}} & \cB_{\cX} \ar[r]^{\tau_{\cX}} & \cB'_{\cX}
    }
\end{equation*}
\begin{remark}
  \label{rem:extended groupoid charts}
  While the formulas for the groupoid mutations were only defined over the positive real orthants $L^\times_{\cX,\bfr}$ and $L^\times_{\cA,\bfr}$, they can often also be used to glue the various groupoids over larger (sometimes dense) subsets of $L_{\cX,\bfr}$ and $L_{\cA,\bfr}$ as in Remark~\ref{rem:extended cluster charts}.
  However, certain additional subtleties must be overcome in the groupoid case:
  \begin{enumerate}
    \item In the real case, the mutation rules for $q'_\ell$ and $p'_j$ can only be extended to the source fibers over the coordinate hyperplanes in $L_{\cX,\bfr}$ and $L_{\cA,\bfr}$ since only the logarithms of positive real numbers are again real.
      In the complex case, these mutation rules can be extended to open dense subsets of $\cG_{\cX,\bfr}$ and $\cG_{\cA,\bfr}$, namely to the loci where $\frac{Z_k^\circ\left(y_k^\varepsilon e^{\varepsilon\sum_{\ell'=1}^n d_{\ell'} B_{\ell' k}q_{\ell'} y_{\ell'}}\right)}{Z_k^\circ(y_k^\varepsilon)}$ and $\frac{Z_k^\circ\left(\hat y_k^\varepsilon e^{-\varepsilon d_kp_kx_k}\right)}{Z_k^\circ(\hat y_k^\varepsilon)}$ are defined and nonzero.
      Indeed, although $\log(-)$ is multivalued as a function on $\CC_\times$, the functions $(\mu_{\cX,\bfr}^{k,\varepsilon})^*(q'_\ell)$ and $(\mu_{\cA,\bfr}^{k,\varepsilon})^*(p'_j)$ are well-defined on these dense subsets of $\cG_{\cX,\bfr}$ and $\cG_{\cA,\bfr}$ respectively.
      To see this, note that $\mu_{\cX,\bfr}^{k,\varepsilon}$ and $\mu_{\cA,\bfr}^{k,\varepsilon}$ are groupoid morphisms and hence must carry the identity image to the identity image.
      This in particular fixes the branch of $\log(-)$ to choose when computing the values of $(\mu_{\cX,\bfr}^{k,\varepsilon})^*(q'_\ell)$ and $(\mu_{\cA,\bfr}^{k,\varepsilon})^*(p'_j)$ at the origin of each source fiber.
      Thus, by simple-connectedness, we are able to fix the choice of branch cut to use when computing their values at arbitrary points of the source fibers. 
    \item Again in the real case, the existence of an even skew-symmetrizer $d_k$ similarly restricts the extendability of the mutation rules for $v'_\ell$, $t'_\ell$ and $u'_j$, $s'_j$ to the source fibers over the coordinate hyperplanes in $L_{\cX,\bfr}$ and $L_{\cA,\bfr}$.
      In the complex case or when all skew-symmetrizers are odd, we must avoid the loci where $\frac{Z_k^\circ\left(y_k^\varepsilon \prod_{\ell'=1}^n (v_{\ell'} y_{\ell'} + 1)^{\varepsilon d_{\ell'} B_{\ell' k}}\right)}{Z_k^\circ(y_k^\varepsilon)}$, $\frac{Z_k^\circ\left(y_k^\varepsilon \prod_{\ell'=1}^n t_{\ell'}^{\varepsilon d_{\ell'} B_{\ell' k}}\right)}{Z_k^\circ(y_k^\varepsilon)}$, $\frac{Z_k^\circ\left(\hat y_k^\varepsilon (u_k x_k +1)^{-\varepsilon d_k}\right)}{Z_k^\circ(\hat y_k^\varepsilon)}$, and $\frac{Z_k^\circ\left(\hat y_k^\varepsilon s_k^{-\varepsilon d_k}\right)}{Z_k^\circ(\hat y_k^\varepsilon)}$ are undefined or zero.
      Away from these loci in the complex case, the identity image can again be used to define a splitting of the $d_k$-fold covering map from $\CC_\times$ to $\CC_\times$.
    \item Finally, the mutations $\mu_{\cA,\bfr}^{k,\varepsilon}$ appear to be undefined on $p'_j$ and $u'_j$ for $j\ne k$ when $x_j=0$.
      However, this is a removable singularity and these maps are actually continuous when appropriate values are assigned.
      To see this, we first observe that $[-\varepsilon B_{jk} r_k]_+ p_k x_k - \frac{B_{jk}}{d_k}\log\left(\frac{Z_k^\circ(\hat y_k^\varepsilon e^{-\varepsilon d_kp_kx_k})}{Z_k^\circ(\hat y_k^\varepsilon)}\right)=0$ for $j\ne k$ and $x_j=0$ as follows:
      \begin{itemize}
        \item If $\varepsilon B_{jk}>0$, the first term above is clearly zero and $x_j$ appears in $\hat y_k^\varepsilon$ with a positive exponent.
          Using that $x_j=0$, it follows that the $\log$ term reduces to zero as well since $Z_k^\circ$ has nonzero constant term.
        \item When $B_{jk}=0$ there is nothing to show.
        \item If $\varepsilon B_{jk}<0$, the first term does not vanish but $x_j$ appears in $\hat y_k^\varepsilon$ with a negative exponent.
          Using that $x_j=0$, the term inside the logarithm reduces to $e^{-\varepsilon d_k r_k p_k x_k}$ and thus we again get zero.
      \end{itemize}
      Therefore, we may take the limit as $x_j\to0$ in these formulas and set
      \begin{align*}
        (\mu_{\cA,\bfr}^{k,\varepsilon}|_{x_j=0})^*(p'_j)
        &=p_j+\lim_{x_j \to 0}\frac{[-\varepsilon B_{jk} r_k]_+ p_k x_k - \frac{B_{jk}}{d_k}\log\left(\frac{Z_k^\circ(\hat y_k^\varepsilon e^{-\varepsilon d_kp_kx_k})}{Z_k^\circ(\hat y_k^\varepsilon)}\right)}{x_j}\\
        &=p_j-\frac{\varepsilon B_{jk}^2}{d_k} \lim_{x_j\to 0} \frac{ \hat y_k^\varepsilon e^{-\varepsilon d_kp_kx_k} (Z_k^\circ)'(\hat y_k^\varepsilon e^{-\varepsilon d_kp_kx_k}) Z_k^\circ(\hat y_k^\varepsilon) - \hat y_k^\varepsilon Z_k^\circ(\hat y_k^\varepsilon e^{-\varepsilon d_kp_kx_k}) (Z_k^\circ)'(\hat y_k^\varepsilon) }{ x_j Z_k^\circ(\hat y_k^\varepsilon e^{-\varepsilon d_kp_kx_k}) Z_k^\circ(\hat y_k^\varepsilon) }\\
        &=\begin{cases} 
          p_j & \text{if $|B_{jk}| \ne 1$;}\\ 
          p_j - \frac{1}{d_k} \prod_{i\ne j} x_i^{B_{ik}} z_{k,1} (e^{-d_kp_kx_k} - 1) & \text{if $B_{jk} = 1$;}\\ 
          p_j + \frac{1}{d_k} \prod_{i\ne j} x_i^{-B_{ik}} z_{k,r_k-1} (e^{d_kp_kx_k} - 1) & \text{if $B_{jk} = -1$;}
        \end{cases}
      \end{align*}
      where the last equality uses that $a Z_k'(a) Z_k(b) - b Z_k(a) Z_k'(b)$ has lowest degree term $z_{k,1}(a-b)$ and highest degree term $z_{k,r_k-1} a^n b^n (b^{-1}-a^{-1})$ while $a (Z_k^*)'(a) Z_k^*(b) - b Z_k^*(a) (Z_k^*)'(b)$ has lowest degree term $z_{k,r_k-1}(a-b)$ and highest degree term $z_{k,1} a^n b^n (b^{-1}-a^{-1})$.

      Observe that this restricted mutation map is again not defined for $x_i=0$ when $B_{jk}=1$ and $B_{ik} < 0$ nor when $B_{jk} = -1$ and $B_{ik} > 0$.
      Following Remark~\ref{rem:extended cluster charts}, the gluing maps $\mu_{\cA,\bfr}^{k,\varepsilon}:L_{\cA,\bfr}\to L_{\cA',\bfr}$ are also not defined on these loci since the conditions imply $\left(\prod\limits_{i=1}^m x_i^{[-\varepsilon B_{ik}r_k]_+}\right)Z_k^\circ(\hat y_k^\varepsilon)=0$, in particular the gluing of fibers will not be attempted over these loci. 

      We leave it as an exercise for the reader to verify the following mutation rule for the coordinates $u'_j$ over the locus where $x_j=0$:
      \begin{align*}
        (\mu_{\cA,\bfr}^{k,\varepsilon}|_{x_j=0})^*(u'_j)
        &=\begin{cases} 
          u_j & \text{if $|B_{jk}| \ne 1$;}\\ 
          u_j - \frac{1}{d_k} \prod_{i\ne j} x_i^{B_{ik}} z_{k,1} ((u_k x_k+1)^{-d_k} - 1) & \text{if $B_{jk} = 1$;}\\ 
          u_j + \frac{1}{d_k} \prod_{i\ne j} x_i^{-B_{ik}} z_{k,r_k-1} ((u_k x_k +1)^{d_k} - 1) & \text{if $B_{jk} = -1$.}
        \end{cases}
      \end{align*}
  \end{enumerate}
\end{remark}

As in Section~\ref{sec:cluster mutations}, we record iterated mutations using the $n$-regular rooted tree $\TT_n$ with root vertex $t_0$.
That is, over each cluster chart $L^\times_{\cX,\bfr;t}$ and $L^\times_{\cA,\bfr;t}$ we have, respectively, groupoids $\cG^\times_{\cX,\bfr;t}$, $\cB^\times_{\cX,\bfr;t}$, $\cD^\times_{\cX,\bfr;t}$ and $\cG^\times_{\cA,\bfr;t}$, $\cB^\times_{\cA,\bfr;t}$, $\cD^\times_{\cA,\bfr;t}$ with all of the structure as above associated to the pairs $(\tilde B_t,\Omega_t)$ for each $t\in\TT_n$.
Our goal is to glue the various groupoid charts over the cluster charts to get groupoids over $\cX$ and $\cA$.

Given a pair of vertices $t,t'\in\TT_n$, we will define iterated mutations $\mu_{\cX,\bfr}^{t',t}$ and $\mu_{\cA,\bfr}^{t',t}$ of the groupoid charts as we did for the cluster charts by composing the maps in Theorem~\ref{th:groupoid mutation} along the unique path from $t$ to $t'$ in $\TT_n$ (with mutations always taken with respect to the tropical sign from Definition~\ref{def:tropical signs}).
We have seen that certain sequences of cluster mutations simply permute the coordinates of the corresponding cluster charts (c.f.\ Definition~\ref{def:periodicity}), let $(t_1,\ldots,t_w;\sigma)$ be such a periodicity of cluster mutations.
These lift to simple permutations of the coordinates on the corresponding groupoid charts.
\begin{proposition}
  \label{prop:gpd periodicity}
  Given any periodicity $(t_1,\ldots,t_w;\sigma)$ of cluster mutations, the following equalities hold:
  \begin{enumerate}
    \item $\big(\mu_{\cX,\bfr}^{t_w,t_1}\big)^*(q_{\ell;t_w})=q_{\sigma(\ell);t_1}$ for $1\le\ell\le n$;
    \item $\big(\mu_{\cX,\bfr}^{t_w,t_1}\big)^*(t_{\ell;t_w})=t_{\sigma(\ell);t_1}$ for $1\le\ell\le n$;
    \item $\big(\mu_{\cX,\bfr}^{t_w,t_1}\big)^*(v_{\ell;t_w})=v_{\sigma(\ell);t_1}$ for $1\le\ell\le n$;
    \item $\big(\mu_{-,\bfr}^{t_w,t_1}\big)^*(a_{\ell,j;t_w})=a_{\sigma(\ell),j;t_1}$ for $1\le\ell\le n$ and $1\le j\le r_\ell-1$.
    \item $\big(\mu_{\cA,\bfr}^{t_w,t_1}\big)^*(p_{j;t_w})=p_{\sigma(j);t_1}$ for $1\le j\le m$;
    \item $\big(\mu_{\cA,\bfr}^{t_w,t_1}\big)^*(s_{j;t_w})=s_{\sigma(j);t_1}$ for $1\le j\le m$;
    \item $\big(\mu_{\cA,\bfr}^{t_w,t_1}\big)^*(u_{j;t_w})=u_{\sigma(j);t_1}$ for $1\le j\le m$.
  \end{enumerate}
\end{proposition}
\begin{proof}
  Recall that the identity map of a Poisson manifold lifts to the identity map of its source-simply-connected symplectic groupoid (c.f. Proposition~\ref{prop:gpd identity}), so parts (1) and (5) immediately follow together with part (4) for the mutations on $\cG_{\cX_\bfr}$ and $\cG_{\cA_\bfr}$.
  If we restrict to the case of $\RR_\times$, then the blow-up groupoid $\cB$ and the symplectic double $\cD$ are both source-simply-connected also.
  Thus parts (2), (3), (6), (7) and (4) follow from Proposition~\ref{prop:gpd identity}; the case of $\CC_\times$ follows when we complexify the mutations. 
\end{proof}

Thus it is reasonable to give the following definition.
\begin{definition}
  Fix an initial $m\times n$ exchange matrix $\tilde B_{t_0}$ and an initial skew-symmetric $m\times m$ matrix~$\Omega_{t_0}$ which is $D$-compatible with $\tilde B_{t_0}$.
  Define groupoids $\cG_{\cX_\bfr}$, $\cB_{\cX_\bfr}$, $\cD_{\cX_\bfr}$ and $\cG_{\cA_\bfr}$, $\cB_{\cA_\bfr}$, $\cD_{\cA_\bfr}$ over the cluster varieties $\cX_\bfr$ and $\cA_\bfr$ respectively by gluing the groupoid charts $\cG^\times_{\cX,\bfr;t}$, $\cB^\times_{\cX,\bfr;t}$, $\cD^\times_{\cX,\bfr;t}$ to $\cG^\times_{\cX,\bfr;t'}$, $\cB^\times_{\cX,\bfr;t'}$, $\cD^\times_{\cX,\bfr;t'}$ respectively and by gluing the groupoid charts $\cG^\times_{\cA,\bfr;t}$, $\cB^\times_{\cA,\bfr;t}$, $\cD^\times_{\cA,\bfr;t}$ to $\cG^\times_{\cA,\bfr;t'}$, $\cB^\times_{\cA,\bfr;t'}$, $\cD^\times_{\cA,\bfr;t'}$ respectively, for $t,t'\in\TT_n$, along the appropriate groupoid mutations $\mu_{\cX,\bfr}^{t',t}$ and $\mu_{\cA,\bfr}^{t',t}$ (where mutations are always performed according to the tropical signs from Definition~\ref{def:tropical signs}).
\end{definition}
\begin{remark}
  Recall that the symplectic groupoid charts $\cG_{\cX,\bfr;t}$ and $\cG_{\cA,\bfr;t}$ were constructed in Section~\ref{sec:local} using Poisson sprays.
  Ideally these Poisson sprays would be compatible with mutations and thus glue to give a global construction of the source-simply-connected symplectic groupoids over the varieties $\cX_\bfr$ and $\cA_\bfr$.
  Unfortunately, this is not the case and we must make due with the iterative construction presented above.
\end{remark}

To understand the iterated mutations $\mu_{-,\bfr}^{t,t_0}$, we introduce analogues of the ``separation of additions'' formulas from Theorem~\ref{th:separation}.
As a first step in lifting the separation of additions formulas~\eqref{eq:separation of additions 1} and~\eqref{eq:separation of additions 2} to the level of groupoids we introduce analogues of the $\bfc$-vectors and $\bfg$-vectors which describe iterations of the cluster mutations.
To begin let $C^\vee_{t_0}$ be an $m\times m$ identity matrix and let $G^\vee_{t_0}$ be an $n\times n$ identity matrix.
We then assign matrices $C^\vee_t:=(C^\vee_{ij;t})$ and $G^\vee_t:=(G^\vee_{ij;t})$ to the vertices $t\in\TT_n$ by the following recursion.
For vertices $t,t'\in\TT_n$ joined by an edge labeled $k$, the matrices $C^\vee_t$ and $C^\vee_{t'}$ are related by
\begin{equation}
  \label{eq:dual c-matrix mutation}
  C^\vee_{ij;t'}=
  \begin{cases}
    -C^\vee_{ij;t} & \text{if $i=k$;}\\
    C^\vee_{ij;t}+[-\varepsilon_{\bfr,k;t} B_{ik;t} r_k]_+ C^\vee_{kj;t} & \text{if $i\ne k$;}
  \end{cases}
\end{equation}
while the matrices $G^\vee_t$ and $G^\vee_{t'}$ are related by
\begin{equation}
  \label{eq:dual g-matrix mutation}
  G^\vee_{ij;t'}=
  \begin{cases}
    -G^\vee_{kj;t}+\sum\limits_{\ell=1}^n [\varepsilon_{\bfr,k;t} r_k B_{k\ell;t}]_+ G^\vee_{\ell j;t} & \text{if $i=k$;}\\
    G^\vee_{ij;t} & \text{if $i\ne k$.}
  \end{cases}
\end{equation}
\begin{remark}
  An easy induction shows that $C^\vee_{ij;t}$ is independent of $t$ for $1\le i\le m$ and $n+1\le j\le m$.
  In particular, $C^\vee_{ij;t}=\delta_{ij}$ for $1\le i\le m$ and $n+1\le j\le m$.
\end{remark}
\begin{lemma}
  \label{le:dual vectors}
  For $1\le k,\ell\le n$, we have $d_kC^\vee_{k\ell;t}=d_\ell C_{\ell k;t}$ and $d_kG^\vee_{k\ell;t}=d_\ell G_{\ell k;t}$ for all $t\in\TT_n$.
\end{lemma}
\begin{proof}
  This is an easy induction using the recursions \eqref{eq:c-matrix mutation 2} and \eqref{eq:dual c-matrix mutation} or \eqref{eq:g-matrix mutation} and \eqref{eq:dual g-matrix mutation} together with the identities $d_k B_{k\ell;t}=-d_\ell B_{\ell k;t}$ for all $t\in\TT_n$.
\end{proof}

Using these we have the following analogues of the separation of additions formulas describing the cluster coordinates.
\begin{theorem}
  \label{th:groupoid separation of additions}
  For any vertex $t\in\TT_n$, we have the following formulas for the iterated mutations $\mu_{-,\bfr}^{t,t_0}$, where each $\beta$ below denotes the target map of the appropriate groupoid:
  \begin{align}
    \nonumber \big(\mu_{\cX,\bfr}^{t,t_0}\big)^*(q_{k;t})&=\frac{ \sum_{\ell=1}^n G^\vee_{k\ell;t} y_{\ell;t_0} q_{\ell;t_0} +\frac{1}{d_k} \log\left(\frac{\beta^*\left(F_{\bfr,k;t}\big(y_{1;t_0},\ldots,y_{n;t_0}\big)\right)}{F_{\bfr,k;t}(y_{1;t_0},\ldots,y_{n;t_0})}\right)}{\prod_{\ell=1}^n y_{\ell;t_0}^{C_{\ell k;t}} F_{\bfr,\ell;t}(y_{1;t_0},\ldots,y_{n;t_0})^{B_{\ell k;t}}};\\
    \nonumber \big(\mu_{\cX,\bfr}^{t,t_0}\big)^*(t_{k;t})&=\left(\prod_{\ell=1}^n t_{\ell;t_0}^{G^\vee_{k\ell;t}}\right) \left(\frac{\beta^*\left(F_{\bfr,k;t}\big(y_{1;t_0},\ldots,y_{n;t_0}\big)\right)}{F_{\bfr,k;t}(y_{1;t_0},\ldots,y_{n;t_0})}\right)^{1/d_k};\\
    \nonumber \big(\mu_{\cX,\bfr}^{t,t_0}\big)^*(v_{k;t})&=\frac{\left(\prod_{\ell=1}^n (y_{\ell;t_0} v_{\ell;t_0} + 1)^{G^\vee_{k\ell;t}}\right) \left(\frac{\beta^*\left(F_{\bfr,k;t}\big(y_{1;t_0},\ldots,y_{n;t_0}\big)\right)}{F_{\bfr,k;t}(y_{1;t_0},\ldots,y_{n;t_0})}\right)^{1/d_k} - 1}{\prod_{\ell=1}^n y_{\ell;t_0}^{C_{\ell k;t}} F_{\bfr,\ell;t}(y_{1;t_0},\ldots,y_{n;t_0})^{B_{\ell k;t}}};\\
    \nonumber \big(\mu_{\cA,\bfr}^{t,t_0}\big)^*(p_{i;t})&=\frac{\left(\sum_{j=1}^m C^\vee_{ij;t} x_{j;t_0} p_{j;t_0}\right) + \sum_{\ell=1}^n \frac{B_{i\ell;t}}{d_\ell}\log\left(\frac{\beta^*\left(F_{\bfr,\ell;t}(\hat y_{1;t_0},\ldots,\hat y_{n;t_0})\right)}{F_{\bfr,\ell;t}(\hat y_{1;t_0},\ldots,\hat y_{n;t_0})}\right)}{\left(\prod_{j=1}^m x_{j;t_0}^{G_{ji;t}}\right) F_{\bfr,i;t}(\hat y_{1;t_0},\ldots,\hat y_{n;t_0})};\\
    \label{eq:groupoid separation of additions 2}
    \big(\mu_{\cA,\bfr}^{t,t_0}\big)^*(s_{i;t})&=\left(\prod_{j=1}^m s_{j;t_0}^{C^\vee_{ij;t}}\right) \prod_{\ell=1}^n \left(\frac{\beta^*\left(F_{\bfr,\ell;t}(\hat y_{1;t_0},\ldots,\hat y_{n;t_0})\right)}{F_{\bfr,\ell;t}(\hat y_{1;t_0},\ldots,\hat y_{n;t_0})}\right)^{B_{i\ell;t}/d_\ell};\\
    \nonumber \big(\mu_{\cA,\bfr}^{t,t_0}\big)^*(u_{i;t})&=\frac{\left(\prod_{j=1}^m (x_{j;t_0} u_{j;t_0} + 1)^{C^\vee_{ij;t}}\right) \prod_{\ell=1}^n \left(\frac{\beta^*\left(F_{\bfr,\ell;t}(\hat y_{1;t_0},\ldots,\hat y_{n;t_0})\right)}{F_{\bfr,\ell;t}(\hat y_{1;t_0},\ldots,\hat y_{n;t_0})}\right)^{B_{i\ell;t}/d_\ell} - 1}{\left(\prod_{j=1}^m x_{j;t_0}^{G_{ji;t}}\right) F_{\bfr,i;t}(\hat y_{1;t_0},\ldots,\hat y_{n;t_0})}.
  \end{align}
  The coordinates $a_{\ell,j;t}$ on $T^*L_\bfr$ transform according to
  \begin{align}
    \label{eq:cotangent mutations}
    \big(\mu_{\cX,\bfr}^{t,t_0}\big)^*(a_{\ell,j;t})&=a^\circ_{\ell,j;t_0}+\sum_{i:k_i=\ell} \frac{\varepsilon_{\bfr,\ell;t_i}}{d_\ell} \int_{\big(\mu_{\cX,\bfr}^{t_i,t_0}\big)^*(y_{\ell;t_i}^{\varepsilon_{\bfr,\ell;t_i}})}^{\big(\mu_{\cX,\bfr}^{t_i,t_0}\big)^*\beta^*(y_{\ell;t_i}^{\varepsilon_{\bfr,\ell;t_i}})} \frac{u^{\varepsilon_{\bfr,\ell;t_i} j^\circ-1}}{Z_{\ell;t_i}(u^{\varepsilon_{\bfr,\ell;t_i}})}du;\\
    \big(\mu_{\cA,\bfr}^{t,t_0}\big)^*(a_{\ell,j;t})&=a^\circ_{\ell,j;t_0}+\sum_{i:k_i=\ell} \frac{\varepsilon_{\bfr,\ell;t_i}}{d_\ell} \int_{\big(\mu_{\cA,\bfr}^{t_i,t_0}\big)^*(\hat y_{\ell;t_i}^{\varepsilon_{\bfr,\ell;t_i}})}^{\big(\mu_{\cA,\bfr}^{t_i,t_0}\big)^*\beta^*(\hat y_{\ell;t_i}^{\varepsilon_{\bfr,\ell;t_i}})} \frac{u^{\varepsilon_{\bfr,\ell;t_i} j^\circ-1}}{Z_{\ell;t_i}(u^{\varepsilon_{\bfr,\ell;t_i}})}du;
  \end{align}
  where 
  \begin{itemize}
    \item $t_0$ and $t$ are connected through vertices $t_1,\ldots,t_w$ by mutations in directions $k_0,k_1,\ldots,k_w$;
    \item we set $a^\circ_{\ell,j;t_0}=a_{\ell,j;t_0}$ if there are an even number of $i$ so that $k_i=\ell$ while we take $a^\circ_{\ell,j;t_0}=a^*_{\ell,j;t_0}$ if there is an odd number of $i$ so that $k_i=\ell$;
    \item we set $j^\circ=j$ if there are an even number of $i'$ with $i\le i'\le w$ so that $k_{i'}=\ell$ while $j^\circ=j^*$ if there is an odd number of $i'$ with $i\le i'\le w$ so that $k_{i'}=\ell$.
  \end{itemize}
\end{theorem}
\begin{proof}
  We prove only equation~\eqref{eq:groupoid separation of additions 2}, working by induction on the distance from $t_0$ to $t$ inside $\TT_n$ in this case.
  First consider when $t$ is joined to $t_0$ by an edge labeled $k$.
  Then $F_{\bfr,k;t}=Z_{k;t_0}^\circ(u_k)$ and $F_{\bfr,\ell;t}=1$ for $\ell\ne k$ so that \eqref{eq:groupoid separation of additions 2} just becomes the mutation formula \eqref{eq:DA mutation}.

  Now suppose $t'$ and $t$ are joined by an edge labeled $k$ with $t$ lying on the unique shortest path from $t_0$ to $t'$.
  Then $\mu_{\cA,\bfr}^{t',t_0}=\mu_{\cA,\bfr}^{t',t}\circ\mu_{\cA,\bfr}^{t,t_0}$ so that
  \begin{align*}
    \big(\mu_{\cA,\bfr}^{t',t_0}\big)^*(s_{k;t'})
    &=\big(\mu_{\cA,\bfr}^{t,t_0}\big)^*\circ\big(\mu_{\cA,\bfr}^{t',t}\big)^*(s_{k;t'})\\
    &=\big(\mu_{\cA,\bfr}^{t,t_0}\big)^*(s_{k;t}^{-1})\\
    &=\left(\prod_{j=1}^m s_{j;t_0}^{-C^\vee_{kj;t}}\right) \prod_{\ell=1}^n \left(\frac{F_{\bfr,\ell;t}(\hat y_{1;t_0} s_{1;t_0}^{-d_1},\ldots,\hat y_{n;t_0} s_{n;t_0}^{-d_n})}{F_{\bfr,\ell;t}(\hat y_{1;t_0},\ldots,\hat y_{n;t_0})}\right)^{-B_{k\ell;t}/d_\ell}\\
    &=\left(\prod_{j=1}^m s_{j;t_0}^{C^\vee_{kj;t'}}\right) \prod_{\ell=1}^n \left(\frac{F_{\bfr,\ell;t'}(\hat y_{1;t_0} s_{1;t_0}^{-d_1},\ldots,\hat y_{n;t_0} s_{n;t_0}^{-d_n})}{F_{\bfr,\ell;t'}(\hat y_{1;t_0},\ldots,\hat y_{n;t_0})}\right)^{B_{k\ell;t'}/d_\ell},
  \end{align*}
  where the last equality used that $\ell\ne k$ in the product since $B_{kk;t}=0$.
  On the other hand, writing $\varepsilon=\varepsilon_{\bfr,k;t}$ below, for $i\ne k$ we have
  \begin{align*}
    \big(\mu_{\cA,\bfr}^{t',t_0}\big)^*(s_{i;t'})
    &=\big(\mu_{\cA,\bfr}^{t,t_0}\big)^*\circ\big(\mu_{\cA,\bfr}^{t',t}\big)^*(s_{i;t'})\\
    &=\big(\mu_{\cA,\bfr}^{t,t_0}\big)^*\left(s_{i;t} s_{k;t}^{[-\varepsilon B_{ik;t} r_k]_+} \left(\frac{Z_{k;t}^\circ\left(\hat y_{k;t}^\varepsilon s_{k;t}^{-\varepsilon d_k}\right)}{Z_{k;t}^\circ(\hat y_{k;t}^\varepsilon)}\right)^{-B_{ik;t}/d_k}\right).
  \end{align*}
  We handle the ``$\bfc$-vector part'' of the above expression first.
  Applying \eqref{eq:groupoid separation of additions 2} to the first two terms above and recording only the $\bfc$-vector part gives
  \begin{align*}
    \left(\prod_{j=1}^m s_{j;t_0}^{C^\vee_{ij;t}}\right)\left(\prod_{j=1}^m s_{j;t_0}^{C^\vee_{kj;t}}\right)^{[-\varepsilon B_{ik;t} r_k]_+}
    =\prod_{j=1}^m s_{j;t_0}^{C^\vee_{ij;t} + [-\varepsilon B_{ik;t} r_k]_+ C^\vee_{kj;t}}
    =\prod_{j=1}^m s_{j;t_0}^{C^\vee_{ij;t'}},
  \end{align*}
  which is the desired $\bfc$-vector part of $\big(\mu_{\cA,\bfr}^{t',t_0}\big)^*(s_{i;t'})$.
  Meanwhile the $F$-polynomial part of this is given by
  \begin{align*}
    &\prod_{\ell=1}^n \left(\frac{F_{\bfr,\ell;t}(\hat y_{1;t_0} s_{1;t_0}^{-d_1},\ldots,\hat y_{n;t_0} s_{n;t_0}^{-d_n})}{F_{\bfr,\ell;t}(\hat y_{1;t_0},\ldots,\hat y_{n;t_0})}\right)^{B_{i\ell;t}/d_\ell}
    \prod_{\ell=1}^n \left(\frac{F_{\bfr,\ell;t}(\hat y_{1;t_0} s_{1;t_0}^{-d_1},\ldots,\hat y_{n;t_0} s_{n;t_0}^{-d_n})}{F_{\bfr,\ell;t}(\hat y_{1;t_0},\ldots,\hat y_{n;t_0})}\right)^{[-\varepsilon B_{ik;t} r_k]_+ B_{k\ell;t}/d_\ell}\\
    & = \left(\frac{F_{\bfr,k;t}(\hat y_{1;t_0} s_{1;t_0}^{-d_1},\ldots,\hat y_{n;t_0} s_{n;t_0}^{-d_n})}{F_{\bfr,k;t}(\hat y_{1;t_0},\ldots,\hat y_{n;t_0})}\right)^{B_{ik;t}/d_k}
    \prod_{\ell\ne k} \left(\frac{F_{\bfr,\ell;t}(\hat y_{1;t_0} s_{1;t_0}^{-d_1},\ldots,\hat y_{n;t_0} s_{n;t_0}^{-d_n})}{F_{\bfr,\ell;t}(\hat y_{1;t_0},\ldots,\hat y_{n;t_0})}\right)^{B_{i\ell;t'}/d_\ell-B_{ik;t} [\varepsilon r_k B_{k\ell;t}]_+/d_\ell},
  \end{align*}
  where the final exponent above is given by \eqref{eq:matrix mutation} and by skew-symmetrizability this may be rewritten as $B_{i\ell;t'}/d_\ell-B_{ik;t} [-\varepsilon B_{\ell k;t} r_k]_+/d_k$.
  On the other hand, considering only $\big(\mu_{\cA,\bfr}^{t,t_0}\big)^*\left(\hat y_{k;t}^\varepsilon s_{k;t}^{-\varepsilon d_k}\right)$ gives
  \begin{align*}
    \nonumber &\left(\prod_{\ell=1}^n \hat y_{\ell;t_0}^{\varepsilon C_{\ell k;t}} F_{\bfr,\ell;t}(\hat y_{1;t_0},\ldots,\hat y_{n;t_0})^{\varepsilon B_{\ell k;t}}\right) \left(\prod_{j=1}^m s_{j;t_0}^{-\varepsilon d_k C^\vee_{kj;t}}\right) \prod_{\ell=1}^n \left(\frac{F_{\bfr,\ell;t}(\hat y_{1;t_0} s_{1;t_0}^{-d_1},\ldots,\hat y_{n;t_0} s_{n;t_0}^{-d_n})}{F_{\bfr,\ell;t}(\hat y_{1;t_0},\ldots,\hat y_{n;t_0})}\right)^{-\varepsilon d_k B_{k\ell;t}/d_\ell}\\
    \nonumber &=\left(\prod_{\ell=1}^n \hat y_{\ell;t_0}^{\varepsilon C_{\ell k;t}} F_{\bfr,\ell;t}(\hat y_{1;t_0},\ldots,\hat y_{n;t_0})^{\varepsilon B_{\ell k;t}}\right) \left(\prod_{\ell=1}^n s_{\ell;t_0}^{-\varepsilon d_k C^\vee_{k\ell;t}}\right) \prod_{\ell=1}^n \left(\frac{F_{\bfr,\ell;t}(\hat y_{1;t_0} s_{1;t_0}^{-d_1},\ldots,\hat y_{n;t_0} s_{n;t_0}^{-d_n})}{F_{\bfr,\ell;t}(\hat y_{1;t_0},\ldots,\hat y_{n;t_0})}\right)^{\varepsilon B_{\ell k;t}}\\
    &=\left(\prod_{\ell=1}^n (\hat y_{\ell;t_0} s_{\ell;t_0}^{-d_\ell})^{\varepsilon C_{\ell k;t}}\right) \prod_{\ell=1}^n \left(F_{\bfr,\ell;t}(\hat y_{1;t_0} s_{1;t_0}^{-d_1},\ldots,\hat y_{n;t_0} s_{n;t_0}^{-d_n})\right)^{\varepsilon B_{\ell k;t}},
  \end{align*}
  where the last equality used Lemma~\ref{le:dual vectors}.
  Thus combining the $F$-polynomial part above with the expression $\big(\mu_{\cA,\bfr}^{t,t_0}\big)^*\left(\frac{Z_{k;t}^\circ\left(\hat y_{k;t}^\varepsilon s_{k;t}^{-\varepsilon d_k}\right)}{Z_{k;t}^\circ(\hat y_{k;t}^\varepsilon)}\right)^{-B_{ik;t}/d_k}$ and applying the $F$-polynomial recursion \eqref{eq:F-polynomial mutation} gives the total $F$-polynomial part
  \[\prod_{\ell\ne k} \left(\frac{F_{\bfr,\ell;t}(\hat y_{1;t_0} s_{1;t_0}^{-d_1},\ldots,\hat y_{n;t_0} s_{n;t_0}^{-d_n})}{F_{\bfr,\ell;t}(\hat y_{1;t_0},\ldots,\hat y_{n;t_0})}\right)^{B_{i\ell;t'}/d_\ell} \left(\frac{F_{\bfr,k;t'}(\hat y_{1;t_0} s_{1;t_0}^{-d_1},\ldots,\hat y_{n;t_0} s_{n;t_0}^{-d_n})}{F_{\bfr,k;t'}(\hat y_{1;t_0},\ldots,\hat y_{n;t_0})}\right)^{-B_{ik;t}/d_k}.\]
  But $B_{ik;t'}=-B_{ik;t}$ so that this final expression gives the desired $F$-polynomial part of $\big(\mu_{\cA,\bfr}^{t',t_0}\big)^*(s_{i;t'})$.
\end{proof}

\begin{remark}
  Combining the periodicity result from Proposition~\ref{prop:gpd periodicity} and the iterated mutation formula~\eqref{eq:cotangent mutations} for $T^* L_\bfr$, we obtain the following collection of identities, one for each $1\le\ell\le n$, each choice of $1\le j\le r_\ell$, and any periodicity $(t_1,\ldots,t_w;\sigma)$ of cluster mutations in which $t_i$ is related to $t_{i+1}$ by mutation in direction $k_i$:
  \begin{align}
    \sum_{i:k_i=\ell} \frac{\varepsilon_i}{d_\ell} \int_{\big(\mu_{\cX,\bfr}^{t_i,t_1}\big)^*(y_{\ell;t_i}^{\varepsilon_i})}^{\big(\mu_{\cX,\bfr}^{t_i,t_1}\big)^*\beta^*(y_{\ell;t_i}^{\varepsilon_i})} \frac{u^{\varepsilon_i j^\circ-1}}{Z_{\ell;t_i}(u^{\varepsilon_i})}du=0;\\
    \sum_{i:k_i=\ell} \frac{\varepsilon_i}{d_\ell} \int_{\big(\mu_{\cA,\bfr}^{t_i,t_1}\big)^*(\hat y_{\ell;t_i}^{\varepsilon_i})}^{\big(\mu_{\cA,\bfr}^{t_i,t_1}\big)^*\beta^*(\hat y_{\ell;t_i}^{\varepsilon_i})} \frac{u^{\varepsilon_i j^\circ-1}}{Z_{\ell;t_i}(u^{\varepsilon_i})}du=0;
  \end{align}
  where $j^\circ=j$ if there are an even number of $i'$ with $1\le i'\le i-1$ so that $k_{i'}=\ell$ while $j^\circ=j^*$ if there is an odd number of $i'$ with $1\le i'\le i-1$ so that $k_{i'}=\ell$ and $\varepsilon_i$ denotes the tropical sign in the mutation from $t_i$ to $t_{i+1}$ along the unique mutation sequence from $t_1$ to $t_w$.

  Similar considerations using charts of the form $L_\cX\times T^* L_\bfr$ yield the analogous identities:
  \begin{equation}
    \label{eq:dilogarithm derivatives}
    \sum_{i:k_i=\ell} \frac{\varepsilon_i}{d_\ell} \int_0^{\big(\mu_{\cX,\bfr}^{t_i,t_1}\big)^*(y_{\ell;t_i}^{\varepsilon_i})} \frac{u^{\varepsilon_i j^\circ-1}}{Z_{\ell;t_i}(u^{\varepsilon_i})}du=0.
  \end{equation}
  These identities are closely related to the connection between periodicities of cluster mutations and Roger's dilogarithm identities.
  More precisely, taking the derivative with respect to $z_{\ell,j}$ on the left-hand side of the Roger's dilogarithm identity \cite[Equation (4.28)]{Nak16} results in a number of integral terms coming from taking these derivatives in the Euler dilogarithm terms (c.f.\ \cite[Equation (3.5)]{Nak16}) present in \cite[Equation (4.28)]{Nak16}.
  Equation~\eqref{eq:dilogarithm derivatives} states that the sum of these terms is equal to zero, the remaining collection of logarithm terms also summing to zero by \cite[Equation (4.28)]{Nak16}.
  The vanishing of these logarithm terms seems closely related to the constancy condition \cite[Theorem 4.4]{Nak16}, it would be interesting to make this connection precise.
\end{remark}

\bibliographystyle{hyperamsplain}
\bibliography{cluster_symplectic}

\end{document}